\documentclass[12pt]{book}%
\usepackage{amsmath}
\usepackage{amsfonts}
\usepackage{amssymb}
\usepackage{graphicx}%
\providecommand{\U}[1]{\protect \rule{.1in}{.1in}}
\newtheorem{theorem}{Theorem}[section]

\newtheorem{corollary}[theorem]{Corollary}

\newtheorem{definition}[theorem]{Definition}
\newtheorem{example}[theorem]{Example}

\newtheorem{lemma}[theorem]{Lemma}
\newtheorem{notation}[theorem]{Notation}
\newtheorem{problem}[theorem]{Problem}
\newtheorem{proposition}[theorem]{Proposition}
\newtheorem{remark}[theorem]{Remark}

\newenvironment{proof}[1][Proof]{\noindent \textbf{#1.} }{\  \rule{0.5em}{0.5em}}
\begin{document}

\title{$G$--Brownian Motion and Dynamic Risk Measure under Volatility Uncertainty}
\author{Shige PENG\thanks{The author thanks the partial support from
The National Basic Research Program of China (973 Program) grant No.
2007CB814900 (Financial Risk) .}\\Institute of Mathematics\\Shandong
University\\250100, Jinan, China\\peng@sdu.edu.cn}
\date{Version: November 18, 2007} \maketitle
\tableofcontents

\chapter{Introduction}

It is well-known in classical probability theory that a expectation
$\mathbb{E}$ satisfies the following relation for random variables
$X$ and $Y$:
$$\mathbb{E}[aX+Y+c]=a\mathbb{E}[X]+\mathbb{E}[Y]+c,\ \ \ \ \forall a,c \in \mathbb{R}.$$
A sublinear expectation $\mathbb{\hat{E}}$ satisfies the following
weaker condition:
$$\mathbb{\hat{E}}[X+Y]\leq \mathbb{\hat{E}}[X]+\mathbb{\hat{E}}[Y],\ \ \ \mathbb{\hat{E}}[|a|X+c]=|a|\mathbb{\hat{E}}[X]+c.$$
This $\mathbb{\hat{E}}$ keeps the monotonicity property: If $X\geq
Y$ then $\mathbb{\hat{E}}[X]\geq \mathbb{\hat{E}}[Y]$.

The notion of sublinear expectations is proved to be a basic tool in
volatility uncertainty which is crucial in superhedging,
superpricing (See \cite{Avellaneda} and \cite{Lyons})  and measures
of risk in finance which caused a great attention in finance since
the pioneer work of Artzner, Delbaen, Eber and Heath \cite{ADEH1}
and \cite{ADEH2}. This is also the start point of a new theory
stochastic calculus which gives us a new insight to characterize and
calculate varies kinds of financial risk.

In this note we will introduce a crucial notion of $G$-normal
distributions corresponding to the well-known normal distributions
in classical probability theory. This $G$-normal distribution will
bring us to a new framework of stochastic calculus of It\^{o}'s type
through the corresponding $G$-Brownian motion. We will also present
analytical calculations and some new statistical methods with
application to risk analysis in finance under volatility
uncertainty.

Our basic point of view is: sublinear expectation theory is very
like its special situation of linear expectation in the classical
probability theory. Under a sublinear expectation (or even more
general nonlinear expectation) space we still can introduce the
notion of distributions, of random variables, as well as  the
notions of joint distributions, marginal distributions, etc. We
still denoted $X\sim Y$ if $X$ and $Y$ are identically distributed.
We still have the notion of independence. A particularly interesting
phenomenon in sublinear situations is that a random variable $Y$ is
independent to $X$ does not automatically implies that $X$ is
independent to $Y$.

We will prove two important theorems in the framework of sublinear
expectation theory: The law of large number and the central limit
theorem. A very interesting result of our new central limit theorem
under a sublinear expectation is that a sequence of zero-mean
independent and identically distributed random will converge in law
to a `$G$-normal distribution'
$\mathcal{N}(0,[{\underline\sigma}^2,{\overline\sigma}^2])$. Briefly
speaking a random variable $X$ in a sublinear expectation space is
said to be
$\mathcal{N}(0,[{\underline\sigma}^2,{\overline\sigma}^2])$-distributed
if for any $Y$ independent and identically distributed w.r.t. $X$
and for any real function $a$ and $b$ we have $aX+bY\sim
\sqrt{a^2+b^2}X$. Here ${\overline\sigma}^2=\mathbb{\hat{E}}[X^2]$
and ${\underline\sigma}^2=-\mathbb{\hat{E}}[-X^2]$. In a special
case ${\underline\sigma}^2={\overline\sigma}^2$, this $G$-normal
distribution becomes a classical normal distribution
$\mathcal{N}(0,{\overline\sigma}^2)$.

We will define a sublinear expectation on the space of continuous
paths from $\mathbb{R}_{+}$ to $\mathbb{R}^{d}$ which will be an
analogue of Wiener's law, by which a $G$-Brownian motion is
formulated. Briefly speaking a $G$-Brownian motion $(B_t)_{t\geq 0}$
is a continuous process with independent and stationary increments
under a given sublinear expectation $\mathbb{\hat{E}}[\cdot]$.

$G$--Brownian motion has a very rich and interesting new structure
which non trivially generalizes the classical one. We can establish
the related
stochastic calculus, especially $G$--It\^{o}'s integrals (see \cite[1942]%
{Ito}) and the related quadratic variation process $\left \langle
B\right \rangle $. A very interesting new phenomenon of our
$G$-Brownian motion is that its quadratic process $\left \langle
B\right \rangle $ also has independent and stationary increments. \
The corresponding $G$--It\^{o}'s formula is obtained. We then
introduce the notion of $G$--martingales and the related Jensen
inequality for a new type of
\textquotedblleft$G$--convex\textquotedblright \ functions. We have
also established the existence and uniqueness of solution to
stochastic differential equation under our stochastic calculus by
the same Picard iterations as in the classical situation. Books on
stochastic calculus e.g., \cite{Chu-Will}, \cite{HWY},\  \cite{IW},
\cite{Ito-McKean}, \cite{KSh}, \cite{Oksendal}, \cite{Protter},
\cite{Revuz-Yor}, \cite{Yong-Zhou} are recommended for understanding
the present results and some further possible developments of this
new stochastic calculus.

A sublinear expectation can be regarded as a coherent risk measure.
This, together with the related conditional expectations
$\mathbb{E}[\cdot|\mathcal{H}_{t}]_{t\geq0}$ makes a dynamic risk
measure: $G$--risk measure.

The other motivation of our $G$--expectation is the notion of
(nonlinear) $g$--expectations introduced in \cite{Peng1997},
\cite{Peng1997b}. Here $g$ is the generating function of a backward
stochastic differential equation (BSDE) on a given probability space
$(\Omega,\mathcal{F},\mathbf{P})$. The natural definition of the
conditional $g$--expectations with respect to the past induces rich
properties of nonlinear $g$--martingale theory (see among others,
\cite{BCHMP1}, \cite{Chen98}, \cite{CE}, \cite{CKJ}, \cite{CHMP},
\cite{CHMP3}, \cite{CP}, \cite{CP1}, \cite{Jiang}, \cite{JC},
\cite{Peng1999}, \cite{Peng2005a}, \cite{Peng2005b}, \cite{PX2003}).
Recently $g$--expectations are also studied as dynamic risk
measures: $g$--risk measure (cf. \cite{Roazza2003}, \cite{El-Bar},
\cite{DPR}). Fully nonlinear super-hedging is also a possible
application (cf. \cite{Avellaneda}, \cite{Lyons} and \cite{Touzi}
where new BSDE\ approach was introduced).

The notion of $g$--expectation is defined on a given probability
space. In \cite{Peng2005} (see also \cite{Peng2004}). As compared
with the framework of $g$--expectations, the theory of
$G$--expectation is more intrinsic, a meaning similar to
\textquotedblleft intrinsic geometry\textquotedblright \ in the
sense that it cannot be based on a given (linear) probability space.
\ Since the classical Brownian expectation as well as many other
linear and nonlinear expectations are dominated by our
$G$--Expectation, our theory also provides a flexible theoretical
framework.

The whole results of this paper are based on the very basic
knowledge of Banach space and the parabolic partial differential
equation (\ref{eq-G-heat}). When this $G$-heat equation
(\ref{eq-G-heat}) is linear, our $G$-Brownian motion becomes the
classical Brownian motion. This paper still provides an analytic
shortcut to reach the sophistic It\^{o}'s calculus.

A very basic knowledge of Banach space and probability theory are
necessary. Stochastic analysis, i.e., the theory of stochastic
processes will be very helpful but not necessary.\ We also use some
basic knowledge on the smooth solutions of parabolic partial
differential equations. Good background in statistics will be very
helpful. This note was written for several series of lectures: In
the 2nd Workshop ``Stochastic Equations and Related Topic'' Jena,
July 23--29, 2006; In graduate Courses of Yantai Summer School in
Finance, Yantai University,  July 06--21, 2007; as well as in
Graduate Courses of Wuhan Summer School, July 24--26, 2007. Also in
mini-course of Institute of Applied Mathematics, AMSS, April 16-18
2007 and a mini-course in Fudan University, May 2007; In graduate
courses of CSFI, Osaka University, May 15--June 13, 2007. The
hospitalities and encouragements of the above institutions and the
enthusiasm of the audiences are the main engine to realize this
lecture notes.

I thank for many comments and suggestions given during those
courses, especially to Li Juan and Hu Mingshang. References are
given at the end. Historical remarks are still under preparation.
This lecture note are mainly based on my recent research papers:

\vskip 1cm

\noindent Peng, S. (2006) $G$--Expectation, $G$--Brownian Motion and
Related Stochastic Calculus of It\^{o}'s type, (pdf-file available
in arXiv:math.PR/0601035v2, to appear in \emph{Proceedings of the
2005 Abel Symposium}, Springer.
\medskip

\noindent Peng, S. (2006) Multi-dimensional $G$--Brownian motion and
related stochastic calculus under $G$--expectation, Preprint,
(pdf-file available in arXiv:math.PR/0601699 v2).
\medskip

\noindent Peng, S. (2007) Law of large numbers and central limit
theorem under nonlinear expectations, in arXiv:math.PR/0702358v1 13
Feb 2007
\medskip

which were stimulated by:
\medskip

\noindent Peng, S. (2004) Filtration Consistent Nonlinear
Expectations and Evaluations of Contingent Claims, \emph{Acta
Mathematicae Applicatae Sinica,} English Series \textbf{20}(2),
1--24.
\medskip

\noindent Peng, S. (2005) Nonlinear expectations and nonlinear
Markov chains, \emph{Chin. Ann. Math.} \textbf{26B}(2) ,159--184.

\medskip

\chapter{Risk Measures  and Sublinear Expectations  }

\section{How to measure risk of financial positions}

Let $\Omega$ be a set the set of scenarios. We are given a linear
subspace $\mathcal{X}$ of real valued and bounded functions on
$\Omega$ $\mathcal{X}$ is the collection of are all possible risk
positions in a financial market. We assume that all constants are in
$\mathcal{X}$ and that $X\in\mathcal{X}$ implies
$|X|\in\mathcal{X}$. For each $X\in \mathcal{X}$ we denote
\[
X^{\ast}=\sup_{\omega \in \Omega}X(\omega)\text{ \  \ and \  \
}\left \Vert X\right \Vert =\left \vert X\right \vert ^{\ast}.
\]
$\left \Vert \cdot \right \Vert $ is a Banach norm on $\mathcal{X}$.
In this section $\mathcal{X}$ is assumed to be a Banach space, i.e.,
if a sequence $\{X_{i}\}_{i=1}^{\infty}$ of $\mathcal{X}$ converges
uniformly to some function $X$ on $\Omega$, then $X\in \mathcal{X}$.

\begin{remark}
If $S\in \mathcal{X}$, then for each constant $c$, $S\vee c$,
$S\wedge c$ are all in $\mathcal{X}$. One typical example in finance
is that $S$ is the tomorrow's price of a stock. In this case any
European call or put options of
forms%
\[
(S-k)^{+},\  \ (k-S)^{+}%
\]
are in $\mathcal{X}$.
\end{remark}

\begin{problem}
Prove that if $S\in \mathcal{X}$ then $\varphi(S)\in \mathcal{X}$
for each continuous function $\varphi$ on $\mathbb{R}$.
\end{problem}

\begin{example} Let $(\Omega,\mathcal{F})$ be a measurable space and
let $\mathbb{L}^{\infty}(\Omega,\mathcal{F})$ be the space of all
bounded
$\mathcal{F}$-measurable random variables. $\mathbb{L}^{\infty}%
(\Omega,\mathcal{F})$ is considered as a space of risk positions in
a financial market. $\mathbb{L}^{\infty}(\Omega,\mathcal{F})$ is a
Banach space under the norm $\left \Vert X\right \Vert
_{\infty}=\sup_{\omega \in \Omega }\left \vert X(\omega)\right \vert
$.
\end{example}

\begin{example}
If $\Omega$ is a metric space, then we can consider $\mathcal{X}=C_{b}%
(\Omega \mathcal{)}$, the set of all bounded and continuous
functions on $\Omega$.
\end{example}

\subsection{{Coherent measures of risk}}

 A risk supervisor is responsible for taking a rule to tell traders,
stocks companies, banks or other institutions under his supervision,
which kind of risk positions is unacceptable and thus a minimum
amount of risk capitals should be deposited to make the positions to
be acceptable. The collection of acceptable positions is defined by:
\[
\mathcal{%
A%
}=\{X\in \mathcal{X},\;X\text{ is acceptable}\}.
\]
This set has the following economically meaningful properties:

\begin{definition}
(\textbf{Coherent acceptable set}{)} \newline \textbf{\textsl{(i)} }%
{\textbf{Monotonicity:}
\[
X\in \mathcal{A},\;Y\geq X\; \; \Rightarrow \; \;Y\in \mathcal{A}%
\]
\textbf{\textsl{(ii)} }$0\in \mathcal{A}$ but $-1\not \in \mathcal{A}$. }\newline%
\textbf{\textsl{(iii)} }{\textbf{Positively homogeneity}
\[
X\in \mathcal{%
A%
}\; \; \Rightarrow \; \; \lambda X\in \mathcal{%
A%
},\; \; \forall \lambda \geq0.
\]
\textbf{\textsl{(iv)}} \textbf{Convexity:}
\[
X,Y\in \mathcal{A}\Rightarrow \alpha X+(1-\alpha)Y\in \mathcal{A},\;
\alpha \in \lbrack0,1]
\]
}
\end{definition}

\begin{remark}
{(iii) and (iv) \textbf{imply }\newline\textbf{\textsl{(v)}} \textbf{Sublinearity:} \textbf{ }%
\[
X,Y\in \mathcal{A}\Rightarrow \mu X+\nu Y\in \mathcal{A},\; \forall
\mu,\nu \geq0.
\]
}
\end{remark}

\begin{remark}
If we remove the condition of the positive homogeneity, then
$\mathcal{A}$ is called a convex acceptable set. In this course we
mainly study the coherent case. Once the rule of the acceptable set
is fixed, the minimum requirement of risk deposit is then
automatically determined.
\end{remark}

\begin{definition}
\textbf{(Risk measure related to a given acceptable set
$\mathcal{A}$) }{The functional }$\rho(\cdot)$ defined by {
\[\rho(X)=\rho_{\mathcal{A }}(X):=\inf \{m\in \mathbb{R}:\;m+X\in
\mathcal{A }\},\  \  \ X\in \mathcal{X}
\]
is called the risk measure related to }%
$\mathcal{A}$.
\end{definition}

It is easy to see that

{%
\[
\rho(X+\rho(X))=0.
\]
}

\begin{proposition}
{ \textbf{ }$\rho(\cdot)$ is a coherent risk measure, namely\newline%
\newline \textbf{(a) Monotonicity:} $X\geq Y$ implies $\rho(X)\leq \rho(Y)$;\,
\, \newline \textbf{(b) Constant preservation:}$\mathbf{\
}\rho(1)=-\rho(-1)=-1; $ \newline} {\textbf{(c) Sublinearity:} For
each $X,Y\in \mathcal{X}$, \textbf{ }$\rho(X+Y)\leq \rho(X)+\rho(Y);
$ }\newline \textbf{(d) Positive homogeneity:
}$\rho(\lambda X)=\lambda \rho(X),\  \  \forall \lambda \geq0.\mathbf{\ }%
\newline \ $
\end{proposition}

\begin{proof}
{ (a), (b) are obvious. To prove (c) we first have $\rho(\mu X)=\mu
\rho(X)$. In fact the case }${ \mu=0}$  is trivial; when $ \mu>0$,
\begin{align*}
\rho(\mu X)  &  =\inf \{m\in \mathbb{R}:\;m+\mu X\in \mathcal{A
}\} \  \ \ \ \ \ \ (n=\frac{m}{\mu})\\
&  =\mu \inf \{n\in \mathbb{R}:\;n+X\in \mathcal{A}\}=\mu \rho(X).
\end{align*}
Now, for each positive $\mu$ and $\nu$,
\begin{align*} \rho(\mu X+\nu Y)=  &  \inf
\{m\in \mathbb{R}:\;m+(\mu X+\nu Y)\in \mathcal{A
}\} \\
=  &  \inf \{m+n:m,n\in \mathbb{R},\;(m+\mu X)+(n+\nu Y)\in \mathcal{%
A
}\} \\
\  \leq &  \inf \{m\in \mathbb{R}:\;m+\mu X\in \mathcal{%
A
}\}+\inf \{n\in \mathbb{R}:\;n+\nu Y\in \mathcal{%
A}\} \\
=  &  \rho(\mu X)+\rho(\nu Y).
\end{align*}

\end{proof}

Now let a ${ \rho(\cdot)}$ be a functional satisfying (a)--(d).
Then we can inversely define %
\[
\mathcal{A}_{\rho}{ =\{X\in X:\rho(X)\leq0\}.}%
\]
{ It is easy to prove that }$A_{\rho}${ satisfies (i)--(iv).}

\subsection{{Sublinear expectation of risk loss}}

{ From now we will denote $X$ to be a loss position, namely $-X$ is
the corresponding financial position. Related to a coherent risk
measure $\rho$, we evaluate the risk loss $X$ by:
\[
\mathbb{\hat{E}}[X]:=\rho(-X),\  \  \  \ X\in \mathcal{X}.
\]
This functional satisfies the following properties: }

\noindent{\textbf{(a) Monotonicity: }%
\[
X\geq Y\  \  \implies \  \  \mathbb{\hat{E}}[X]\geq
\mathbb{\hat{E}}[Y].
\]
}

\noindent\textbf{(b) Constant preserving}
\[
\mathbb{\hat{E}}[c]=c,\  \  \  \forall c\in \mathbb{R}.
\]

\noindent\textbf{(c) Sub-additivity: } For each $X,Y\in
\mathcal{X}$,
\[
\mathbb{\hat{E}}[X+Y]\leq \mathbb{\hat{E}}[X]+\mathbb{\hat{E}}[Y].
\]

\noindent\textbf{(d) Positive homogeneity:}
\[
\mathbb{\hat{E}}[\lambda X]=\lambda \mathbb{\hat{E}}[X],\  \  \
\forall
\lambda \geq0\text{.}%
\]

A real valued functional $\mathbb{\hat{E}}[\cdot]$ defined on
$\mathcal{X}$ satisfying (a)--(d) will be called a sublinear
expectation and will be systematically studied in this lecture.

\begin{remark}
 (c)+(d) is called sublinearity. This sublinearity implies: \newline(e)
\textbf{Convexity:
\[
\mathbb{\hat{E}}[\alpha X+(1-\alpha)Y]\leq \alpha \mathbb{\hat{E}}%
[X]+(1-\alpha)\mathbb{\hat{E}}[Y],\  \  \forall \alpha \in
\lbrack0,1];
\]
}
\end{remark}

\begin{remark}
\textbf{(b) + (d)} implies \newline(f) \textbf{Cash
translatability:}
\[
\mathbb{\hat{E}}[X+c]=\mathbb{\hat{E}}[X]+c.
\]
Indeed, we have
\begin{align*}
\mathbb{\hat{E}}[X]+c  &  =\mathbb{\hat{E}}[X]-\mathbb{\hat{E}}[-c]\\
&  \leq \mathbb{\hat{E}}[X+c]\leq \mathbb{\hat{E}}[X]+\mathbb{\hat{E}%
}[c]=\mathbb{\hat{E}}[X]+c.
\end{align*}

\end{remark}

\begin{remark}
{ \textbf{(c) + (d)} $\iff$ } (d) + the following
{\textbf{Convexity:}
\[
\mathbb{\hat{E}}[\alpha X+(1-\alpha)Y]\leq \alpha \mathbb{\hat{E}}%
[X]+(1-\alpha)\mathbb{\hat{E}}[Y],\  \  \forall \alpha \in
\lbrack0,1].
\]
}
\end{remark}

\begin{remark}
{ \label{Rem-1}We recall the notion of the above expectations
satisfying (c)--(d) was systematically introduced by Artzner,
Delbaen, Eber and Heath \cite{ADEH1}, \cite{ADEH2}, in the case
where $\Omega$ is a finite set, and by Delbaen \cite{Delbaen} in
general situation with the notation of risk measure:
$\rho(X)=\mathbb{\hat{E}}[-X]$. See also in Huber \cite{Huber} for
even early study of this notion $\mathbb{\hat{E}}$ (called upper
expectation $\mathbf{E}^{\ast}$ in Ch.10 of \cite{Huber}) in a
finite set $\Omega$. See Rosazza Gianin \cite{Roazza2003} or Peng
\cite{Peng2003}, El Karoui \& Barrieu \cite{El-Bar},
\cite{El-Bar2005} for dynamic risk measures using $g$--expectations.
Super-hedging and super pricing (see \cite{EQ} and \cite{EPQ}) are
also closely related to this formulation. }
\end{remark}

\subsection{Examples of sublinear expectations}

Let $\mathbb{\hat{E}}^{1}$ and $\mathbb{\hat{E}}^{2}$ be two
nonlinear expectations defined on $(\Omega,\mathcal{X})$.
$\mathbb{\hat{E}}^{1}$ is said to be \ dominated by
$\mathbb{\hat{E}}^{2}$ if
\[
\mathbb{\hat{E}}^{1}[X]-\mathbb{\hat{E}}^{1}[Y]\leq
\mathbb{\hat{E}}^{2}[X-Y],\ \ \  \ \forall X, Y\in \mathcal{X}.
\]
The strongest sublinear expectation on $\mathcal{X}$ is
\[
\mathbb{\hat{E}}^{\infty}[X]:=X^{\ast}=\sup_{\omega \in
\Omega}X(\omega).
\]
Namely, all other sublinear expectations are dominated by $\mathbb{\hat{E}%
}^{\infty}[\cdot]$. From (c) a sublinear expectation is dominated by
itself.

\begin{notation}
 Let $\mathcal{P}_{f}$ be the collection of all finitely additive
probability measures on $(\Omega,\mathcal{F})$.
\end{notation}

\begin{example}
{ (A \textbf{linear expectation) }We consider $\mathbb{L}_{0}^{\infty}%
(\Omega,\mathcal{F})$ the collection of risk positions with finite
values. It
is a subspace of $\mathcal{X}$ consisting of risk positions $X$ of the form%
\begin{equation}
X(\omega)=\sum_{i=1}^{N}x_{i}\mathbf{1}_{A_{i}}(\omega),\ x_{i}\in
\mathbb{R},\ A_{i}\in \mathcal{F},i=1,\cdots,N.\  \  \label{xiAi}%
\end{equation}
It is easy to check that, under the norm $\left \Vert \cdot \right
\Vert _{\infty}$, $\mathbb{L}_{0}^{\infty}(\Omega,\mathcal{F})$ is
dense on $\mathbb{L}^{\infty}(\Omega,\mathcal{F})$. For a fixed
$Q\in \mathcal{P}_{f}$ and $X\in
\mathbb{L}_{0}^{\infty}(\Omega,\mathcal{F})$ we define
\[
E_{Q}[X]=E_{Q}[\sum_{i=1}^{N}x_{i}\mathbf{1}_{A_{i}}(\omega)]:=\sum_{i=1}%
^{N}x_{i}Q(A_{i})=\int_{\Omega}X(\omega)Q(d\omega)
\]
$E_{Q}:\mathbb{L}_{0}^{\infty}(\Omega,\mathcal{F})\rightarrow
\mathbb{R}$ is a linear functional. It is easy to check that $E_{Q}$
satisfies (a)-(b). It is also continuous under $\left \Vert X\right
\Vert _{\infty}$.
\[
|E_{Q}[X]|\leq \sup_{\omega \in \Omega}|X(\omega)|=\left \Vert
X\right \Vert
_{\infty}\text{. }%
\]
Since $\mathbb{L}_{0}^{\infty}$ is dense in $\mathbb{L}^{\infty}$ we
then can extend $E_{Q}$ from $\mathbb{L}_{0}^{\infty}$ to a linear
continuous functional on $\mathbb{L}^{\infty}(\Omega,\mathcal{F})$.
}
\end{example}

\begin{proposition}
The functional $E_{Q}[\cdot]:${$\mathcal{X}\mapsto \mathbb{R}$
satisfies (a)
and (b). Inversely each linear functional }$\eta(\cdot):${$\mathcal{X}%
\mapsto \mathbb{R}$} satisfying (a) and (b) induces a finitely
additive probability measure via $Q_{\eta}(A)=\eta(\mathbf{1}_{A})$,
$A\in \mathcal{F}$.
The corresponding expectation is $\eta$ itself:%
\[
\eta(X)=\int_{\Omega}X(\omega)Q_{\eta}(d\omega).
\]

\end{proposition}

\begin{notation}
Let $Q\in \mathcal{P}_{f}$ be given and let $X:\Omega \mapsto
\mathbb{R}$ be a $\mathcal{F}$-measurable function such that
$|X(\omega)|<\infty$ for each $\omega$. The distribution of $X$ \ on
$(\Omega,\mathcal{F},Q)$ is defined as a linear functional
\[
F_{X}[\varphi]=E_{Q}[\varphi(X)]:\varphi \in \mathbb{L}^{\infty}(\mathbb{R}%
,\mathcal{B}(\mathbb{R}))\mapsto \mathbb{R}.
\]
$F_{X}[\cdot]$ is a linear functional satisfying (a) and (b). Thus
it induces
a finitely additive probability measure on $(\mathbb{R},\mathcal{B}%
(\mathbb{R}))$ via $F_{X}(B):=F_{X}(\mathbf{1}_{B})\,$, $B\in \mathcal{B}%
(\mathbb{R})$. We have%
\[
F_{X}[\varphi]=\int_{\mathbb{R}}\varphi(x)F_{X}(dx).
\]

\end{notation}

\begin{remark}
Usually, people call the probability measure $F_{X}(\cdot)$ to be
the distribution of $X$ under $Q$. But we will see that in sublinear
or more general situation the functional version is necessary.
\end{remark}

\begin{definition}
Let $\mathbb{E}$ be a linear finitely additive probability defined
on a measurable space $(\Omega,\mathcal{F})$. We say that a random
variable $Y$ is independent to $X$ if, for each $\varphi \in
\mathbb{L}^{\infty}(\mathbb{R}^{2})$
\[
\mathbb{E}[\varphi(X,Y)]=\mathbb{E}[\mathbb{E}[\varphi(x,Y)]_{x=X}].
\]

\end{definition}

\subsection{Representation of a sublinear expectation}

Let a risk supervisor take a linear expectation $E[\cdot]$ to be his
risk measure. This means that he takes the corresponding finitely
additive probability induced by $E[\cdot]$ as his probability
measure. But in many cases, he cannot decide precisely which
probability he should take. He has a set of finitely additive
probabilities $\mathcal{Q}$. The size of $\mathcal{Q}$ characterizes
his model-uncertainty. A robust risk measure under such model uncertainty is:%
\[
\mathbb{\hat{E}}^{\mathcal{Q}}[X]=\sup_{Q\in \mathcal{Q}}E_{Q}[X]:\mathcal{X}%
\mapsto \mathbb{R}.
\]
It is easy to prove that $\mathbb{\hat{E}}^{\mathcal{Q}}[\cdot]$ is
a sublinear expectation.

\begin{theorem}
{ A sublinear expectation $\mathbb{\hat{E}}[\cdot]$ has the
following representation: there exists a subset $\mathcal{Q}\subset
\mathcal{P}_{f}$, such that
\[
\mathbb{\hat{E}}[X]=\max_{Q\in \mathcal{Q}}E_{Q}[X],\  \  \  \forall
X\in \mathcal{X}.
\]
}
\end{theorem}

{ \textbf{Proof. }It suffices to prove that, for each $X_{0}\in
\mathcal{X}$,
there exists a $Q_{X_{0}}\in \mathcal{P}_f$ such that }$E_{Q_{X_{0}}}%
[${$X]\leq \mathbb{\hat{E}}[X]$, for all $X$, and such that $Q_{X_{0}}%
(X_{0})=\mathbb{\hat{E}}[X_{0}]$. We only consider the case $\mathbb{\hat{E}%
}[X_{0}]=1$ (otherwise we may consider $\bar{X}_{0}=X_{0}-\mathbb{\hat{E}%
}[X_{0}]+1$). }

{ Let $U_{1}:=\{X:\mathbb{\hat{E}}[X]<1\}$. Since $U_{1}$ is an open
and convex set, by the well-known separation theorem for convex
sets, there exists a continuous linear functional $\eta$ on
$\mathcal{X}$ such that $\eta (X)<\eta(X_{0})$, for all $X\in
U_{1}$. Since $0\in U_{1}$ we have particularly $\eta(X_{0})>0$, and
we then can normalize $\eta$ so that
$\eta(X_{0})=1=\mathbb{\hat{E}}[X_{0}]$. Thus }$\eta$ satisfies:%
\[
\eta(X)<1,\  \  \  \forall X\in \mathcal{X},\  \  \text{s.t.\ }\mathbb{\hat{E}%
}[X]<1.
\]
By the positive homogeneity of $\mathbb{\hat{E}}$, for each $c>0$,%
\[
\eta(X)<c,\  \  \  \forall X\in \mathcal{X},\  \  \text{s.t.\ }\mathbb{\hat{E}%
}[X]<c.
\]
This implies that%
\begin{equation}
\mathbb{\hat{E}}[X]\geq \eta(X),\  \text{for all }X\text{ such that
}\mathbb{\hat{E}}[X]>0. \label{eq-Domine}%
\end{equation}
{ We now prove that }$\eta$ is monotone in the sense of (a). {For a
given $Y\geq0$ and $\lambda>0$, since $-\lambda Y\in U_{1}$, we have
\[
\eta(-\lambda Y)<\eta(X_{0}),\  \  \text{thus
}\eta(Y)>-\lambda^{-1}\eta (X_{0})=-\lambda^{-1},\text{\  \  \
}\forall \lambda>0.
\]
From which it follows that $\eta(Y)\geq0$ and thus (a) holds true.
}From (\ref{eq-Domine}) we have $\eta(1)\leq \mathbb{\hat{E}}[1]=1$.
{We now prove that }$\eta(1)=1$. { Indeed, for each $c>1$, we have
\[
\mathbb{\hat{E}}[2X_{0}-c]=2\mathbb{\hat{E}}[X_{0}]-c=2-c<1,
\]
hence }$2X_{0}-c\in U_{1}$ and{
\[
\eta(2X_{0}-c)=2-c\eta(1)<1\  \  \text{or }\eta(1)>\frac{1}{c},\  \
\forall
c>1\text{;}%
\]
hence $\eta(1)=1$. This, together with (\ref{eq-Domine}),  yields
\[
\mathbb{\hat{E}}[X]\geq \eta(X),\  \text{for all }X\in \mathcal{X}.
\]
We thus proved the desired result. }


\chapter{LLN and Central Limit Theorem}

\section{Preliminary}

 The law of normalsize numbers (LLN) and central limit theorem
(CLT) are long and widely been known as two fundamental results in
the theory of probability and statistics. A striking consequence of
CLT is that accumulated independent and identically distributed
random variables tends to a normal distributed random variable
whatever is the original distribution. It is a very useful tool in
finance since many typical financial positions are accumulations of
a large number of small and independent risk positions. But CLT only
holds in cases of model certainty. In this section we are interested
in CLT with variance-uncertainty. We will prove that the accumulated
risk positions can converge `in law' to what we call $G$-normal
distribution, which is a distribution under sublinear expectation.
In a special case where the variance-uncertainty becomes zero, the
$G$-normal distribution becomes the classical normal distribution.
Technically we introduce a new method to prove a CLT under a
sublinear expectation space.

In the following two chapters we will consider the following type of
spaces of sublinear expectations: Let $\Omega$ be a given set and
let $\mathcal{H}$ be a linear space of real functions defined on
$\Omega$ such that if $X_{1},\cdots,X_{n}\in \mathcal{H}$ then
$\varphi(X_{1},\cdots,X_{n})\in \mathcal{H}$ for each $\varphi \in
C_{l.Lip}(\mathbb{R}^{n})$ where $C_{l.Lip}(\mathbb{R}^{n})$ denotes
the linear space of functions $\varphi$ satisfying
\begin{align*}
|\varphi(x)-\varphi(y)|  &  \leq C(1+|x|^{m}+|y|^{m})|x-y|,\  \
\forall
x,y\in \mathbb{R}^{n}\text{, \ }\\
\  &  \text{for some }C>0\text{, }m\in \mathbb{N}\text{ depending on
}\varphi.
\end{align*}
$\mathcal{H}$ is considered as a space of \textquotedblleft random
variables\textquotedblright.

\begin{remark}
In particular, if $X,Y\in \mathcal{H}$, then $|X|$, $X^{m}\in
\mathcal{H}$ are in $\mathcal{H}$. More generally
$\varphi(X)\psi(Y)$ is still in $\mathcal{H}$ if $\varphi, \psi\in
C_{l.Lip}(\mathbb{R})$.
\end{remark}

Here we use $C_{l.Lip}(\mathbb{R}^{n})$ in our framework only for
some convenience of techniques. In fact our essential requirement is
that $\mathcal{H}$ contains all constants and, moreover, $X\in
\mathcal{H}$ implies
$\left \vert X\right \vert \in \mathcal{H}$. In general $C_{l.Lip}(\mathbb{R}%
^{n})$ can be replaced by the following spaces of functions defined
on $\mathbb{R}^{n}$.

\begin{itemize}
\item { $\mathbb{L}^{\infty}(\mathbb{R}^{n})$: the space bounded
Borel-measurable functions; }

\item { $C_{b}(\mathbb{R}^{n})$: the space of bounded and continuous functions;
}

\item {$C_{b}^{k}(\mathbb{R}^{n})$: the space of bounded and $k$-time
continuously differentiable functions with bounded derivatives of
all orders less than or equal to  $k$;}

\item { $C_{unif}(\mathbb{R}^{n})$: the space of bounded and uniformly
continuous functions; }

\item { $C_{b.Lip}(\mathbb{R}^{n})$: the space of bounded and Lipschitz
continuous functions; }

\item { $L^{0}(\mathbb{R}^{n})$: the space of Borel measurable functions. }
\end{itemize}

\begin{definition}
{\label{Def-1} { \textbf{Sublinear expectation}$\mathbb{\hat{E}}$ on
$\mathcal{H}$ is a functional $\mathbb{\hat{E} }: \mathcal{H}\mapsto
\mathbb{R}$ satisfying the
following properties: for all $X,Y\in \mathcal{H}$, we have\newline%
\  \  \newline \textbf{(a) Monotonicity:} \ \ \ \ \ \ \ \ \ \ \ \ \ If $X\geq Y$ then $\mathbb{\hat{E}%
}[X]\geq \mathbb{\hat{E}}[Y].$\newline \textbf{(b) Constant
preserving: \  \ }\ \   $\mathbb{\hat{E}}[c]=c$.\newline
\textbf{(c)} \textbf{Sub-additivity:
\  \  \  \ }}}\ \ \ \ \ \ \ \ $\mathbb{\hat{E}}[X]-\mathbb{\hat{E}}[Y]\leq \mathbb{\hat{E}}%
[X-Y].$\newline{{\textbf{(d) Positive homogeneity: } \
$\mathbb{\hat{E}}[\lambda X]=\lambda \mathbb{\hat{E}}[X]$,$\  \
\forall \lambda \geq0$.\newline}}\newline(In many situation
{{\textbf{(c) }}}is also called property of self--domination). The
triple $(\Omega,\mathcal{H},\mathbb{\hat{E}}\mathbb{)}$ is called a
\textbf{sublinear expectation space} (compare with a probability
space $(\Omega,\mathcal{F},\mathbb{P})$).
\end{definition}

\begin{example}
{ In a game we select at random a ball from a box containing $W$
white, $B$ black and $Y$ yellow balls. The owner of the box, who is
the banker of the game, does not tell us the exact numbers of $W,B$
and $Y$. He or she only informs us $W+B+Y=100$ and $W=B\in
\lbrack20,25]$. Let $\xi$ be a random variable
\[
\xi=\left \{
\begin{array}
[c]{rcc}%
1 &  & \text{if we get a white ball;}\\
0 &  & \text{if we get a yellow ball;}\\
-1 &  & \text{if we get a black ball.}%
\end{array}
\right.
\]
Problem: How to measure a loss $X=\varphi(\xi)$, for a given
function
$\varphi$ on $\mathbb{R}$. We know that the distribution of $\xi$ is%
\[
\left \{
\begin{array}
[c]{ccc}%
-1 & 0 & 1\\
\frac{p}{2} & 1-p & \frac{p}{2}%
\end{array}
\right \}  \  \  \text{with uncertainty: $p\in$}[\underline{\sigma}^{2}%
,\overline{\sigma}^{2}]=[0.4,0.5].
\]
Thus the robust expectation of $X=\varphi(\xi)$ is:%
\begin{align*}
\hat{\mathbb{E}}[\varphi(\xi)]  &  :=\sup_{P\in
\mathcal{P}}E_{P}[\varphi
(\xi)]\\
&  =\sup_{p\in \lbrack
\underline{\sigma^{2}},\overline{\sigma}^{2}]}[\frac
{p}{2}[\varphi(1)+\varphi(-1)]+(1-p)\varphi(0)].
\end{align*}
$\xi$ has distribution uncertainty. }
\end{example}

\begin{example}
{ A more general situation is that the banker of a game can choose
among a set of distribution }${\{F(\theta,A)}\}_{A\in
\mathcal{B}(\mathbb{R}),\theta \in \Theta}${ of a random variable
}$\xi${ . In this situation the robust expectation of a risk
position $\varphi(\xi)$ for some $\varphi\in C_{l.Lip}(\mathbb{R})$
is:}
\[
\hat{\mathbb{E}}[\varphi(\xi)]:=\sup_{\theta \in \Theta}\int_{\mathbb{R}}%
\varphi(x)F(\theta,dx).
\]

\end{example}

\section{{Distributions and independence}}

We now consider the notion of the distributions of random variables
under
sublinear expectations. Let $X=(X_{1},\cdots,X_{n})$ be a given $n$%
-dimensional random vector on a sublinear expectation space
$(\Omega_{1},\mathcal{H}_{1},\mathbb{\hat{E}})$. We define a
functional on $C_{l.Lip}(\mathbb{R}^{n})$ by
\begin{equation}
\mathbb{\hat{F}}_{X}[\varphi]:=\mathbb{\hat{E}}[\varphi(X)]:\varphi
\in
C_{l.Lip}(\mathbb{R}^{n})\mapsto(-\infty,\infty). \label{X-Distr}%
\end{equation}
The triple $(\mathbb{R}^{n},C_{l.Lip}(\mathbb{R}^{n}),\mathbb{\hat{F}}%
_{X}[\cdot])$ forms a sublinear expectation space.
$\mathbb{\hat{F}}_{X}$ is called the distribution of $X$.

\begin{definition}
Let $X_{1}$ and $X_{2}$ be two $n$--dimensional random vectors
defined
respectively in {sublinear expectation spaces }$(\Omega_{1},\mathcal{H}%
_{1},\mathbb{\hat{E}}_{1})${ and
}$(\Omega_{2},\mathcal{H}_{2},\mathbb{\hat {E}}_{2})$. They are
called identically distributed, denoted by $X_{1}\sim X_{2}$, if
\[
\mathbb{\hat{E}}_{1}[\varphi(X_{1})]=\mathbb{\hat{E}}_{2}[\varphi
(X_{2})],\  \  \  \forall \varphi \in C_{l.Lip}(\mathbb{R}^{n}).
\]
It is clear that $X_{1}\sim X_{2}$ if and only if their
distributions coincide.
\end{definition}

\begin{remark}
If the distribution $\mathbb{\hat{F}}_{X}$ of $X\in \mathcal{H}$ is
not a linear expectation, then $X$ is said to have distributional
uncertainty. The distribution of $X$ has the following four typical
parameters:
\[
\overline{\mu}:=\hat{\mathbb{E}}[X],\  \  \underline{\mu}:=-\mathbb{\hat{E}%
}[-X],\  \  \  \  \  \  \  \  \overline{\sigma}^{2}:=\hat{\mathbb{E}}[X^{2}%
],\  \  \underline{\sigma}^{2}:=-\hat{\mathbb{E}}[-X^{2}].\  \
\]
The subsets $[\underline{\mu},\overline{\mu}]$ and $[\underline{\sigma}%
^{2},\overline{\sigma}^{2}]$ characterize the mean-uncertainty and
the variance-uncertainty of $X$. The problem of mean uncertainty
have been studied in [Chen-Epstein] using the notion of
$g$-expectations. In this lecture we are mainly concentrated on the
situation of variance-uncertainty and thus set
$\overline{\mu}=\underline{\mu}$.
\end{remark}

The following simple property is very useful in our sublinear
analysis.

\begin{proposition}
{ \label{Prop-X+Y}Let $X,Y\in \mathcal{H}$ be such that $\mathbb{\hat{E}%
}[Y]=-\mathbb{\hat{E}}[-Y]$, i.e. }$Y$ has no mean uncertainty.{ Then we have%
\[
\mathbb{\hat{E}}[X+Y]=\mathbb{\hat{E}}[X]+\mathbb{\hat{E}}[Y].
\]
In particular, if $\mathbb{\hat{E}}[Y]=\mathbb{\hat{E}}[-Y]=0$, then
$\mathbb{\hat{E}}[X+Y]=\mathbb{\hat{E}}[X]$. }
\end{proposition}

\begin{proof}
{ It is simply because we have $\mathbb{\hat{E}}[X+Y]\leq \mathbb{\hat{E}%
}[X]+\mathbb{\hat{E}}[Y]$ and
\[
\mathbb{\hat{E}}[X+Y]\geq
\mathbb{\hat{E}}[X]-\mathbb{\hat{E}}[-Y]=\mathbb{\hat
{E}}[X]+\mathbb{\hat{E}}[Y]\text{.}%
\]
}
\end{proof}

The following notion of independence plays a key role:

\begin{definition}
In a sublinear expectation space
$(\Omega,\mathcal{H},\mathbb{\hat{E}})$ a random vector
$Y=(Y_1,\cdots,Y_n)$, $Y_i\in \mathcal{H}$ is said to be independent
to another random vector $X=(X_1,\cdots,X_m)$, $X_i\in \mathcal{H}$
under $\mathbb{\hat{E}}[\cdot]$ if for each test function $\varphi
\in C_{l.Lip}(\mathbb{R}^{m}\times \mathbb{R}^{n})$ we have
\[
\mathbb{\hat{E}}[\varphi(X,Y)]=\mathbb{\hat{E}}[\mathbb{\hat{E}}%
[\varphi(x,Y)]_{x=X}].
\]
A random variable $Y\in \mathcal{H}$ is said to be weakly
independent to $X\in \mathcal{H}^{\otimes m}${ if the above test
functions
$\varphi$ are only taken from the following class:%
\[
\varphi(x,y)=\psi_{0}(x)+\psi_{1}(x)y+\psi_{2}(x)y^{2},\  \
\psi_{i}\in C_{l.Lip}(\mathbb{R}^{m}\times \mathbb{R}),\  \ i=0,1,2.
\]
}
\end{definition}

\begin{remark}
{In the case of linear expectation, this notion of independence is
just the classical one. It is important to note that under sublinear
expectations the condition \textquotedblleft$Y$ is independent to
$X$\textquotedblright \ does not implies automatically that
\textquotedblleft$X$ is independent to $Y$\textquotedblright. }
\end{remark}

\begin{example}
We consider a case where $X,Y\in \mathcal{H}$ are identically
distributed and
$\mathbb{\hat{E}}[X]=\mathbb{\hat{E}}[-X]=0$ but $\overline{\sigma}%
^{2}=\mathbb{\hat{E}}[X^{2}]>\underline{\sigma}^{2}=-\mathbb{\hat{E}}[-X^{2}]$.
We also assume that
$\mathbb{\hat{E}}[|X|]=\mathbb{\hat{E}}[X^{+}+X^{-}]>0$,
thus $\mathbb{\hat{E}}[X^{+}]=\frac{1}{2}\mathbb{\hat{E}}[|X|+X]=$$\frac{1}%
{2}\mathbb{\hat{E}}[|X|]>0$. In the case where $Y$ is independent to
$X$, we
have%
\[
\mathbb{\hat{E}}[XY^{2}]=\mathbb{\hat{E}}[X^{+}\overline{\sigma}^{2}%
-X^{-}\underline{\sigma}^{2}]=(\overline{\sigma}^{2}-\underline{\sigma}%
^{2})\mathbb{\hat{E}}[X^{+}]>0.
\]
But if $X$ is independent to $Y$ we have%
\[
\mathbb{\hat{E}}[XY^{2}]=0.
\]

\end{example}

The independence property of two random vectors $X,Y$ involves only
the joint distribution of $(X,Y)$. The following result tell us how
to construct random vectors with given sublinear distributions and
with joint independence.

\begin{proposition}
Let $X_{i}$ be $n_{i}$-dimensional random vectors respectively in
sublinear expectation spaces $(\Omega_{i},\mathcal{H}_{i},$
$\mathbb{\hat{E}}_{i})$, $i=1,\cdots,N$. We can construct random
vectors $Y_{1},\cdots,Y_{N}$ in a new sublinear expectation space
$(\Omega,\mathcal{H},\mathbb{\hat{E}})$ such that
$Y_{i}\sim X_{i}$ and such that $Y_{i+1}$ is independent to $(Y_{1}%
,\cdots,Y_{i})$, for each $i$.\label{Prop-2-10}
\end{proposition}

\begin{proof}
We first consider the case $N=2$. We set: $\Omega=\mathbb{R}^{n_{1}}%
\times \mathbb{R}^{n_{2}}$, i.e., $\omega=(x_{1},x_{2})\in \Omega,\  \ x_{1}%
\in \mathbb{R}^{n_{1}}$ $x_{2}\in \mathbb{R}^{n_{2}}$; a space of
random
variables $\mathcal{H}=\{X(\omega)=\varphi(\omega),\varphi \in C_{l.Lip}%
(\Omega)\}$ and a functional $\mathbb{\hat{E}}$ on
$(\Omega,\mathcal{H})$
defined by%
\begin{align*}
\mathbb{\hat{E}}[\varphi(Y)]  &  =\mathbb{\hat{E}}_{1}[\bar{\varphi}%
(X_{1})],\  \ \hbox{ where } \ \  \bar{\varphi}(x_{1}):=\mathbb{\hat{E}}_{2}[\varphi(x_{1}%
,X_{2})],\  \ x_{1}\in \mathbb{R}^{n_{1}},\  \  \\
\forall Y(\omega)  &  =\varphi(\omega),\  \  \  \varphi \in C_{l.Lip}%
(\mathbb{R}^{n_{1}}\times \mathbb{R}^{n_{2}}).
\end{align*}
It is easy to check that $\mathbb{\hat{E}}$ forms a sublinear
expectation on $(\Omega,\mathcal{H})$. Now let us consider two
random vectors in $\mathcal{H}$:
\[
Y_{1}(\omega)=x_{1},\ Y_{2}(\omega)=x_{2},\  \  \omega=(x_{1},x_{2}%
)\in \mathbb{R}^{n_{1}}\times \mathbb{R}^{n_{2}}.
\]
It is easy to check that $Y_{1}$, $Y_{2}$ meet our requirement:
$Y_{2}$ is independent to $Y_{1}$ under $\mathbb{\hat{E}}$, and
$Y_{1}\sim X_{1}$, $Y_{2}\sim X_{2}$. The case of $N>2$ can be
proved by repeating the above procedure.
\end{proof}

\begin{example}
{We consider a situation where two random variables $X$ and $Y$ in
$\mathcal{H}$ are identically distributed and their common
distribution is
\[
\mathbb{\hat{F}}_{X}[\varphi]=\mathbb{\hat{F}}_{Y}[\varphi]=\sup_{\theta
\in \Theta}\int_{\mathbb{R}}\varphi(y)F(\theta,dy),\  \  \varphi \in
C_{l.Lip}(\mathbb{R}).
\]
where, for each $\theta \in \Theta$, $\{F(\theta,A)\}_{A\in \mathcal{B}%
(\mathbb{R})}$ is a probability measure on $(\mathbb{R},\mathcal{B}%
(\mathbb{R}))$. In this case \textquotedblleft$Y$ is independent to
$X$\textquotedblright \ means that the joint distribution of $X$ and
$Y$ is:
\begin{align*}
\mathbb{\hat{F}}_{X,Y}[\psi]  &  =\sup_{\theta_{1}\in \Theta}\int_{\mathbb{R}%
}\left[  \sup_{\theta_{2}\in
\Theta}\int_{\mathbb{R}}\psi(x,y)F(\theta
_{2},dy)\right]  F(\theta_{1},dx),\  \  \\
\psi &  \in C_{l.Lip}(\mathbb{R}^{2}).
\end{align*}
}
\end{example}

\begin{remark}
{The situation \textquotedblleft$Y$ is independent to $X$\textquotedblright%
 often appears when $Y$ occurs after $X$, thus a very
robust expectation should take the information of $X$ into account.
}
\end{remark}

\begin{definition}
A sequence of random variables $\left \{  \eta_{i}\right \}
_{i=1}^{\infty}$ in $\mathcal{H}$ is said to converge in
distribution under $\mathbb{\hat{E}}$ if for each bounded and
$\varphi \in C_{b.Lip}(\mathbb{R})$, $\left \{
\mathbb{\hat{E}}[\varphi(\eta_{i})]\right \}  _{i=1}^{\infty}$
converges.
\end{definition}

\section{$G$-normal distributions}

\subsection{$G$-normal distribution}

A fundamentally important distribution in sublinear expectation
theory is
$\mathcal{N}(0;[\underline{\sigma}^{2},\overline{\sigma}^{2}])$-distributed
random variable $X$ under $\hat{\mathbb{E}}[\cdot]$:

\begin{definition}
(\textbf{$G$-normal distribution}) {\label{Def-Gnormal}}In a
sublinear expectation space $(\Omega,\mathbb{\hat{E}},\mathcal{H})$,
a random variable
$X\in \mathcal{H}$ with%
\[
\overline{\sigma
}^{2}=\mathbb{\hat{E}}[X^{2}],\  \  \underline{\sigma}^{2}=-\mathbb{\hat{E}%
}[-X^{2}]
\]
is said to be $\mathcal{N}(0;[\underline{\sigma}^{2},\overline{\sigma}^{2}%
])$-distributed, denoted by $X\sim \  \mathcal{N}(0;[\underline{\sigma}%
^{2},\overline{\sigma}^{2}])$, if for each $Y\in\mathcal{H}$ which
is independent
to $X$ such that $Y\sim X$ we have%
\begin{equation}
aX+bY\sim \sqrt{a^{2}+b^{2}}X,\  \  \  \forall a,b\geq0. \label{G-normal}%
\end{equation}
\end{definition}

Proposition \ref{Prop-2-10} tells us how to construct the above $Y$
based on the distribution of $X$.

\begin{remark} By the above definition, we have
$\sqrt{2}\mathbb{\hat{E}}[X]=\mathbb{\hat{E}}[X+Y]=2\mathbb{\hat{E}}[X]$
and  $\sqrt{2}\mathbb{\hat{E}}[-X]=\mathbb{\hat{E}}[-X-Y]=2\mathbb{\hat{E}%
}[-X]$ it follows that
\[
\mathbb{\hat{E}}[X]=\mathbb{\hat{E}}[-X]=0,
\]
Namely an $\mathcal{N}(0;[\underline{\sigma}^{2},\overline{\sigma}^{2}%
])$-distributed random variable $X$ has no mean uncertainty.
\end{remark}

\begin{remark}
If $X$ is independent to $Y$ and $X\sim Y$, such that
(\ref{G-normal}) satisfies, then $-X$ is independent to $-Y$,
$-X\sim-Y$. We also have
$a(-X)+b(-Y)\sim \sqrt{a^{2}+b^{2}}(-X)$, $a,b\geq0$.
Thus
\[
X\sim \mathcal{N}(0;[\underline{\sigma}^{2},\overline{\sigma}^{2}])
\ \ \ \hbox{\textit { iff }  } \ \ \  -X\sim
\mathcal{N}(0;[\underline{\sigma}^{2},\overline{\sigma}^{2}]).
\]
\end{remark}

The following proposition and corollary show that
$\mathcal{N}(0;[\underline {\sigma}^{2},\overline{\sigma}^{2}])$ is
a uniquely defined sublinear distribution on
$(\mathbb{R},C_{l.Lip}(\mathbb{R}))$. We will show that an
{$\mathcal{N}(0;[\underline{\sigma}^{2},\overline{\sigma}^{2}])$
distribution is characterized, or generated, by the following
parabolic PDE }defined on $[0,\infty)\times \mathbb{R}$:
\begin{equation}
\partial_{t}u-G(\partial_{xx}^{2}u)=0, \label{eq-G-heat}%
\end{equation}
with Cauchy condition$\  \ u|_{t=0}=\varphi$, where $G$, called the
\textbf{generating function} of (\ref{eq-G-heat}), is the following
sublinear real function parameterized by $\underline{\sigma}$ and
$\overline{\sigma}$:
\[
G(\alpha):=\frac{1}{2}\mathbb{\hat{E}}[X^2\alpha]=\frac{1}{2}(\overline{\sigma}^{2}\alpha^{+}-\underline{\sigma}%
^{2}\alpha^{-}),\  \  \alpha \in \mathbb{R}.
\]
Here we denote $\alpha^{+}:=\max \{0,\alpha \}$ and
$\alpha^{-}:=(-\alpha)^{+}$. (\ref{eq-G-heat}) is called generating
heat equation, or $G$-heat
equation of the sublinear distribution $\mathcal{N}(0;[\underline{\sigma}%
^{2},\overline{\sigma}^{2}])$. We also call this sublinear
distribution \textbf{$G$-normal distribution}.

\begin{remark}
We will use the notion of viscosity solutions to the generating heat
equation (\ref{eq-G-heat}). This notion was introduced by Crandall
and Lions. For the existence and uniqueness of solutions and related
very rich references we refer to Crandall, Ishii and Lions
\cite{CIL}. We note that, in the situation where
$\underline{\sigma}^{2}>0$, the viscosity solution (\ref{eq-G-heat})
becomes a classical $C^{1+\frac{\alpha}{2},2+\alpha}$-solution (see
\cite{Krylov}, \cite{Krylov1} and the recent works of
\cite{Caff1997} and \cite{WangL}). Readers can understand
(\ref{eq-G-heat}) in the classical meaning.
\end{remark}

\begin{definition}
\text{ A} real-valued continuous function $u\in C([0,T]\times
\mathbb{R})$ is called a viscosity subsolution (respectively,
supersolution) of (\ref{eq-G-heat}) if, for each function $\psi \in
C_{b}^{3}((0,\infty)\times \mathbb{R})$ and for each minimum
(respectively, maximum) point $(t,x)\in(0,\infty)\times \mathbb{R}$
of $\psi -u$, we have
\[
\partial_{t}\psi-G(\partial_{xx}^{2}\psi)\leq0\  \ (\text{respectively,  }\geq0).
\]
$u$ is called a viscosity solution of (\ref{eq-G-heat}) if it is
both super and subsolution.
\end{definition}

\begin{proposition}
Let $X$ be an
$\mathcal{N}(0;[\underline{\sigma}^{2},\overline{\sigma}^{2}])$
distributed random variable. For each $\varphi \in
C_{l.Lip}(\mathbb{R})$ we
define a function%
\[
u(t,x):=\mathbb{\hat{E}}[\varphi(x+\sqrt{t}X)],\ (t,x)\in \lbrack
0,\infty)\times \mathbb{R}.\
\]
Then we have%
\begin{equation}
u(t+s,x)=\mathbb{\hat{E}}[u(t,x+\sqrt{s}X)],\  \ s\geq0. \label{DPP}%
\end{equation}
We also have the estimates: For each $T>0$ there exist constants
$C,k>0$ such that, for all $t,s\in \lbrack0,T]$ and $x,y\in
\mathbb{R}$,
\begin{equation}
|u(t,x)-u(t,y)|\leq C(1+|x|^{k}+|y|^{k})|x-y| \label{|x-y|}%
\end{equation}
and%
\begin{equation}
|u(t,x)-u(t+s,x)|\leq C(1+|x|^{k})|s|^{1/2}. \label{|s|}%
\end{equation}
Moreover, $u$ is the unique viscosity solution, continuous in the
sense of (\ref{|x-y|}) and (\ref{|s|}), of the generating PDE
(\ref{eq-G-heat}).
\end{proposition}

\begin{proof}
Since%
\begin{align*}
u(t,x)-u(t,y)  &  =\mathbb{\hat{E}}[\varphi(x+\sqrt{t}X)]-\mathbb{\hat{E}%
}[\varphi(y+\sqrt{t}X)]\\
&  \leq \mathbb{\hat{E}}[\varphi(x+\sqrt{t}X)-\varphi(y+\sqrt{t}X)]\\
&  \leq \mathbb{\hat{E}}[C_{1}(1+|X|^{k}+|x|^{k}+|y|^{k})|x-y|]\\
&  \leq C(1+|x|^{k}+|y|^{k})|x-y|.
\end{align*}
We then have (\ref{|x-y|}). Let $Y$ be independent to $X$ such that
$X\sim Y$.
Since $X$ is {$\mathcal{N}(0;[\underline{\sigma}^{2},\overline{\sigma}^{2}%
])$-distributed, then}%
\begin{align*}
u(t+s,x)  &  =\mathbb{\hat{E}}[\varphi(x+\sqrt{t+s}X)]\\
&  =\mathbb{\hat{E}}[\varphi(x+\sqrt{s}X+\sqrt{t}Y)]\\
&  =\mathbb{\hat{E}}[\mathbb{\hat{E}}[\varphi(x+\sqrt{s}z+\sqrt{t}%
Y)]_{z=X}]\\
&  =\mathbb{\hat{E}}[u(t,x+\sqrt{s}X)].
\end{align*}
We thus obtain (\ref{DPP}). From this and (\ref{|x-y|}) it follows
that
\begin{align*}
u(t+s,x)-u(t,x)  &  =\mathbb{\hat{E}}[u(t,x+\sqrt{s}X)-u(t,x)]\\
&  \leq \mathbb{\hat{E}}[C_{1}(1+|x|^{k}+|X|^{k})|s|^{1/2}|X|].
\end{align*}
Thus we obtain (\ref{|s|}). Now, for a fixed
$(t,x)\in(0,\infty)\times \mathbb{R}$, let $\psi \in
C_{b}^{1,3}([0,\infty)\times \mathbb{R})$ be such that $\psi \geq u$
and $\psi(t,x)=u(t,x)$. By (\ref{DPP}) it follows that, for $\delta
\in(0,t)$
\begin{align*}
0  &  \leq \mathbb{\hat{E}}[\psi(t-\delta,x+\sqrt{\delta}X)-\psi(t,x)]\\
&  \leq-\partial_{t}\psi(t,x)\delta+\mathbb{\hat{E}}[\partial_{x}%
\psi(t,x)\sqrt{\delta}X+\frac{1}{2}\partial_{xx}^{2}\psi(t,x)\delta
X^{2}]+\bar{C}\delta^{3/2}\\
&  =-\partial_{t}\psi(t,x)\delta+\mathbb{\hat{E}}[\frac{1}{2}\partial_{xx}%
^{2}\psi(t,x)\delta X^{2}]+\bar{C}\delta^{3/2}\\
&  =-\partial_{t}\psi(t,x)\delta+\delta
G(\partial_{xx}^{2}\psi(t,x))+\bar {C}\delta^{3/2}.
\end{align*}
From which it is easy to check that
\[
\lbrack \partial_{t}\psi-G(\partial_{xx}^{2}\psi)](t,x)\leq0.
\]
It follows that $u$ is a viscosity supersolution of
(\ref{eq-G-heat}). Similarly we can prove that $u$ is a viscosity
subsolution of (\ref{eq-G-heat}).
\end{proof}

\begin{corollary}
If both $X$ and $\bar{X}$ are {$\mathcal{N}(0;[\underline{\sigma}%
^{2},\overline{\sigma}^{2}])$ distributed. Then }$X\sim \bar{X}$. In
particular, $X\sim-X$.
\end{corollary}

\begin{proof}
For each $\varphi \in C_{l.Lip}(\mathbb{R})$ we set
\[
u(t,x):=\mathbb{\hat{E}}[\varphi(x+\sqrt{t}X)],\
\bar{u}(t,x):=\mathbb{\hat {E}}[\varphi(x+\sqrt{t}\bar{X})],\
(t,x)\in \lbrack0,\infty)\times \mathbb{R}.
\]
By the above Proposition, both $u$ and $\bar{u}$ are viscosity
solutions of the G-heat equation (\ref{eq-G-heat}) with Cauchy
condition $u|_{t=0}=\bar {u}|_{t=0}=\varphi$. It follows from the
uniqueness of the viscosity solution
that $u\equiv \bar{u}$. In particular%
\[
\mathbb{\hat{E}}[\varphi(X)]=\mathbb{\hat{E}}[\varphi(\bar{X})]\text{.}%
\]
Thus $X\sim \bar{X}$.
\end{proof}

\begin{corollary}
In the case where $\underline{\sigma}^{2}=\overline{\sigma}^{2}>0$,
$\mathcal{N}(0;[\underline{\sigma}^{2},\overline{\sigma}^{2}])$ is
just the classical normal distribution
$\mathcal{N}(0;\overline{\sigma}^{2})$.
\end{corollary}

\begin{proof}
In fact the solution of the generating PDE (\ref{eq-G-heat}) becomes
a classical heat equation
\[
\partial_{t}u=\frac{\overline{\sigma}^{2}}{2}\partial_{xx}^{2}u,\  \ u|_{t=0}%
=\varphi
\]
where the solution is
\[
u(t,x)=\frac{1}{\sqrt{2\pi
\overline{\sigma}^{2}t}}\int_{-\infty}^{\infty
}\varphi(y)\exp(-\frac{(x-y)^{2}}{2\overline{\sigma}^{2}t})dy
\]
thus, for each $\varphi$,
\[
\mathbb{\hat{E}}[\varphi(X)]=u(1,0)=\frac{1}{\sqrt{2\pi \overline{\sigma}^{2}}%
}\int_{-\infty}^{\infty}\varphi(y)\exp(-\frac{y^{2}}{2\overline{\sigma}^{2}%
})dy.
\]

\end{proof}

In two typical situations the calculation of
$\hat{\mathbb{E}}[\varphi(X)]$ is very easy:

\begin{itemize}
\item (i) For each \textbf{convex} $\varphi$, we have
\[
\hat{\mathbb{E}}[\varphi(X)]=\frac{1}{\sqrt{2\pi \overline{\sigma}^{2}}}%
\int_{-\infty}^{\infty}\varphi(y)\exp(-\frac{y^{2}}{2\overline{\sigma}^{2}%
})dy
\]
Indeed, for each fixed $t\geq0$, it is easy to check that the
function
{$u(t,x):=\hat{\mathbb{E}}[\varphi(x+\sqrt{t}X)]$ is convex:}%
\begin{align*}
u(t,\alpha x+(1-\alpha)y)  &  =\mathbb{\hat{E}}[{\varphi(\alpha
x+(1-\alpha
)y+\sqrt{t}X)]}\\
&  \leq{}{{\alpha}\mathbb{\hat{E}}[\varphi(x+\sqrt{t}X)]+(1-{\alpha
)}\mathbb{\hat{E}}[\varphi(x+\sqrt{t}X)]}\\
&  =\alpha u(t,x)+(1-\alpha)u(t,y)
\end{align*}
It follows that $(\partial_{xx}^{2}u)^{-}\equiv0$ and thus the
G-heat equation
(\ref{eq-G-heat}) becomes {%
\[
\partial_{t}u=\frac{\overline{\sigma}^{2}}{2}\partial_{xx}^{2}u,\  \  \ u|_{t=0}%
=\varphi.\  \
\]
}

\item (ii) But for each \textbf{concave} $\varphi$, we have,%
\[
\hat{\mathbb{E}}[\varphi(X)]=\frac{1}{\sqrt{2\pi \underline{\sigma}^{2}}}%
\int_{-\infty}^{\infty}\varphi(y)\exp(-\frac{y^{2}}{2\underline{\sigma}^{2}%
})dy
\]

In particular%
\[
\hat{\mathbb{E}}[X]=\mathbb{\hat{E}}[-X]=0,\  \ \hat{\mathbb{E}}[X
^{2}]=\overline{\sigma}^{2},\  \
-\hat{\mathbb{E}}[-X^{2}]=\underline{\sigma
}^{2}%
\]
and
\[
\mathbb{\hat{E}}[X^{4}]=6\overline{\sigma}^{4},\ -\mathbb{\hat{E}}%
[-X^{4}]=6\underline{\sigma}^{4}\  \
\]

\end{itemize}

\subsection{Construction of $G$-normal distributed random
variables}

To construct an $\mathcal{N}(0;[\underline{\sigma}^{2},\overline{\sigma}%
^{2}])$--distributed random variable $\xi$, let $u=u^{\varphi}$ be
the unique viscosity solution of the $G$-heat equation
(\ref{eq-G-heat}) with $u^{\varphi}|_{t=0}=\varphi$.
Then we take $\widetilde{\Omega}=\mathbb{R}^{2}$, $\widetilde{\mathcal{H}%
}=C_{l.Lip}(\mathbb{R}^{2})$, $\omega=(x,y)\in \mathbb{R}^{2}$. The
corresponding sublinear expectation $\widetilde{\mathbb{E}}[\cdot]$
is defined
by, for each $X(\omega)=\psi(\omega)$ such that $\psi \in C_{l.Lip}%
(\mathbb{R}^{2})$,
\[
\widetilde{\mathbb{E}}[X]=u^{\bar{\psi}(\cdot)}(1,0),\text{ where we set }%
\bar{\psi}(x):=u^{\psi(x,\cdot)}(1,0).
\]
We now consider two random variables $\xi(\omega)=x$,
$\eta(\omega)=y$. It is clear that the
\[
\widetilde{\mathbb{E}}[\varphi(\xi)]=\widetilde{\mathbb{E}}[\varphi
(\eta)]=u^{\varphi}(1,0),\  \  \  \  \forall \varphi \in
C_{l.Lip}(\mathbb{R}).
\]
By the definition $\eta$ is independent to $\xi$ and $\xi \sim \eta$
under $\widetilde{\mathbb{E}}$. To prove that the distribution of
$\xi$ and $\eta$ satisfies condition (\ref{G-normal}), it suffices
to observe that, for each
$\varphi \in C_{l.Lip}(\mathbb{R})$ and for each fixed $\lambda>0$, $\bar{x}%
\in \mathbb{R}$, since the function $v$ defined by
$v(t,x):=u^{\varphi}(\lambda t,\bar{x}+\sqrt{\lambda}x)$ solves
exactly the same $G$-heat equation
(\ref{eq-G-heat}) but with Cauchy condition $v|_{t=0}=\varphi(\bar{x}%
+\sqrt{\lambda}\times \cdot)$, we then have
\[
\widetilde{\mathbb{E}}[\varphi(\bar{x}+\sqrt{\lambda}\xi)]=v(1,\bar
{x})=u^{\varphi(\sqrt{\lambda}\times \cdot)}(1,\bar{x})=u^{\varphi}%
(\lambda,\bar{x}),\  \bar{x}\in \mathbb{R}.\
\]
Thus, for each $t>0$ and $s>0$,
\begin{align*}
\widetilde{\mathbb{E}}[\varphi(\sqrt{t}\xi+\sqrt{s}\eta)] &
=\widetilde
{\mathbb{E}}[\widetilde{\mathbb{E}}[\varphi(\sqrt{t}x+\sqrt{s}\eta)]_{x=\xi
}]\\
&  =u^{u^{\varphi}(s,\cdot)}(t,0)=u^{\varphi}(t+s,0)\\
&  =\widetilde{\mathbb{E}}[\varphi(\sqrt{t+s}\xi)].
\end{align*}
Namely $\sqrt{t}\xi+\sqrt{s}\eta \sim \sqrt{t+s}\xi$. Thus $\xi$ and
$\eta$ are both
{$\mathcal{N}(0;[\underline{\sigma}^{2},\overline{\sigma}^{2}])$
distributed.}

We need to check that the functional
$\widetilde{\mathbb{E}}[\cdot]:C_{l.Lip}(\mathbb{R})\mapsto
\mathbb{R}$ forms a sublinear expectation, i.e., (a)-(d) of
Definition \ref{Def-1} are satisfied. Indeed, (a) is simply the
consequence of comparison theorem, or the maximum principle of
viscosity solution (see Appendix). It is also easy to check that,
when $\varphi \equiv c$, then the unique solution of
(\ref{eq-G-heat}) is also $u\equiv c$; hence (b) holds true. (d)
also holds since $u^{\lambda \varphi}=\lambda u^{\varphi}$, $\lambda
\geq0$. The sub-additivity (c) will be proved in Appendix.

\section{{Central Limit Theorem}}

{Our main result is: }

\begin{theorem}
{(Central Limit Theorem) Let a sequence $\left \{  X_{i}\right \}
_{i=1}^{\infty}$ in $\mathcal{H}$ be identically distributed with
each others. We also assume that, each $X_{n+1}$ is independent (or
weakly independent) to $(X_{1},\cdots,X_{n})$ for $n=1,2,\cdots$. We
assume furthermore that
\[
\mathbb{\hat{E}}[X_{1}]=\mathbb{\hat{E}}[-X_{1}]=0\text{,\  \ }\mathbb{\hat{E}%
}[X_{1}^{2}]=\overline{\sigma}^{2},\ -\mathbb{\hat{E}}[-X_{1}^{2}%
]=\underline{\sigma}^{2},
\]
for some fixed numbers $0<\underline{\sigma}\leq
\overline{\sigma}<\infty$. Then the sequence $\left \{
S_{n}/\sqrt{n}\right \}  _{n=1}^{\infty}$, of the sum
}$S_{n}=X_{1}+\cdots+X_{n}$, {converges in law to
$\mathcal{N}(0;[\underline {\sigma}^{2},\overline{\sigma}^{2}])$:
\begin{equation}
\lim_{n\rightarrow \infty}\mathbb{\hat{E}}[\varphi(\frac{S_{n}}{\sqrt{n}%
})]=\widetilde{\mathbb{E}}[\varphi(\xi)],\  \forall \varphi \in lip_{b}%
(\mathbb{R}),\  \  \  \  \label{CLT}%
\end{equation}
where $\xi \  \sim \  \mathcal{N}(0;[\underline{\sigma}^{2},\overline{\sigma}%
^{2}])$. }
\end{theorem}

\begin{proof}
{For a small but fixed $h>0$, let $V$ be the unique viscosity solution of%
\begin{equation}
\partial_{t}V+G(\partial_{xx}^{2}V)=0,\ (t,x)\in \lbrack0,1+h]\times
\mathbb{R}\text{,}\  \ V|_{t=1+h}=\varphi. \label{eq-V}%
\end{equation}
We have, according to the definition of $G$-normal distribution%
\[
V(t,x)=\widetilde{\mathbb{E}}[\varphi(x+\sqrt{1+h-t}\xi)].
\]
Particularly,
\begin{equation}
V(h,0)=\widetilde{\mathbb{E}}[\varphi(\xi)],\  \
V(1+h,x)=\varphi(x).
\label{equ-0}%
\end{equation}
Since (\ref{eq-V}) is a uniformly parabolic PDE and $G$ is a convex
function, thus, by the interior regularity of $V$ (see Wang
\cite{WangL}, Theorem 4.13) we have
\[
\left \Vert V\right \Vert _{C^{1+\alpha/2,2+\alpha}([0,1]\times \mathbb{R}%
)}<\infty,\  \text{for some }\alpha \in(0,1).
\]
We set $\delta=\frac{1}{n}$ and $S_{0}=0$. Then \
\begin{align*}
&
V(1,\sqrt{\delta}S_{n})-V(0,0)=\sum_{i=0}^{n-1}\{V((i+1)\delta,\sqrt
{\delta}S_{i+1})-V(i\delta,\sqrt{\delta}S_{i})\} \\
&  =\sum_{i=0}^{n-1}\left \{  [V((i+1)\delta,\sqrt{\delta}S_{i+1}%
)-V(i\delta,\sqrt{\delta}S_{i+1})]+[V(i\delta,\sqrt{\delta}S_{i+1}%
)-V(i\delta,\sqrt{\delta}S_{i})]\right \} \\
&  =\sum_{i=0}^{n-1}\left \{  I_{\delta}^{i}+J_{\delta}^{i}\right \}
\end{align*}
with, by Taylor's expansion,%
\[
J_{\delta}^{i}=\partial_{t}V(i\delta,\sqrt{\delta}S_{i})\delta+\frac{1}%
{2}\partial_{xx}^{2}V(i\delta,\sqrt{\delta}S_{i})X_{i+1}^{2}\delta
+\partial_{x}V(i\delta,\sqrt{\delta}S_{i})X_{i+1}\sqrt{\delta}%
\]%
\begin{align*}
&
I_{\delta}^{i}=\int_{0}^{1}[\partial_{t}V((i+\beta)\delta,\sqrt{\delta
}S_{i+1})-\partial_{t}V(i\delta,\sqrt{\delta}S_{i+1})]d\beta \delta \\
&
+[\partial_{t}V(i\delta,\sqrt{\delta}S_{i+1})-\partial_{t}V(i\delta
,\sqrt{\delta}S_{i})]\delta \\
&  +\int_{0}^{1}\int_{0}^{1}[\partial_{xx}^{2}V(i\delta,\sqrt{\delta}%
S_{i}+\gamma \beta X_{i+1}\sqrt{\delta})-\partial_{xx}^{2}V(i\delta
,\sqrt{\delta}S_{i})]\gamma d\beta d\gamma X_{i+1}^{2}\delta.
\end{align*}
Thus
\begin{equation}
\mathbb{\hat{E}}[\sum_{i=0}^{n-1}J_{\delta}^{i}]-\mathbb{\hat{E}}[-\sum
_{i=0}^{n-1}I_{\delta}^{i}]\leq \mathbb{\hat{E}}[V(1,\sqrt{\delta}%
S_{n})]-V(0,0)\leq \mathbb{\hat{E}}[\sum_{i=0}^{n-1}J_{\delta}^{i}%
]+\mathbb{\hat{E}}[\sum_{i=0}^{n-1}I_{\delta}^{i}] \label{major1}%
\end{equation}
We now prove that
$\mathbb{\hat{E}}[\sum_{i=0}^{n-1}J_{\delta}^{i}]=0$. Indeed, the
3rd term of $J_{\delta}^{i}$ has mean-certainty:
\[
\mathbb{\hat{E}}[\partial_{x}V(i\delta,\sqrt{\delta}S_{i})X_{i+1}\sqrt{\delta
}]=\mathbb{\hat{E}}[-\partial_{x}V(i\delta,\sqrt{\delta}S_{i})X_{i+1}%
\sqrt{\delta}]=0.
\]
For the second term, we have
\[
\mathbb{\hat{E}}[\frac{1}{2}\partial_{xx}^{2}V(i\delta,\sqrt{\delta}%
S_{i})X_{i+1}^{2}\delta]=\mathbb{\hat{E}}[G(\partial_{xx}^{2}V(i\delta
,\sqrt{\delta}S_{i}))\delta].
\]
We then combine the above two equalities with
$\partial_{t}V+G(\partial_{xx}^{2}V)=0$ as well as the independence
of $X_{i+1}$ to $(X_{1},\cdots,X_{i})$, it follows that
\[
\mathbb{\hat{E}}[\sum_{i=0}^{n-1}J_{\delta}^{i}]=\mathbb{\hat{E}}[\sum
_{i=0}^{n-2}J_{\delta}^{i}]=\cdots=0.
\]
Thus (\ref{major1}) can be rewritten as%
\[
-\mathbb{\hat{E}}[-\sum_{i=0}^{n-1}I_{\delta}^{i}]\leq \mathbb{\hat{E}%
}[V(1,\sqrt{\delta}S_{n})]-V(0,0)\leq \mathbb{\hat{E}}[\sum_{i=0}%
^{n-1}I_{\delta}^{i}].
\]
But since both $\partial_{t}V$ and $\partial_{xx}^{2}V$ are
uniformly $\alpha $-h\"{o}lder continuous in $x$ and\
$\frac{\alpha}{2}$-h\"{o}lder continuous in $t$ on $[0,1]\times
\mathbb{R}$, we then have $|I_{\delta}^{i}|\leq
C\delta^{1+\alpha/2}[1+|X_{i+1}|+|X_{i+1}|^{2+\alpha}]$. It follows
that
\[
\mathbb{\hat{E}}[|I_{\delta}^{i}|]\leq C\delta^{1+\alpha/2}(1+\mathbb{\hat{E}%
}[|X_{1}|^{\alpha}]+\mathbb{\hat{E}}[|X_{1}|^{2+\alpha}]).
\]
Thus%
\begin{align*}
-C(\frac{1}{n})^{\alpha/2}(1+\mathbb{\hat{E}}[|X_{1}|^{\alpha}+|X_{1}%
|^{2+\alpha}])  &  \leq \mathbb{\hat{E}}[V(1,\sqrt{\delta}S_{n})]-V(0,0)\\
&  \leq C(\frac{1}{n})^{\alpha/2}(1+\mathbb{\hat{E}}[|X_{1}|^{\alpha}%
+|X_{1}|^{2+\alpha}]).
\end{align*}
As $n\rightarrow \infty$ we have
\begin{equation}
\lim_{n\rightarrow
\infty}\mathbb{\hat{E}}[V(1,\sqrt{\delta}S_{n})]=V(0,0).
\label{equ-h}%
\end{equation}
On the other hand, for each $t,t^{\prime}\in \lbrack0,1+h]$ and
$x\in \mathbb{R}$,
\begin{align*}
|V(t,x)-V(t^{\prime},x)|  &
=|\widetilde{\mathbb{E}}[|\varphi(x+\sqrt
{1+h-t}\xi)-\varphi(\sqrt{1+h-t^{\prime}}\xi)|]\\
&  \leq
k_{\varphi}|\sqrt{1+h-t}-\sqrt{1+h-t^{\prime}}|\times\widetilde
{\mathbb{E}}[|\xi|]\\
&  \leq C\sqrt{|t-t^{\prime}|},
\end{align*}
where $k_{\varphi}$ denotes the Lipschitz constant of $\varphi$.
Thus $|V(0,0)-V(h,0)|\leq C\sqrt{h}$ and, by (\ref{equ-0}),
\begin{align*}
&
|\mathbb{\hat{E}}[V(1,\sqrt{\delta}S_{n})]-\mathbb{\hat{E}}[\varphi
(\sqrt{\delta}S_{n})]|\\
&
=|\mathbb{\hat{E}}[V(1,\sqrt{\delta}S_{n})]-\mathbb{\hat{E}}[V(1+h,\sqrt
{\delta}S_{n})]|\leq C\sqrt{h}.
\end{align*}
It follows form (\ref{equ-h}) and (\ref{equ-0}) that
\[
\limsup_{n\rightarrow \infty}|\mathbb{\hat{E}}[\varphi(\frac{S_{n}}{\sqrt{n}%
})]-\widetilde{\mathbb{E}}[\varphi(\xi)]|\leq2C\sqrt{h}.
\]
Since $h$ can be arbitrarily small we thus have%
\[
\lim_{n\rightarrow \infty}\mathbb{\hat{E}}[\varphi(\frac{S_{n}}{\sqrt{n}%
})]=\widetilde{\mathbb{E}}[\varphi(\xi)].
\]
}
\end{proof}

\begin{corollary}
{The convergence (\ref{CLT}) holds for the case where $\varphi$ is a
bounded and uniformly continuous function. }
\end{corollary}

\begin{proof}
{We can find a sequence $\left \{  \varphi_{k}\right \}
_{k=1}^{\infty}$ in $lip_{b}(\mathbb{R})$ such that
$\varphi_{k}\rightarrow \varphi$ uniformly on $\mathbb{R}$. By
\begin{align*}
|\mathbb{\hat{E}}[\varphi(\frac{S_{n}}{\sqrt{n}})]-\widetilde{\mathbb{E}%
}[\varphi(\xi)]|  &  \leq|\mathbb{\hat{E}}[\varphi(\frac{S_{n}}{\sqrt{n}%
})]-\mathbb{\hat{E}}[\varphi_{k}(\frac{S_{n}}{\sqrt{n}})]|\\
&
+|\widetilde{\mathbb{E}}[\varphi(\xi)]-\widetilde{\mathbb{E}}[\varphi
_{k}(\xi)]|+|\mathbb{\hat{E}}[\varphi_{k}(\frac{S_{n}}{\sqrt{n}}%
)]-\widetilde{\mathbb{E}}[\varphi_{k}(\xi)]|.
\end{align*}
We can easily check that (\ref{CLT}) holds. }
\end{proof}


\chapter{$G$-Brownian Motion: 1-Dimensional Case}

\section{$1$-dimensional $G$-Brownian motion}

In this lecture I will introduce the notion of $G$-Brownian motion
related to the $G$-normal distribution in a space of a sublinear
expectation. We first give the definition of the $G$-Brownian
motion.

\begin{definition}
\label{Def-3} A process $\{B_{t}(\omega)\}_{t\geq 0}$ in a sublinear
expectation space $(\Omega,\mathcal{H},\mathbb{\hat{E}})$ is called
a $G$\textbf{--Brownian motion} if for each $n\in \mathbb{N}$ and
$0\leq t_{1},\cdots ,t_{n}<\infty$,$\ B_{t_{1}},\cdots,B_{t_{n}}\in
\mathcal{H}$ and the following properties are satisfied:
\newline\textsl{(i)} $B_{0}(\omega)=0$;\newline\textsl{(ii)} For
each $t,s\geq0$, the increment $B_{t+s}-B_{t}$ is $\mathcal{N}
(0;[\underline{\sigma}^{2}s,\overline{\sigma}^{2}s])$-distributed
and is independent to $(B_{t_{1}},B_{t_{2}},\cdots,B_{t_{n}})$, for
each $n\in \mathbb{N}$ and $0\leq t_{1}\leq\cdots\leq t_{n}\leq t$.
\end{definition}

\begin{remark}
The letter $G$ indicates that the process $B$ is characterized by
it's `generating function' $G$ defined by
\[
G(\alpha):=\mathbb{\hat{E}}[\alpha B_1^2],\ \ \ \ \alpha\in
\mathbb{R}.
\]
 We can prove that, for each $\lambda>0$, $(\lambda^{-\frac{1}{2}}B_{\lambda
t})_{t\geq0}$ is also a $G$-Brownian motion. For each $t_{0}>0$.
$(B_{t+t_{0}}-B_{t_{0}})_{t\geq0}$ is also a $G$-Brownian motion.
This is the scaling property of $G$-Brownian motion, which is the
same as that of the usual Brownian motion.

 In
this course we assume without loss of generality that
$\overline{\sigma}=1$ and $\underline{\sigma}=\sigma\leq 1$.
\end{remark}

\begin{theorem}
Let $(\tilde{B}_{t})_{t\geq0}$ be a process defined in a sublinear
expectation space $(\Omega ,\widetilde{\mathcal{H}}$,
$\widetilde{\mathbb{E}}${$)$ such that
\newline\textsl{(i)}
$\tilde{B}_{0}=0$;\newline\textsl{(ii)} For each $t,s\geq0$, the difference $\tilde{B}%
${$_{t+s}-\tilde{B}_{t}$ and }$\tilde{B}${$_{s}$ are identically
distributed and independent to
$(\tilde{B}_{t_{1}},\tilde{B}_{t_{2}},\cdots,\tilde {B}_{t_{n}})$,
for each $n\in \mathbb{N}$ and $0\leq t_{1}\leq\cdots\leq t_{n}\leq
t$.
\newline\textsl{(iii)} }$B_{0}=0$, $\widetilde{\mathbb{E}}${$[\tilde{B}_{t}]=$%
}$\widetilde{\mathbb{E}}${$[-\tilde{B}_{t}]=0$ and $\lim_{t\downarrow0}$%
}$\widetilde{\mathbb{E}}${$[|\tilde{B}_{t}|^{3}]t^{-1}=0$. \newline
Then }$\tilde{B}${ is a
$G_{\underline{\sigma},\overline{\sigma}}$-Brownian motion with
$\overline{\sigma}^{2}=$}$\widetilde{\mathbb{E}}[\tilde{B}_{1}^{2}]$
and
$\underline{\sigma}^{2}=-$}$\widetilde{\mathbb{E}}[-\tilde{B}_{1}^{2}]$.
\end{theorem}

{ {\medskip \noindent \textbf{Proof. }}We need only to prove that $\tilde{B}%
_{t}$ is  $\mathcal{N}(0;[\underline{\sigma}^{2}t,\overline{\sigma}^{2}%
t])$-distributed. We first prove that
\[
\widetilde{\mathbb{E}}[\tilde{B}_{t}^{2}]=\overline{\sigma}^{2}t\text{
\ and
}-\widetilde{\mathbb{E}}[-\tilde{B}_{t}^{2}]=\underline{\sigma}^{2}t.
\]
We set $b(t):=\widetilde{\mathbb{E}}${$[\tilde{B}_{t}^{2}]$. Then
}$b(0)=0$
and $\lim_{t\downarrow0}b(t)\leq \widetilde{\mathbb{E}}${$[|\tilde{B}_{t}%
|^{3}]^{\frac{2}{3}}\rightarrow0$.} {Since for each $t,s\geq0$,%
\begin{align*}
b(t+s) &  =\widetilde{\mathbb{E}}[(\tilde{B}_{t+s})^{2}]=\widetilde
{\mathbb{E}}[(\tilde{B}_{t+s}-\tilde{B}_{s}+\tilde{B}_{s})^{2}]=\widetilde
{\mathbb{E}}[(\tilde{B}_{t+s}-\tilde{B}_{s})^{2}+\tilde{B}_{s}^{2}]\\
&  =b(t)+b(s).
\end{align*}
Thus }$b(t)$ is linear and uniformly continuous in $t$; hence there
exists a
constant $\overline{\sigma}\geq0$ such that $\widetilde{\mathbb{E}}$%
{$[\tilde{B}_{t}^{2}]=\overline{\sigma}^{2}t$. Similarly, there
exists a
constant $\underline{\sigma}\in \lbrack0,\overline{\sigma}]$ such that $-$%
}$\widetilde{\mathbb{E}}${$[-\tilde{B}_{t}^{2}]=\underline{\sigma}^{2}t$.
We
have }$\overline{\sigma}^{2}=${ }$\widetilde{\mathbb{E}}$$[\tilde{B}_{1}%
^{2}]\geq-\widetilde{\mathbb{E}}${$[-\tilde{B}_{1}^{2}]=$}$\underline{\sigma
}^{2}$.{ } }

{ {We now prove that }$\tilde{B}_{t}$ is
$\mathcal{N}(0;[\underline{\sigma
}^{2}t,\overline{\sigma}^{2}t])$-distributed. We just need to prove
that, for each fixed $\varphi \in C_{l.Lip}(\mathbb{R)}$, the
function
\[
u(t,x):=\widetilde{\mathbb{E}}[\varphi(x+\tilde{B}_{t})],\  \
(t,x)\in
\lbrack0,\infty)\times \mathbb{R}%
\]
is the viscosity solution of the following $G_{\underline{\sigma}%
,\overline{\sigma}}$-heat equation}%
\begin{equation}
\partial_{t}u-G(\partial_{xx}^{2}u)=0,\  \text{for }t>0;\  \ u|_{t=0}%
=\varphi.\label{G-heat-BM}%
\end{equation}
{with }${G(a)=G_{\underline{\sigma},\overline{\sigma}}(a)=\frac{1}%
{2}(\overline{\sigma}^{2}a^{+}-\underline{\sigma}^{2}a^{-})}$. Thus
for each
$\mathcal{N}(0;[\underline{\sigma}^{2},\overline{\sigma}^{2}])$-distributed
$\xi$, we have { }%
\[
\widetilde{\mathbb{E}}[\varphi(\tilde{B}_{t})]=\widetilde{\mathbb{E}}%
[\varphi(\sqrt{t}\xi)],\  \  \forall \varphi \in
C_{l.Lip}(\mathbb{R}).
\]
Thus $\tilde{B}_{t}\sim \sqrt{t}\xi \sim \mathcal{N}(0;[\underline{\sigma}%
^{2}t,\overline{\sigma}^{2}t])$.

{ {We first prove that }$u$ is locally Lipschitz in $x$ and locally
$\frac {1}{2}$-H\"{o}lder in $t$. In fact, {for each fixed }$t$,{
}$u(t,\cdot)\in ${$C_{l.Lip}(\mathbb{R)}$ }since
\begin{align*}
|\widetilde{\mathbb{E}}[\varphi(x+\tilde{B}_{t})]-\widetilde{\mathbb{E}%
}[\varphi(y+\tilde{B}_{t})]|  & \leq\widetilde{\mathbb{E}}[|\varphi
(x+\tilde{B}_{t})-\varphi(y+\tilde{B}_{t})|]\\
&  \leq|\widetilde{\mathbb{E}}[C(1+|x|^{m}+|y|^{m}+|\tilde{B}_{t}|^m)|x-y|]\\
&  \leq C_1(1+|x|^{m}+|y|^{m})|x-y|.
\end{align*}
{For each $\delta \in \lbrack0,t]$, since
$\tilde{B}_{t}-\tilde{B}_{\delta}$ is independent to
$\tilde{B}_{\delta}$, w}e also have
\begin{align*}
u(t,x)  &
=\widetilde{\mathbb{E}}[\mathbb{\varphi(}x+\tilde{B}_{\delta
}+(\tilde{B}_{t}-\tilde{B}_{\delta})]\\
&  =\widetilde{\mathbb{E}}[\widetilde{\mathbb{E}}[\varphi(y+(\tilde{B}%
_{t}-\tilde{B}_{\delta}))]_{y=x+\tilde{B}_{\delta}}],
\end{align*}
hence%
\begin{equation}
u(t,x)=\widetilde{\mathbb{E}}[u(t-\delta,x+\tilde{B}_{\delta})]. \label{Dyna}%
\end{equation}
{Thus}%
\begin{align*}
|u(t,x)-u(t-\delta,x)|  &
=|\widetilde{\mathbb{E}}[u(t-\delta,x+\tilde
{B}_{\delta})-u(t-\delta,x)]|\\
&  \leq
\widetilde{\mathbb{E}}[|u(t-\delta,x+\tilde{B}_{\delta})-u(t-\delta
,x)|]\\
&  \leq
\widetilde{\mathbb{E}}[C(1+|x|^{m}+|\tilde{B}_{\delta}|^{m})|\tilde
{B}_{\delta}|]\\
&  \leq C_{1}(1+|x|^{m})\overline{\sigma}\delta^{\frac{1}{2}}.
\end{align*}
}

{ To prove that $u$ is a viscosity solution of {(\ref{G-heat-BM})},
{we fix a
time-space $(t,x)\in(0,\infty)\times \mathbb{R}$} and{ let $v\in C_{b}%
^{2,3}([0,\infty)\times \mathbb{R})$ be such that $v\geq u$ and $v(t,x)=u(t,x)$%
. From (\ref{Dyna}) we have
\[
v(t,x)=\widetilde{\mathbb{E}}[u(t-\delta,x+\tilde{B}_{\delta})]\leq
\widetilde{\mathbb{E}}[v(t-\delta,x+\tilde{B}_{\delta})]
\]
Therefore by Taylor's expansion,
\begin{align*}
0 &  \leq \widetilde{\mathbb{E}}[v(t-\delta,x+\tilde{B}_{\delta})-v(t,x)]\\
&
=\widetilde{\mathbb{E}}[v(t-\delta,x+\tilde{B}_{\delta})-v(t,x+\tilde
{B}_{\delta})+(v(t,x+\tilde{B}_{\delta})-v(t,x))]\\
&  =\widetilde{\mathbb{E}}[-\partial_{t}v(t,x)\delta+\partial_{x}%
v(t,x)\tilde{B}_{\delta}+\frac{1}{2}\partial_{xx}^{2}v(t,x)\tilde{B}_{\delta
}^{2}+I_{\delta}]\\
&
\leq-\partial_{t}v(t,x)\delta+\widetilde{\mathbb{E}}[\frac{1}{2}\partial_{xx}^{2}v(t,x)\tilde
{B}_{\delta}^{2}]-\widetilde{\mathbb{E}}[I_{\delta}]\\
&
=-\partial_{t}v(t,x)\delta+G(\partial_{xx}^{2}(t,x))\delta-\widetilde
{\mathbb{E}}[I_{\delta}].
\end{align*}
where}%
\begin{align*}
I_{\delta} &  =\int_{0}^{1}-[\partial_{t}v(t-\beta
\delta,x+\tilde{B}_{\delta
})-\partial_{t}v(t,x)]\delta d\beta \\
&  +\int_{0}^{1}\int_{0}^{1}[\partial_{xx}^{2}v(t,x+\alpha \beta
\tilde {B}_{\delta})-\partial_{xx}^{2}v(t,x)]\alpha d\beta d\alpha
\tilde {B}_{\delta}^2
\end{align*}
With the assumption \textsl{(iii)} we can check that $\lim_{\delta
\downarrow 0}\widetilde{\mathbb{E}}[|I_{\delta}|]\delta^{-1}=0$;
from which we get
$\partial_{t}v(t,x)-G(\partial_{xx}^{2}(t,x))\leq0$; hence $u$ is a
viscosity supersolution of {(\ref{G-heat-BM}). We can analogously
prove that }$u$ is a
viscosity subsolution. But by the definition of $G$-normal distribution%
\[
\tilde{B}_{t}\  \sim \  \sqrt{t}\xi \  \sim \ {\mathcal{N}(0;[\underline{\sigma}%
^{2}t,\overline{\sigma}^{2}t]).}%
\]
Thus $(\tilde{B}_{t})_{t\geq0}$ is a $G$-Brownian motion. }

\subsection{{Existence of{ $G$-Brownian motion}}}

{ {In the rest of this course, we denote by
$\Omega=C_{0}(\mathbb{R}^{+})$ the space of all $\mathbb{R}$--valued
continuous paths $(\omega_{t})_{t\in \mathbb{R}^{+}}$ with
$\omega_{0}=0$, equipped with the distance
\[
\rho(\omega^{1},\omega^{2}):=\sum_{i=1}^{\infty}2^{-i}[(\max_{t\in
\lbrack 0,i]}|\omega_{t}^{1}-\omega_{t}^{2}|)\wedge1].
\]
For each fixed $T\geq0$, we consider the following space of random variables:}%
\begin{align*}
(\mathcal{H}_T^0)=L_{ip}^{0}(\mathcal{F}_{T})  &  :=\{X(\omega)=\varphi(\omega_{t_{1}}%
,\cdots,\omega_{t_{m}}),\  \forall m\geq1,\\
\  &  \  \  \  \  \; \  \  \  \  \  \  \  \ t_{1},\cdots,t_{m}\in
\lbrack0,T],\forall \varphi \in C_{l.Lip}(\mathbb{R}^{m})\}.
\end{align*}
 It is clear that $\mathcal{H}_t^0\subseteq L_{ip}^{0}%
(\mathcal{F}_{T})$, for $t\leq T$. We also denote%
\[
(\mathcal{H}^0)=L_{ip}^{0}(\mathcal{F}):=%
{\displaystyle \bigcup \limits_{n=1}^{\infty}}
L_{ip}^{0}(\mathcal{F}_{n}).
\]
 }

\begin{remark}
{ { Obviously $C_{l.Lip}(\mathbb{R}^{m})$ and then $L_{ip}%
^{0}(\mathcal{F}_{T})$, $L_{ip}^{0}(\mathcal{F})$ are vector
lattices. Moreover, since $\varphi,\psi \in
C_{l.Lip}(\mathbb{R}^{m})$ implies
$\varphi \cdot \psi \in C_{l.Lip}(\mathbb{R}^{m})$ thus $X$, $Y\in L_{ip}%
^{0}(\mathcal{F}_{T})$ implies $X\cdot Y\in
L_{ip}^{0}(\mathcal{F}_{T})$. } }
\end{remark}

We will consider the canonical space and set
$B_{t}(\omega)=\omega_{t}$, $t\in \lbrack0,\infty)$, for $\omega \in
\Omega$.

For each fixed $T\in \lbrack0,\infty)$, we set
\[
L_{ip}(\mathcal{F}_{T}):=\{
\varphi(B_{t_{1}},\cdots,B_{t_{n}}):0\leq t_{1},\cdots,t_{n}\leq T,\
\varphi \in C_{l.Lip}(\mathbb{R}^{n}),\ n\in
\mathbb{N}\}.\text{ }%
\]
In particular, for each $t\in \lbrack0,\infty)$, $B_{t}\in L_{ip}%
(\mathcal{F}_{t})$. {We are given a sublinear function $G(a)=G_{
\sigma,1}(a)=\frac{1}{2}(a^{+}-\sigma^{2}a^{-})$,
$a\in \mathbb{R}$. Let }$\xi$ be a $G$-normal distributed, or {$\mathcal{N}%
(0;[\sigma^{2},1])$-distributed,} random variable in a sublinear
expectation space
$(\widetilde{\Omega},\widetilde{\mathcal{H}},\widetilde{\mathbb{E}})${.}

We now introduce a sublinear expectation $\mathbb{\hat{E}}$ defined
on $\mathcal{H}_{T}^{0}=L_{ip}^{0}(\mathcal{F}_{T})$, as well as on
$\mathcal{H}^{0}=L_{ip}^{0}(\mathcal{F})$, via the following
procedure: For each  $X\in \mathcal{H}_T^0$ with
\[
X=\varphi(B_{t_{1}}-B_{t_{0}},B_{t_{2}}-B_{t_{1}},\cdots,B_{t_{m}}-B_{t_{m-1}%
})
\]
{ for some {$\varphi \in C_{l.Lip}(\mathbb{R}^{m})$ and $0=t_{0}<
t_{1}<\cdots<t_{m}<\infty$, we set%
\[
\mathbb{\hat{E}}[\varphi(B_{t_{1}}-B_{t_{0}},B_{t_{2}}-B_{t_{1}}%
,\cdots,B_{t_{m}}-B_{t_{m-1}})]
\]
\[
=\widetilde{\mathbb{E}}[\varphi(\sqrt{t_{1}-t_{0}}\xi_{1},\cdots,\sqrt
{t_{m}-t_{m-1}}\xi_{m})],
\]
where $(\xi_{1},\cdots,\xi_{n})$ is an $m$-dimensional $G$-normal
distributed random vector in a sublinear expectation space
$(\widetilde{\Omega
},\widetilde{\mathcal{H}},\widetilde{\mathbb{E}})$ such that $\xi
_{i}\sim\mathcal{N}(0;[\sigma^{2},1])$ and such that $\xi_{i+1}$ is
independent to $(\xi_{1},\cdots,\xi_{i})$ for each $i=1,\cdots,m-1$.

{{The related conditional expectation of $X=\varphi(B_{t_{1}},B_{t_{2}%
}-B_{t_{1}},\cdots,B_{t_{m}}-B_{t_{m-1}})$ under
$\mathcal{H}_{t_{j}}$ is
defined by%
\begin{align}
\mathbb{\hat{E}}[X|\mathcal{H}_{t_{j}}]  &
=\mathbb{\hat{E}}[\varphi
(B_{t_{1}},B_{t_{2}}-B_{t_{1}},\cdots,B_{t_{m}}-B_{t_{m-1}})|\mathcal{H}%
_{t_{j}}]\label{Condition}\\
&  =\psi(B_{t_{1}},\cdots,B_{t_{j}}-B_{t_{j-1}})\nonumber
\end{align}
where}}%
\[
\psi(x_{1},\cdots,x_{j})=\widetilde{\mathbb{E}}[\varphi(x_{1},\cdots
,x_{j},\sqrt{t_{j+1}-t_{j}}\xi_{j+1},\cdots,\sqrt{t_{m}-t_{m-1}}\xi_{m})]
\]
{{It is easy to check that $\mathbb{\hat{E}}[\cdot]$ consistently
defines a sublinear expectation on as well as on
$L_{ip}^{0}(\mathcal{F})$ satisfying (a)--(d) of Definition
\ref{Def-1}.} }

\begin{definition}
{ { The expectation $\mathbb{\hat{E}}[\cdot]:L_{ip}^{0}(\mathcal{F}%
)\mapsto \mathbb{R}$ defined through the above procedure is called $G$%
\textbf{--expectation}. The corresponding canonical process
$(B_{t})_{t\geq0}$ in the sublinear expectation space
$(\Omega,\mathcal{H},\mathbb{\hat{E}})$ is called  a $G$--Brownian
motion. } }
\end{definition}

\begin{proposition}
 \label{Prop-1-7-1}{We list the properties of $\mathbb{\hat{E}}%
[\cdot|\mathcal{H}_{t}]$ that hold for each $X,Y\in \mathcal{H}^0=L_{ip}%
^{0}(\mathcal{F})$:}\newline{\textbf{(a') }{If $X\geq Y$, then
$\mathbb{\hat
{E}}[X|\mathcal{H}_{t}]\geq \mathbb{\hat{E}}[Y|\mathcal{H}_{t}]$.\newline%
}\textbf{(b') }$\mathbb{\hat{E}}[\eta|\mathcal{H}_{t}]=\eta$, for each }%
$t\in \lbrack0,\infty)$ and {{$\eta \in \mathcal{H}_t^0$%
.\newline \textbf{(c') }$\mathbb{\hat{E}}[X|\mathcal{H}_{t}]-\mathbb{\hat{E}%
}[Y|\mathcal{H}_{t}]\leq \mathbb{\hat{E}}[X-Y|\mathcal{H}_{t}].$\newline%
{\textbf{(d')} $\mathbb{\hat{E}}[\eta X|\mathcal{H}_{t}]=\eta^{+}%
\mathbb{\hat{E}}[X|\mathcal{H}_{t}]+\eta^{-}\mathbb{\hat{E}}[-X|\mathcal{H}%
_{t}]$, for each $\eta \in{\mathcal{H}_t^0}.$}\textbf{ }}}\newline
We also have%
\[
{{\mathbb{\hat{E}}[\mathbb{\hat{E}}[X|\mathcal{H}_{t}]|\mathcal{H}%
_{s}]=\mathbb{\hat{E}}[X|\mathcal{H}_{t\wedge s}],\ }}\text{ in
particular }
\  \ \mathbb{\hat{E}}[\mathbb{\hat{E}}[X|\mathcal{H}_{t}]]=\mathbb{\hat{E}%
}[X].%
\]
For each $X\in{L_{ip}^{0}}(\mathcal{F}_{T}^{t})$, $\mathbb{\hat{E}%
}[X|\mathcal{H}_{t}]=\mathbb{\hat{E}}[X]$, where $L_{ip}%
^{0}(\mathcal{F}_{T}^{t})=\mathcal{H}_{T}^{t}$ is the linear space
of random
variables of the form %
\begin{align*}
&  {\varphi(B_{t_{2}}-B_{t_{1}},B_{t_{3}}-B_{t_{2}},\cdots,B_{t_{m+1}%
}-B_{t_{m}}),}\\
\  &  \  \  m=1,2,\cdots,\  \varphi \in C_{l.Lip}(\mathbb{R}^{m}),\ t_{1}%
,\cdots,t_{m},t_{m+1},\in \lbrack t,\infty).%
\end{align*}
\end{proposition}

\begin{remark}
(b') and (c') imply:%
\[
{{\mathbb{\hat{E}}[X+\eta|\mathcal{H}_{t}]=\mathbb{\hat{E}}[X|\mathcal{H}%
_{t}]+\eta.}}%
\]
Moreover, if $Y\in L_{ip}^{0}(\mathcal{F})$ satisfies ${{\mathbb{\hat{E}%
}[Y|\mathcal{H}_{t}]=-\mathbb{\hat{E}}[-Y|\mathcal{H}_{t}]}}$ then%
\[
{{\mathbb{\hat{E}}[X+Y|\mathcal{H}_{t}]=\mathbb{\hat{E}}[X{|\mathcal{H}_{t}%
]}+{\mathbb{\hat{E}}[Y}|\mathcal{H}_{t}].}}%
\]

\end{remark}

\begin{example}
{ { \label{Exam-1}For each $s<t$, we have
$\mathbb{\hat{E}}[B_{t}-B_{s}|\mathcal{H}_{s}]=0$ and, for
$n=1,2,\cdots,$
\[
\mathbb{\hat{E}}[|B_{t}-B_{s}|^{n}|\mathcal{H}_{s}]=\mathbb{\hat{E}}%
[|B_{t-s}|^{n}]=\frac{1}{\sqrt{2\pi(t-s)}}\int_{-\infty}^{\infty}|x|^{n}%
\exp(-\frac{x^{2}}{2(t-s)})dx.
\]
But we have%
\[
\mathbb{\hat{E}}[-|B_{t}-B_{s}|^{n}|\mathcal{H}_{s}]=\mathbb{\hat{E}%
}[-|B_{t-s}|^{n}]=-\sigma^{n}\mathbb{\hat{E}}[|B_{t-s}|^{n}].
\]
Exactly as in classical cases, we have
\begin{align*}
\mathbb{\hat{E}}[(B_{t}-B_{s})^{2}|\mathcal{H}_{s}]  &
=t-s,\  \  \  \mathbb{\hat{E}}[(B_{t}-B_{s})^{4}|\mathcal{H}_{s}]=3(t-s)^{2},\\
\mathbb{\hat{E}}[(B_{t}-B_{s})^{6}|\mathcal{H}_{s}]  &  =15(t-s)^{3}%
,\  \  \mathbb{\hat{E}}[(B_{t}-B_{s})^{8}|\mathcal{H}_{s}]=105(t-s)^{4},\\
\mathbb{\hat{E}}[|B_{t}-B_{s}||\mathcal{H}_{s}]  &  =\frac{\sqrt{2(t-s)}%
}{\sqrt{\pi}},\  \  \mathbb{\hat{E}}[|B_{t}-B_{s}|^{3}|\mathcal{H}_{s}%
]=\frac{2\sqrt{2}(t-s)^{3/2}}{\sqrt{\pi}},\\
\mathbb{\hat{E}}[|B_{t}-B_{s}|^{5}|\mathcal{H}_{s}]  &
=8\frac{\sqrt {2}(t-s)^{5/2}}{\sqrt{\pi}}.
\end{align*}
}}
\end{example}

\begin{definition}
{ A process }$(M_{t})_{t\geq0}$ is called a $G$-martingale
(respectively, $G$-supermartingale; $G$-submartingale) if for each
$t\in \lbrack0,\infty)$,
$M_{t}\in L_{ip}^{0}(\Omega)$ and for each $s\in \lbrack0,t]$, we have%
\[
\mathbb{E}[M_{t}|\mathcal{H}_{s}]=M_{s},\  \  \ (\text{respectively,
\ \ }\leq M_{s};\  \  \  \geq M_{s}).
\]

\end{definition}

\begin{example}
$(B_{t})_{t\geq0}$ and $(-B_{t})_{t\geq0}$ are $G$-Martingale. $(B_{t}%
^{2})_{t\geq0}$ is a $G$-submartingale since%
\begin{align*}
\mathbb{\hat{E}}[B_{t}^{2}|\mathcal{H}_{s}]  &  =\mathbb{\hat{E}}[(B_{t}%
-B_{s})^{2}+B_{s}^{2}+2B_{s}(B_{t}-B_{s})|\mathcal{H}_{s}]\\
&  =\mathbb{\hat{E}}[(B_{t}-B_{s})^{2}]+B_{s}^{2}=t-s+B_{s}^{2}\geq B_{s}%
^{2}\text{.}%
\end{align*}

\end{example}

\newpage

\subsection{{Complete spaces of sublinear expectation}}

We briefly recall the notion of nonlinear expectations introduced in
\cite{Peng2005}. Following Daniell (see Daniell 1918 \cite{Daniell})
in his famous Daniell's integration, we begin with a vector lattice.

Let $\Omega$ be a given set and let $\mathcal{H}$ be a vector
lattice of real functions defined on $\Omega$ containing $1$,
namely, $\mathcal{H}$ is a linear space such that $1\in \mathcal{H}$
and that $X\in \mathcal{H}$ implies $|X|\in \mathcal{H}$.
$\mathcal{H}$ is a space of random variables. We assume the
functions on $\mathcal{H}$ are all bounded. Notice that
\[
a\wedge b=\min \{a,b\}=\frac{1}{2}(a+b-|a-b|),\  \ a\vee
b=-[(-a)\wedge(-b)].
\]
Thus $X$, $Y\in \mathcal{H}$ implies that $X\wedge Y$, $X\vee Y$, $X^{+}%
=X\vee0$ and $X^{-}=(-X)^{+}$ are all in $\mathcal{H}$.  }

{In this course we are mainly concerned with space $\mathcal{H}^0=L_{ip}%
^{0}(\mathcal{F})$. It satisfies%
\[
X_{1},\cdots,X_{n}\in \mathcal{H}\  \Longrightarrow \  \  \varphi(X_{1}%
,\cdots,X_{n})\in \mathcal{H},\  \  \forall \varphi \in
C_{l.Lip}(\mathbb{R}^{n}).
\]

\begin{remark}
{  { \label{Rem-2}For each fixed }$p\geq1$, we observe that $\mathcal{H}%
_{0}^{p}=\{X\in \mathcal{H}$, $\mathbb{\hat{E}}[|X|^{p}]=0\}$ is a
linear subspace of $\mathcal{H}$. To take $\mathcal{H}_{0}^{p}$ as
our null space, we introduce the quotient space
$\mathcal{H}/\mathcal{H}_{0}^{p}$. Observe that, for every
$\mathbf{\{}X\} \in \mathcal{H}/\mathcal{H}_{0}^{p}$ with a
representation $X\in \mathcal{H}$, we can define an expectation
$\mathbb{\hat {E}}\mathbf{[\{}X\}]:=\mathbb{\hat{E}}[X]$ which still
satisfies (a)--(d) of Definition \ref{Def-1}.  We set  $\left \Vert
X\right \Vert _{p}:=\mathbb{\hat{E}}[|X|^{p}]^{\frac{1}{p}}$.
$\left \Vert \cdot \right \Vert _{p}$ forms a Banach norm in $\mathcal{H}%
/\mathcal{H}_{0}^{p}$. }
\end{remark}

\begin{lemma}
{ {F{or $r>0$ and $1<p,q<\infty$ with $\frac{1}{p}+\frac{1}{q}=1$,
we have
\begin{align}
|a+b|^{r}  &  \leq \max \{1,2^{r-1}\}(|a|^{r}+|b|^{r}),\  \  \forall
a,b\in \mathbb{R};\label{ee4.3}\\
|ab|  &  \leq \frac{|a|^{p}}{p}+\frac{|b|^{q}}{q}. \label{ee4.4}%
\end{align}
} } }
\end{lemma}

\begin{proposition}
{ { For each $X,Y\in L_{ip}^{0}(\mathcal{F})$, we have{
\begin{align}
\mathbb{\hat{E}}[|X+Y|^{r}]  &  \leq C_{r}(\mathbb{\hat{E}}[|X|^{r}%
]+\mathbb{\hat{E}[}|Y|^{r}]),\label{ee4.5}\\
\mathbb{\hat{E}}[|XY|]  &  \leq \mathbb{\hat{E}}[|X|^{p}]^{1/p}\cdot
\mathbb{\hat{E}}[|Y|^{q}]^{1/q},\label{ee4.6}\\
\mathbb{\hat{E}}[|X+Y|^{p}]^{1/p}  &  \leq \mathbb{\hat{E}}[|X|^{p}%
]^{1/p}+\mathbb{\hat{E}}[|Y|^{p}]^{1/p}. \label{ee4.7}%
\end{align}
In particular, for $1\leq p<p^{\prime}$, we have $\mathbb{\hat{E}}%
[|X|^{p}]^{1/p}\leq
\mathbb{\hat{E}}[|X|^{p^{\prime}}]^{1/p^{\prime}}.$ } } }
\end{proposition}

\begin{proof}
   (\ref{ee4.5}) follows from (\ref{ee4.3}). We set
\[
\xi=\frac{X}{\mathbb{\hat{E}}[|X|^{p}]^{1/p}},\  \
\eta=\frac{Y}{\mathbb{\hat {E}}[|Y|^{q}]^{1/q}}.
\]
By (\ref{ee4.4}) we have%
\begin{align*}
\mathbb{\hat{E}}[|\xi \eta|]  &  \leq \mathbb{\hat{E}}[\frac{|\xi|^{p}}{p}%
+\frac{|\eta|^{q}}{q}]\leq
\mathbb{\hat{E}}[\frac{|\xi|^{p}}{p}]+\mathbb{\hat
{E}}[\frac{|\eta|^{q}}{q}]\\
&  =\frac{1}{p}+\frac{1}{q}=1.
\end{align*}
Thus (\ref{ee4.6}) follows. We now prove (\ref{ee4.7}):
\begin{align*}
\mathbb{\hat{E}}[|X+Y|^{p}]  &  =\mathbb{\hat{E}}[|X+Y|\cdot|X+Y|^{p-1}]\\
&  \leq
\mathbb{\hat{E}}[|X|\cdot|X+Y|^{p-1}]+\mathbb{\hat{E}}[|Y|\cdot
|X+Y|^{p-1}]\\
&  \leq \mathbb{\hat{E}}[|X|^{p}]^{1/p}\cdot \mathbb{\hat{E}[}|X+Y|^{(p-1)q}%
]^{1/q}\\
&  +\mathbb{\hat{E}}[|Y|^{p}]^{1/p}\cdot \mathbb{\hat{E}[}|X+Y|^{(p-1)q}]^{1/q}.%
\end{align*}
We observe that $(p-1)q=p$. Thus we have (\ref{ee4.7}).

{ { { For each $p,q>0$ with $\frac{1}{p}+\frac
{1}{q}=1$ we have %
\[
\left \Vert XY\right \Vert =\mathbb{\hat{E}}[|XY|]\leq \left \Vert
X\right \Vert _{p}\left \Vert X\right \Vert _{q}.
\]
With this we have $\left \Vert X\right \Vert _{p}\leq \left \Vert
X\right \Vert _{p^{\prime}}$ if $p\leq p^{\prime}$. } } }
\end{proof}

\begin{remark}
{ {{It is easy to check that $\mathcal{H}/\mathcal{H}_{0}^{p}$ is a
normed
space under $\left \Vert \cdot \right \Vert _{p}$. We then extend $\mathcal{H}%
/\mathcal{H}_{0}^{p}$ to its completion $\mathcal{\hat{H}}_{p}$
under\ this norm. $(\mathcal{\hat{H}}_{p},\left \Vert \cdot \right
\Vert _{p})$ is a Banach space. The sublinear expectation
$\mathbb{\hat{E}}[\cdot]$ can be also continuously extended from
$\mathcal{H}/\mathcal{H}_{0}$ to $\mathcal{\hat{H}}_{p}$, which
satisfies (a)--(d). } } }
\end{remark}

{ { { For any $X\in \mathcal{H}$, the mappings
\[
X^{+}(\omega):\mathcal{H\longmapsto H}\  \  \  \text{and \  \ }X^{-}%
(\omega):\mathcal{H\longmapsto H}%
\]
satisfy
\[
|X^{+}-Y^{+}|\leq|X-Y|\text{ \  \ and \ }\ |X^{-}-Y^{-}|=|(-X)^{+}%
-(-Y)^{+}|\leq|X-Y|.
\]
Thus they are both contraction mappings under $\left \Vert \cdot
\right \Vert
_{p}$ and can be continuously extended to the Banach space $(\mathcal{\hat{H}%
}_{p},\left \Vert \cdot \right \Vert _{p})$. } } }

{ { { We define the partial order \textquotedblleft$\geq$\textquotedblright%
\ in this Banach space. } } }

\begin{definition}
{ { { An element $X$ in $(\mathcal{\hat{H}},\left \Vert \cdot \right
\Vert )$ is said to be nonnegative, or $X\geq0$, $0\leq X$, if
$X=X^{+}$. We also denote by $X\geq Y$, or $Y\leq X$, if $X-Y\geq0$.
} } }
\end{definition}

{ { { It is easy to check that $X\geq Y$ and $Y\geq X$ implies $X=Y$
in $(\mathcal{\hat{H}}_{p},\left \Vert \cdot \right \Vert _{p})$. }
} }

{ { { The sublinear expectation $\mathbb{\hat{E}}[\cdot]$ can be
continuously extended to $(\mathcal{\hat{H}}_{p},\left \Vert \cdot
\right \Vert _{p})$ on which \textbf{(a)--(d)} still hold. } } }

\subsection{{{$G$-Brownian motion in a complete sublinear expectation space}}}

{ {{We can check that, for each }$p>0$ and for each $X\in L_{ip}%
^{0}(\mathcal{F})$ with the form $X(\omega)=\varphi(B_{t_{1}},\cdots,B_{t_{m}%
})$, for some $\varphi \in C_{l.Lip}(\mathbb{R}^{m})$, { }%
\[
\mathbb{\hat{E}}[|X|]=0\  \  \  \Longleftrightarrow \  \  \  \mathbb{\hat{E}}%
[|X|^{p}]=0\  \  \Longleftrightarrow \  \  \
\varphi(x)\equiv0\text{, }\forall x\in \mathbb{R}^{m}.
\]
{ For each }$p\geq1$, $\left \Vert X\right \Vert _{p}:=\mathbb{\hat{E}}%
[|X|^{p}]^{\frac{1}{p}}$, $X\in L_{ip}^{0}(\mathcal{F}_{T})$
(respectively, $L_{ip}^{0}(\mathcal{F})$) forms a norm and that
$L_{ip}^{0}(\mathcal{F}_{T})$ (respectively,
$L_{ip}^{0}(\mathcal{F})$), {can be continuously extended to a
Banach space, denoted by
\[
\mathcal{H}_{T}=L_{G}^{p}(\mathcal{F}_{T})\ \ \ \text{ (respectively, } \mathcal{H}=L_{G}^{p}(\mathcal{F}%
)).
\]
 For each $0\leq t\leq T<\infty$ we have $L_{G}^{p}(\mathcal{F}%
_{t})\subseteq L_{G}^{p}(\mathcal{F}_{T})\subset
L_{G}^{p}(\mathcal{F})$. It
is easy to check that, in $L_{G}^{p}(\mathcal{F}_{T})$ (respectively, $L_{G}%
^{p}(\mathcal{F})$), $\mathbb{\hat{E}}[\cdot]$ still satisfies
(a)--(d) in Definition \ref{Def-1}. } } }

{ { { We now consider the conditional expectation introduced in
(\ref{Condition}). For each fixed $t=t_{j}\leq T$, the conditional
expectation
$\mathbb{\hat{E}}[\cdot|\mathcal{H}_{t}]:L_{ip}^{0}(\mathcal{F}_{T})\mapsto
\mathcal{H}_t^0$ is a continuous mapping under $\left \Vert
\cdot \right \Vert $ since $\mathbb{\hat{E}}[\mathbb{\hat{E}}[X|\mathcal{H}%
_{t}]]=\mathbb{\hat{E}}[X]$, $X\in L_{ip}^{0}(\mathcal{F}_{T})$ and%
\begin{align*}
\mathbb{\hat{E}}[\mathbb{\hat{E}}[X|\mathcal{H}_{t}]-\mathbb{\hat{E}%
}[Y|\mathcal{H}_{t}]]  &  \leq \mathbb{\hat{E}}[X-Y],\\
\left \Vert \mathbb{\hat{E}}[X|\mathcal{H}_{t}]-\mathbb{\hat{E}}[Y|\mathcal{H}%
_{t}]\right \Vert  &  \leq \left \Vert X-Y\right \Vert .
\end{align*}
It follows that $\mathbb{\hat{E}}[\cdot|\mathcal{H}_{t}]$ can be
also extended
as a continuous mapping $L_{G}^{p}(\mathcal{F}_{T})\mapsto L_{G}%
^{p}(\mathcal{F}_{t})$. If the above $T$ is not fixed, then we can
obtain
$\mathbb{\hat{E}}[\cdot|\mathcal{H}_{t}]:L_{G}^{1}(\mathcal{F})\mapsto
L_{G}^{1}(\mathcal{F}_{t})$.} } }

\begin{proposition}
{ { { \label{Prop-1-7} The properties of Proposition
\ref{Prop-1-7-1} of  $\mathbb{\hat{E}}[\cdot|\mathcal{H}_{t}]$ still
hold for $X$, $Y\in$
{{$L_{G}^{1}$}}$(\mathcal{F})$:}\newline{\textbf{(a') }{If $X\geq
Y$, then
$\mathbb{\hat{E}}[X|\mathcal{H}_{t}]\geq \mathbb{\hat{E}}[Y|\mathcal{H}_{t}%
]$.\newline}\textbf{(b')
}$\mathbb{\hat{E}}[\eta|\mathcal{H}_{t}]=\eta$, for
each }}$t\in \lbrack0,\infty)$ and $\eta \in{{{{L_{G}^{1}}}}}(\mathcal{F}_{t}%
)$.\newline \textbf{(c') }$\mathbb{\hat{E}}[X|\mathcal{H}_{t}]-\mathbb{\hat{E}%
}[Y|\mathcal{H}_{t}]\leq \mathbb{\hat{E}}[X-Y|\mathcal{H}_{t}].$\newline%
{\textbf{(d')} $\mathbb{\hat{E}}[\eta X|\mathcal{H}_{t}]=\eta^{+}%
\mathbb{\hat{E}}[X|\mathcal{H}_{t}]+\eta^{-}\mathbb{\hat{E}}[-X|\mathcal{H}%
_{t}]$, for each bounded $\eta
\in{L}_{G}^{1}{(\mathcal{F}_{t})}.$}
\newline{\textbf{(e')}} $\mathbb{\hat{E}}[\mathbb{\hat{E}}[X|\mathcal{H}_{t}]|\mathcal{H}%
_{s}]=\mathbb{\hat{E}}[X|\mathcal{H}_{t\wedge s}],\ $ in
particular, ${{\mathbb{\hat{E}}[\mathbb{\hat{E}}[X|\mathcal{H}_{t}]]=\mathbb{\hat{E}%
}[X].\newline}}$
\newline\textbf{(f')} For each $X\in{L_G^1}(\mathcal{F}_{T}^{t})$ we have $\mathbb{\hat{E}%
}[X|\mathcal{H}_{t}]=\mathbb{\hat{E}}[X]$.}
\end{proposition}

\begin{definition}
{ { { An $X\in L_{G}^{1}(\mathcal{F})$ is said to be independent of
$\mathcal{F}_{t}$ under the $G$--expectation $\mathbb{\hat{E}}$ for
some given $t\in \lbrack0,\infty)$, if for each real function $\Phi$
suitably defined on $\mathbb{R}$ such that $\Phi(X)\in
L_{G}^{1}(\mathcal{F})$ we have
\[
\mathbb{\hat{E}}[\Phi(X)|\mathcal{H}_{t}]=\mathbb{\hat{E}}[\Phi(X)].
\]
} } }
\end{definition}

\begin{remark}
{ { { It is clear that all elements in $L_{G}^{1}(\mathcal{F})$ are
independent of $\mathcal{F}_{0}$. Just like the classical situation,
the increments of $G$-Brownian motion $(B_{t+s}-B_{s})_{t\geq0}$ is
independent of $\mathcal{F}_{s}$. In fact it is a new $G$--Brownian
motion since, just like the classical situation, the increments of
$B$ are identically distributed. } } }
\end{remark}

\begin{example}
{ { { \label{Exam-2}For each $n\in \mathbb{N},$ $0\leq t<\infty$ and
$X\in
L_{G}^{1}(\mathcal{F}_{t})$, since $\mathbb{\hat{E}[}B_{T-t}^{2n-1}%
]=\mathbb{\hat{E}[-}B_{T-t}^{2n-1}]$, we have, by \textbf{(f')} of
Proposition \ref{Prop-1-7},
\begin{align*}
\mathbb{\hat{E}}[X(B_{T}-B_{t})^{2n-1}]  &  =\mathbb{\hat{E}}[X^{+}%
\mathbb{\hat{E}[}(B_{T}-B_{t})^{2n-1}|\mathcal{H}_{t}]+X^{-}\mathbb{\hat{E}%
[-}(B_{T}-B_{t})^{2n-1}|\mathcal{H}_{t}]]\\
&  =\mathbb{\hat{E}}[|X|]\cdot \mathbb{\hat{E}[}B_{T-t}^{2n-1}],\\
\mathbb{\hat{E}}[X(B_{T}-B_{t})|\mathcal{H}_{t}]  &  =\mathbb{\hat{E}%
}[-X(B_{T}-B_{t})|\mathcal{H}_{t}]=0.
\end{align*}
We also have%
\[
\mathbb{\hat{E}}[X(B_{T}-B_{t})^{2}|\mathcal{H}_{t}]=X^{+}(T-t)-\sigma
^{2}X^{-}(T-t).
\]
} } }
\end{example}

\begin{remark}
{ { {It is clear that we can define an expectation
$\mathbf{E}[\cdot]$ on $L_{ip}^{0}(\mathcal{F})$ in the same way as
in Definition \ref{Def-3} with the standard normal distribution
}$\mathbf{F=}\mathcal{N}(0,1)${ in place
of $\mathbb{F}_{\xi}=\mathcal{N}(0;[\sigma,1])$ on }$(\mathbb{R}%
,\mathcal{B}(\mathbb{R}))$.{ Since }$\mathbf{F}${ is dominated by
$\mathbb{F}_{\xi}$ in the sense $\mathbf{F}[\varphi]-\mathbf{F}[\psi
]\leq \mathbb{F}_{\xi}[\varphi-\psi]$, then $\mathbf{E}[\cdot]$ can
be continuously extended to $L_{G}^{1}(\mathcal{F})$.
$\mathbf{E}[\cdot]$ is a linear expectation under which
$(B_{t})_{t\geq0}$ behaves as a Brownian motion. We have
\begin{equation}
\mathbf{E}[X]\leq \mathbb{\hat E}[X],\  \  \forall X\in L_{G}^{1}(\mathcal{F}%
).\label{Eg-domi}%
\end{equation}
In particular, $\mathbb{\hat{E}[}B_{T-t}^{2n-1}]=\mathbb{\hat{E}[-}%
B_{T-t}^{2n-1}]\geq \mathbf{E}\mathbb{[-}B_{T-t}^{2n-1}]=0$. Such
kind of extension under a domination relation was discussed in
details in \cite{Peng2005}. } } }
\end{remark}

The following property is very useful.

\begin{proposition}
{ { { \label{E-x+y}Let $X,Y\in L_{G}^{1}(\mathcal{F})$ be such that
$\mathbb{\hat{E}}[Y|\mathcal{H}_{t}]=-\mathbb{\hat{E}}[-Y|\mathcal{H}_{t}]$,
then we have%
\[
\mathbb{\hat{E}}[X+Y|\mathcal{H}_{t}]=\mathbb{\hat{E}}[X|\mathcal{H}%
_{t}]+\mathbb{\hat{E}}[Y|\mathcal{H}_{t}].
\]
In particular, if }}$t=0$ and{{ {$\mathbb{\hat{E}}[Y|\mathcal{H}_{0}]=$%
}$\mathbb{\hat{E}}[Y]=\mathbb{\hat{E}}[-Y]=0$, then $\mathbb{\hat{E}%
}[X+Y]=\mathbb{\hat{E}}[X]$. } } }
\end{proposition}

\begin{proof}
{ { { We just need to use twice the sub-additivity of $\mathbb{\hat{E}}%
[\cdot|\mathcal{H}_{t}]$:
\begin{align*}
\mathbb{\hat{E}}[X+Y{{|\mathcal{H}_{t}}}]  &  \geq \mathbb{\hat{E}%
}[X{{|\mathcal{H}_{t}}}]-\mathbb{\hat{E}}[-Y{{|\mathcal{H}_{t}}}%
]=\mathbb{\hat{E}}[X{{|\mathcal{H}_{t}}}]+\mathbb{\hat{E}}[Y{{|\mathcal{H}%
_{t}}}]\\
&  \geq \mathbb{\hat{E}}[X+Y{{|\mathcal{H}_{t}}}].
\end{align*}
} } }
\end{proof}

\begin{example}
{ { { \label{Exam-B2}We have%
\begin{align*}
\mathbb{\hat{E}}[B_{t}^{2}-B_{s}^{2}|\mathcal{H}_{s}]  &  =\mathbb{\hat{E}%
}[(B_{t}-B_{s}+B_{s})^{2}-B_{s}^{2}|\mathcal{H}_{s}]\\
&  =E[(B_{t}-B_{s})^{2}+2(B_{t}-B_{s})B_{s}|\mathcal{H}_{s}]\\
&  =t-s,
\end{align*}
since $2(B_{t}-B_{s})B_{s}$ satisfies the condition for $Y$ in
Proposition
\ref{E-x+y}, and%
\begin{align*}
\mathbb{\hat{E}}[(B_{t}^{2}-B_{s}^{2})^{2}|\mathcal{H}_{s}]  &
=\mathbb{\hat
{E}}[\{(B_{t}-B_{s}+B_{s})^{2}-B_{s}^{2}\}^{2}|\mathcal{H}_{s}]\\
&  =\mathbb{\hat{E}}[\{(B_{t}-B_{s})^{2}+2(B_{t}-B_{s})B_{s}\}^{2}%
|\mathcal{H}_{s}]\\
&  =\mathbb{\hat{E}}[(B_{t}-B_{s})^{4}+4(B_{t}-B_{s})^{3}B_{s}+4(B_{t}%
-B_{s})^{2}B_{s}^{2}|\mathcal{H}_{s}]\\
&  \leq \mathbb{\hat{E}}[(B_{t}-B_{s})^{4}]+4\mathbb{\hat{E}}[|B_{t}-B_{s}%
|^{3}]|B_{s}|+4(t-s)B_{s}^{2}\\
&  =3(t-s)^{2}+8(t-s)^{3/2}|B_{s}|+4(t-s)B_{s}^{2}.
\end{align*}
}}}
\end{example}

{ { } }

\subsection{{{{It\^{o}'s integral of $G$--Brownian motion}}}}

\subsubsection{{Bochner's integral}}

\begin{definition}
{ { { \label{Def-4}For $T\in \mathbb{R}_{+}$, a partition $\pi_{T}$
of $[0,T]$ is a finite ordered subset $\pi=\{t_{1},\cdots,t_{N}\}$
such that $0=t_{0}<t_{1}<\cdots<t_{N}=T$. We denote
\[
\mu(\pi_{T})=\max \{|t_{i+1}-t_{i}|:\ \ i=0,1,\cdots,N-1\} \text{.}%
\]
We use $\pi_{T}^{N}=\{t_{0}^{N},t_{1}^{N},\cdots,t_{N}^{N}\}$ to
denote a sequence of partitions of $[0,T]$ such that
$\lim_{N\rightarrow \infty}\mu (\pi_{T}^{N})=0$. } } }
\end{definition}

{ { { Let $p\geq1$ be fixed. We consider the following type of
simple processes: For a given partition
$\{t_{0},\cdots,t_{N}\}=\pi_{T}$ of $[0,T]$,
we set%
\[
\eta_{t}(\omega)=\sum_{j=0}^{N-1}\xi_{j}(\omega)\mathbf{I}_{[t_{j},t_{j+1}%
)}(t),
\]
where $\xi_{i}\in L_{G}^{p}(\mathcal{F}_{t_{i}})$,
$i=0,1,2,\cdots,N-1$, are
given. The collection of this form of processes is denoted by $M_{G}%
^{p,0}(0,T)$. } } }

\begin{definition}
{ { { \label{Def-5}For an $\eta \in M_{G}^{p,0}(0,T)$ with
$\eta_{t}=\sum
_{j=0}^{N-1}\xi_{j}(\omega)\mathbf{I}_{[t_{j},t_{j+1})}(t)$ the
related Bochner integral is
\[
\int_{0}^{T}\eta_{t}(\omega)dt=\sum_{j=0}^{N-1}\xi_{j}(\omega)(t_{j+1}%
-t_{j}).
\]
} } }
\end{definition}

\begin{remark}
{ { { For each $\eta \in M_{G}^{p,0}(0,T)$ we set
\[
\mathbb{\hat{E}}_{T}[\eta]:=\frac{1}{T}\int_{0}^{T}\mathbb{\hat{E}}[\eta
_{t}]dt=\frac{1}{T}\sum_{j=0}^{N-1}\mathbb{\hat{E}[}\xi_{j}(\omega
)](t_{j+1}-t_{j}).
\]
It is easy to check that
$\mathbb{\hat{E}}_{T}:M_{G}^{p,0}(0,T)\longmapsto \mathbb{R}$ forms
a sublinear expectation satisfying (a)--(d) of Definition
\ref{Def-1}. From Remark \ref{Rem-2} we can introduce a natural norm
$\left \Vert \eta \right \Vert _{M_{G}^{p}(0,T)}=\left \{  {{\int_{0}%
^{T}\mathbb{\hat{E}}[|\eta_{t}|^{p}]dt}}\right \}  ^{1/p}$. Under
this norm $M_{G}^{p,0}(0,T)$ can be continuously extended to {{a
Banach} space}. } } }
\end{remark}

\begin{definition}
 For each $p\geq1$, we will denote by $M_{G}^{p}(0,T)$ the
completion of
$M_{G}^{p,0}(0,T)$ under the norm%
\[
{{\left \Vert \eta \right \Vert _{M_{G}^{p}(0,T)}=\left \{  {{\int_{0}%
^{T}\mathbb{\hat{E}}[|\eta_{t}|^{p}]dt}}\right \}  ^{1/p}}}.
\]
\end{definition}
For $\eta\in M_{G}^{p,0}(0,T)$, we have
\begin{align*}
\mathbb{\hat{E}}[|\frac{1}{T}\int_{0}^{T}\eta_{t}(\omega)dt|^{p}]  &  \leq \frac{1}{T}\sum_{j=0}^{N-1}\mathbb{\hat{E}[}|\xi_{j}(\omega)|^{p}](t_{j+1}-t_{j})\\
&  =\frac{1}{T}\int_{0}^{T}\mathbb{\hat{E}}[|\eta_{t}|^{p}]dt.
\end{align*}
We then have

\begin{proposition}
The linear mapping $\int_{0}^{T}\eta_{t}(\omega)dt:M_{G}^{p,0}%
(0,T)\mapsto L_{G}^{p}(\mathcal{F}_{T})$ is continuous; thus it can
be continuously extended to $M_{G}^{p}(0,T)\mapsto
L_{G}^{p}(\mathcal{F}_{T})$. We still denote this extended mapping
by $\int_{0}^{T}\eta_{t}(\omega)dt$,
$\eta \in M_{G}^{p}(0,T)$. We have%
\begin{equation}
\mathbb{\hat{E}}[|\frac{1}{T}\int_{0}^{T}\eta_{t}(\omega)dt|^{p}]\leq \frac{1}{T}\int_{0}%
^{T}\mathbb{\hat{E}}[|\eta_{t}|^{p}]dt,\  \  \  \forall \eta \in
M_{G}^{p}(0,T).
\label{ine-dt}%
\end{equation}

\end{proposition}

We have $M_{G}^{p}(0,T)\supset M_{G}^{q}(0,T)$, for $p\leq q$.

\subsubsection{{It\^{o}'s integral of $G$--Brownian motion}}

\begin{definition}
For each $\eta \in M_{G}^{2,0}(0,T)$ with the form
\[\eta_{t}(\omega
)=\sum_{j=0}^{N-1}\xi_{j}(\omega)\mathbf{I}_{[t_{j},t_{j+1})}(t),
\]
 we define
\[
I(\eta)=\int_{0}^{T}\eta(s)dB_{s}:=\sum_{j=0}^{N-1}\xi_{j}(B_{t_{j+1}%
}-B_{t_{j}})\mathbf{.}%
\]

\end{definition}

\begin{lemma}
{ { { \label{bdd}The mapping $I:M_{G}^{2,0}(0,T)\longmapsto L_{G}%
^{2}(\mathcal{F}_{T})$ is a linear continuous mapping and thus can
be
continuously extended to $I:M_{G}^{2}(0,T)\longmapsto L_{G}^{2}(\mathcal{F}%
_{T})$: We have
\begin{align}
\mathbb{\hat{E}}[\int_{0}^{T}\eta(s)dB_{s}]  &  =0,\  \  \label{e1}\\
\mathbb{\hat{E}}[(\int_{0}^{T}\eta(s)dB_{s})^{2}]  &  \leq \int_{0}%
^{T}\mathbb{\hat{E}}[(\eta(t))^{2}]dt. \label{e2}%
\end{align}
} } }
\end{lemma}

\begin{definition}
{ { { We define, for a fixed $\eta \in M_{G}^{2}(0,T)$, the
stochastic integral
\[
\int_{0}^{T}\eta(s)dB_{s}:=I(\eta).
\]
It is clear that (\ref{e1}) and (\ref{e2}) still hold for $\eta \in M_{G}%
^{2}(0,T)$. } } }
\end{definition}

{ { { \textbf{Proof of Lemma \ref{bdd}. }From Example \ref{Exam-2},
for each $j$,
\[
\mathbb{\hat{E}}\mathbf{[}\xi_{j}(B_{t_{j+1}}-B_{t_{j}})|\mathcal{H}_{t_{j}%
}]=\mathbb{\hat{E}}\mathbf{[-}\xi_{j}(B_{t_{j+1}}-B_{t_{j}})|\mathcal{H}%
_{t_{j}}]=0.
\]
We have%
\begin{align*}
\mathbb{\hat{E}}[\int_{0}^{T}\eta(s)dB_{s}]  &
=\mathbb{\hat{E}[}\int
_{0}^{t_{N-1}}\eta(s)dB_{s}+\xi_{N-1}(B_{t_{N}}-B_{t_{N-1}})]\\
&  =\mathbb{\hat{E}[}\int_{0}^{t_{N-1}}\eta(s)dB_{s}+\mathbb{\hat{E}%
}\mathbf{[}\xi_{N-1}(B_{t_{N}}-B_{t_{N-1}})|\mathcal{H}_{t_{N-1}}]]\\
&  =\mathbb{\hat{E}[}\int_{0}^{t_{N-1}}\eta(s)dB_{s}].
\end{align*}
We then can repeat this procedure to obtain (\ref{e1}). We now prove
(\ref{e2}):
\begin{align*}
\mathbb{\hat{E}}[(\int_{0}^{T}\eta(s)dB_{s})^{2}]  &  =\mathbb{\hat{E}%
[}\left(  \int_{0}^{t_{N-1}}\eta(s)dB_{s}+\xi_{N-1}(B_{t_{N}}-B_{t_{N-1}%
})\right)  ^{2}]\\
&  =\mathbb{\hat{E}[}\left(  \int_{0}^{t_{N-1}}\eta(s)dB_{s}\right)
^{2}+\mathbb{\hat{E}}[2\left(
\int_{0}^{t_{N-1}}\eta(s)dB_{s}\right)
\xi_{N-1}(B_{t_{N}}-B_{t_{N-1}})\\
&  +\xi_{N-1}^{2}(B_{t_{N}}-B_{t_{N-1}})^{2}|\mathcal{H}_{t_{N-1}}]]\\
&  =\mathbb{\hat{E}[}\left(  \int_{0}^{t_{N-1}}\eta(s)dB_{s}\right)  ^{2}%
+\xi_{N-1}^{2}(t_{N}-t_{N-1})].
\end{align*}
Thus $\mathbb{\hat{E}}[(\int_{0}^{t_{N}}\eta(s)dB_{s})^{2}]\leq
\mathbb{\hat
{E}[}\left(  \int_{0}^{t_{N-1}}\eta(s)dB_{s}\right)  ^{2}]+\mathbb{\hat{E}%
}[\xi_{N-1}^{2}](t_{N}-t_{N-1})$. We then repeat this procedure to
deduce
\[
\mathbb{\hat{E}}[(\int_{0}^{T}\eta(s)dB_{s})^{2}]\leq \sum_{j=0}^{N-1}%
\mathbb{\hat{E}}[(\xi_{j})^{2}](t_{j+1}-t_{j})=\int_{0}^{T}\mathbb{\hat{E}%
}[(\eta(t))^{2}]dt.
\]
$\blacksquare$ } } }

{ { { We list some main properties of the It\^{o}'s integral of
$G$--Brownian motion. We denote for some $0\leq s\leq t\leq T$,
\[
\int_{s}^{t}\eta_{u}dB_{u}:=\int_{0}^{T}\mathbf{I}_{[s,t]}(u)\eta_{u}dB_{u}.
\]
We have } } }

\begin{proposition}
{ { { \label{Prop-Integ}Let $\eta,\theta \in M_{G}^{2}(0,T)$ and let
$0\leq s\leq r\leq t\leq T$. Then in $L_{G}^{1}(\mathcal{F}_{T})$ we
have\newline\textsl{(i)}
$\int_{s}^{t}\eta_{u}dB_{u}=\int_{s}^{r}\eta_{u}dB_{u}+\int_{r}^{t}\eta
_{u}dB_{u},$\newline\textsl{(ii)} $\int_{s}^{t}(\alpha \eta_{u}+\theta_{u})dB_{u}%
=\alpha \int_{s}^{t}\eta_{u}dB_{u}+\int_{s}^{t}\theta_{u}dB_{u},\
$if$\  \alpha$ is bounded and in
$L_{G}^{1}(\mathcal{F}_{s})$,\newline\textsl{(iii)} $\mathbb{\hat
{E}[}X+\int_{r}^{T}\eta_{u}dB_{u}|\mathcal{H}_{s}]=\mathbb{\hat{E}[}X]$,
$\forall X\in L_{G}^{1}(\mathcal{F})$. } } }
\end{proposition}

\subsubsection{{{{Quadratic variation process of $G$--Brownian motion}}}}

{ { { We now study a very interesting process of the $G$-Brownian
motion. Let $\pi_{t}^{N}$, $N=1,2,\cdots$, be a sequence of
partitions of $[0,t]$. We consider} } }

{ { {
\begin{align*}
B_{t}^{2}  &  =\sum_{j=0}^{N-1}[B_{t_{j+1}^{N}}^{2}-B_{t_{j}^{N}}^{2}]\\
&
=\sum_{j=0}^{N-1}2B_{t_{j}^{N}}(B_{t_{j+1}^{N}}-B_{t_{j}^{N}})+\sum
_{j=0}^{N-1}(B_{t_{j+1}^{N}}-B_{t_{j}^{N}})^{2}.
\end{align*}
As $\mu(\pi_{t}^{N})\rightarrow0$ the first term of the right side
tends to $2\int_{0}^{t}B_{s}dB_{s}$. The second term must converge.
We denote its limit by $\left \langle B\right \rangle _{t}$, i.e.,
\begin{equation}
\left \langle B\right \rangle
_{t}=\lim_{\mu(\pi_{t}^{N})\rightarrow0}\sum
_{j=0}^{N-1}(B_{t_{j+1}^{N}}-B_{t_{j}^{N}})^{2}=B_{t}^{2}-2\int_{0}^{t}%
B_{s}dB_{s}. \label{quadra-def}%
\end{equation}
By the above construction $\{\left \langle B\right \rangle
_{t}\}_{t\geq0}$, is an increasing process with $\left \langle
B\right \rangle _{0}=0$. We call it the \textbf{quadratic variation
process} of the $G$--Brownian motion $B$. Clearly $\left \langle
B\right \rangle $ is an increasing process. It characterizes the
part of statistic uncertainty of $G$--Brownian motion. It is
important to keep in mind that $\left \langle B\right \rangle _{t}$
is not a deterministic process unless the case $\sigma=1$, i.e.,
when $B$ is a classical Brownian motion. In fact we have } } }

\begin{lemma}
{ { { \label{Lem-Q1}We have, for each $0\leq s\leq t<\infty$%
\begin{align}
\mathbb{\hat{E}}[\left \langle B\right \rangle _{t}-\left \langle
B\right \rangle
_{s}|\mathcal{H}_{s}]  &  =t-s,\  \  \label{quadra}\\
\mathbb{\hat{E}}[-(\left \langle B\right \rangle _{t}-\left \langle
B\right \rangle _{s})|\mathcal{H}_{s}]  &  =-\sigma^{2}(t-s). \label{quadra1}%
\end{align}
} } }
\end{lemma}

\begin{proof}
{ { { By the definition of $\left \langle B\right \rangle $ and
Proposition \ref{Prop-Integ}-(iii),
\begin{align*}
\mathbb{\hat{E}}[\left \langle B\right \rangle _{t}-\left \langle
B\right \rangle
_{s}|\mathcal{H}_{s}]  &  =\mathbb{\hat{E}}[B_{t}^{2}-B_{s}^{2}-2\int_{s}%
^{t}B_{u}dB_{u}|\mathcal{H}_{s}]\\
&  =\mathbb{\hat{E}}[B_{t}^{2}-B_{s}^{2}|\mathcal{H}_{s}]=t-s.
\end{align*}
The last step can be checked as in Example \ref{Exam-B2}. We then
have (\ref{quadra}). (\ref{quadra1}) can be proved analogously with
the
consideration of $\mathbb{\hat{E}}[-(B_{t}^{2}-B_{s}^{2})|\mathcal{H}%
_{s}]=-\sigma^{2}(t-s)$. } } }
\end{proof}

{ { { To define the integration of a process $\eta \in
M_{G}^{1}(0,T)$ with
respect to $d\left \langle B\right \rangle $, we first define a mapping:%
\[
Q_{0,T}(\eta)=\int_{0}^{T}\eta(s)d\left \langle B\right \rangle
_{s}:=\sum _{j=0}^{N-1}\xi_{j}(\left \langle B\right \rangle
_{t_{j+1}}-\left \langle B\right \rangle
_{t_{j}}):M_{G}^{1,0}(0,T)\mapsto L^{1}(\mathcal{F}_{T}).
\]
} } }

\begin{lemma}
{ { { \label{Lem-Q2}For each $\eta \in M_{G}^{1,0}(0,T)$,
\begin{equation}
\mathbb{\hat{E}}[|Q_{0,T}(\eta)|]\leq
\int_{0}^{T}\mathbb{\hat{E}}[|\eta
_{s}|]ds.\  \label{dA}%
\end{equation}
Thus $Q_{0,T}:M_{G}^{1,0}(0,T)\mapsto L^{1}(\mathcal{F}_{T})$ is a
continuous linear mapping. Consequently, $Q_{0,T}$ can be uniquely
extended to
$L_{\mathcal{F}}^{1}(0,T)$. We still denote this mapping \ by%
\[
\int_{0}^{T}\eta(s)d\left \langle B\right \rangle
_{s}=Q_{0,T}(\eta),\  \  \eta \in
M_{G}^{1}(0,T)\text{.}%
\]
We still have
\begin{equation}
\mathbb{\hat{E}}[|\int_{0}^{T}\eta(s)d\left \langle B\right \rangle _{s}%
|]\leq \int_{0}^{T}\mathbb{\hat{E}}[|\eta_{s}|]ds,\  \  \forall \eta \in M_{G}%
^{1}(0,T)\text{.} \label{qua-ine}%
\end{equation}
} } }
\end{lemma}

\begin{proof}
{ { { From Lemma \ref{Lem-Q1} (\ref{dA}) can be checked as follows:%
\begin{align*}
\mathbb{\hat{E}}[|\sum_{j=0}^{N-1}\xi_{j}(\left \langle B\right
\rangle
_{t_{j+1}}-\left \langle B\right \rangle _{t_{j}})|]  &  \leq \sum_{j=0}%
^{N-1}\mathbb{\hat{E}[}|\xi_{j}|\cdot \mathbb{\hat{E}}[\left \langle
B\right \rangle _{t_{j+1}}-\left \langle B\right \rangle _{t_{j}}|\mathcal{H}%
_{t_{j}}]]\\
&  =\sum_{j=0}^{N-1}\mathbb{\hat{E}[}|\xi_{j}|](t_{j+1}-t_{j})\\
&  =\int_{0}^{T}\mathbb{\hat{E}}[|\eta_{s}|]ds.
\end{align*}
} } }
\end{proof}

{ { { A very interesting point of the quadratic variation process
$\left \langle B\right \rangle $ is, just like the $G$--Brownian
motion $B$ itself, the increment $\left \langle B\right \rangle
_{t+s}-\left \langle B\right \rangle _{s}$ is independent of
$\mathcal{F}_{s}$ and identically distributed like $\left \langle
B\right \rangle _{t}$. In fact we have } } }

\begin{lemma}
{ { { \label{Lem-Qua2}For each fixed $s\geq0$, $(\left \langle
B\right \rangle _{s+t}-\left \langle B\right \rangle _{s})_{t\geq0}$
is independent of $\mathcal{F}_{s}$. It is the quadratic variation
process of the Brownian motion $B_{t}^{s}=B_{s+t}-B_{s}$, $t\geq0$,
i.e., $\left \langle B\right \rangle _{s+t}-\left \langle B\right
\rangle _{s}=\left \langle B^{s}\right \rangle _{t}$. } } }
\end{lemma}

\begin{proof}
{ { { The independence follows directly from
\begin{align*}
\left \langle B\right \rangle _{s+t}-\left \langle B\right \rangle
_{s}  &
=B_{t+s}^{2}-2\int_{0}^{s+t}B_{r}dB_{r}-[B_{s}^{2}-2\int_{0}^{s}B_{r}dB_{r}]\\
&  =(B_{t+s}-B_{s})^{2}-2\int_{s}^{s+t}(B_{r}-B_{s})d(B_{r}-B_{s})\\
&  =\left \langle B^{s}\right \rangle _{t}.
\end{align*}
} } }
\end{proof}

\begin{proposition}
{ { { \label{Prop-temp}}}Let $0\leq s\leq t$, $\xi \in L_{G}^{1}(F_{s})$. Then%
\begin{align*}
\mathbb{\hat{E}}[X+\xi(B_{t}^{2}-B_{s}^{2})]  &
=\mathbb{\hat{E}}[X+\xi
(B_{t}-B_{s})^{2}]\\
&  =\mathbb{\hat{E}}[X+\xi(\left \langle B\right \rangle _{t}-\left
\langle B\right \rangle _{s})].
\end{align*}
{ } }
\end{proposition}

\begin{proof}
{ { { By (\ref{quadra-def}) and Proposition \ref{E-x+y} we have%
\begin{align*}
\mathbb{\hat{E}}[X+\xi(B_{t}^{2}-B_{s}^{2})]  &  =\mathbb{\hat{E}}%
[X+\xi(\left \langle B\right \rangle _{t}-\left \langle B\right \rangle _{s}%
+2\int_{s}^{t}B_{u}dB_{u})]\\
&  =\mathbb{\hat{E}}[X+\xi(\left \langle B\right \rangle _{t}-\left
\langle B\right \rangle _{s})].
\end{align*}
We also have
\begin{align*}
\mathbb{\hat{E}}[X+\xi(B_{t}^{2}-B_{s}^{2})]  &  =\mathbb{\hat{E}}%
[X+\xi \{(B_{t}-B_{s})^{2}+2(B_{t}-B_{s})B_{s}^{{}}\}]\\
&  =\mathbb{\hat{E}}[X+\xi(B_{t}-B_{s})^{2}].
\end{align*}
} } }
\end{proof}

{ { { We have the following isometry } } }

\begin{proposition}
{ { { Let $\eta \in M_{G}^{2}(0,T)$. We have%
\begin{equation}
\mathbb{\hat{E}}[(\int_{0}^{T}\eta(s)dB_{s})^{2}]=\mathbb{\hat{E}}[\int
_{0}^{T}\eta^{2}(s)d\left \langle B\right \rangle _{s}]. \label{isometry}%
\end{equation}
} } }
\end{proposition}

\begin{proof}
{ { { We first consider $\eta \in M_{G}^{2,0}(0,T)$ with the form
\[
\eta_{t}(\omega)=\sum_{j=0}^{N-1}\xi_{j}(\omega)\mathbf{I}_{[t_{j},t_{j+1}%
)}(t)
\]
and thus $\int_{0}^{T}\eta(s)dB_{s}:=\sum_{j=0}^{N-1}\xi_{j}(B_{t_{j+1}%
}-B_{t_{j}})$\textbf{.} From Proposition \ref{E-x+y} we have
\[
\mathbb{\hat{E}}[X+2\xi_{j}(B_{t_{j+1}}-B_{t_{j}})\xi_{i}(B_{t_{i+1}}%
-B_{t_{i}})]=\mathbb{\hat{E}}[X]\text{, for }X\in L_{G}^{1}(\mathcal{F)}%
\text{, }i\not =j.
\]
Thus%
\[
\mathbb{\hat{E}}[(\int_{0}^{T}\eta(s)dB_{s})^{2}]=\mathbb{\hat{E}[}\left(
\sum_{j=0}^{N-1}\xi_{j}(B_{t_{j+1}}-B_{t_{j}})\right)  ^{2}]=\mathbb{\hat{E}%
[}\sum_{j=0}^{N-1}\xi_{j}^{2}(B_{t_{j+1}}-B_{t_{j}})^{2}].
\]
From this and Proposition \ref{Prop-temp} it follows that
\[
\mathbb{\hat{E}}[(\int_{0}^{T}\eta(s)dB_{s})^{2}]=\mathbb{\hat{E}[}\sum
_{j=0}^{N-1}\xi_{j}^{2}(\left \langle B\right \rangle
_{t_{j+1}}-\left \langle
B\right \rangle _{t_{j}})]=\mathbb{\hat{E}[}\int_{0}^{T}\eta^{2}%
(s)d\left \langle B\right \rangle _{s}].
\]
Thus (\ref{isometry}) holds for $\eta \in M_{G}^{2,0}(0,T)$. We can
continuously extend the above equality to the case $\eta \in
M_{G}^{2}(0,T)$ and prove (\ref{isometry}). }}}
\end{proof}

\newpage

\subsubsection{The distribution of $\left \langle B\right \rangle _{t}$}

The quadratic variation process $\left \langle B\right \rangle $ of
$G$-Brownian motion $B$ is a very interesting process. We have seen
that the $G$-Brownian motion $B$ is a typical process with variance
uncertainty but without mean-uncertainty. This uncertainty is
concentrated in $\left \langle B\right \rangle $. Moreover, $\left
\langle B\right \rangle $ itself is a typical process with
mean-variance. This fact will be applied to measure the
mean-uncertainty of risky positions.

\begin{lemma}
{{{We have
\begin{equation}
\mathbb{\hat{E}}[\left \langle B^{s}\right \rangle _{t}^{2}|\mathcal{H}%
_{s}]=\mathbb{\hat{E}}[\left \langle B\right \rangle
_{t}^{2}]\leq10t^{2}.
\label{Qua2}%
\end{equation}
}}}
\end{lemma}

\begin{proof}
{{{
\begin{align*}
{{{\mathbb{\hat{E}}[\left \langle B\right \rangle _{t}^{2}].}}}  &
=\mathbb{\hat{E}}[\{(B_{t})^{2}-2\int_{0}^{t}B_{u}dB_{u}\}^{2}]\\
&  \leq2\mathbb{\hat{E}}[(B_{t})^{4}]+8\mathbb{\hat{E}}[(\int_{0}^{t}%
B_{u}dB_{u})^{2}]\\
&  \leq6t^{2}+8\int_{0}^{t}\mathbb{\hat{E}[(}B_{u})^{2}]du\\
&  =10t^{2}.
\end{align*}
}}}
\end{proof}

\begin{definition}
A random variable $\xi$ of a sublinear expectation space{{
$(\Omega,\widetilde{\mathcal{H}},\widetilde{\mathbb{E}})$ is said to
be }}$\mathcal{U}_{[\underline{\mu},\overline{\mu}]}$-distributed
for some given interval $[\underline{\mu},\overline{\mu}]\subset
\mathbb{R}$ if for each $\varphi \in
C_{l.Lip}(\mathbb{R})$ we have%
\[
{\widetilde{\mathbb{E}}[\varphi(\xi)]=\sup_{x\in \lbrack
\underline{\mu
},\overline{\mu}]}\varphi(x).}%
\]
{{ \newline}}
\end{definition}

\begin{theorem}
 Let $(b_{t})_{t\geq0}$ be a process defined in
$(\Omega,\widetilde{\mathcal{H}},\widetilde{\mathbb{E}})$ such that
\newline\textsl{(i)} $b_{0}=0$.\newline\textsl{(ii)} For each
$t,s\geq0$, the difference $b${$_{t+s}%
-b_{t}$ and }$b${$_{s}$ are identically distributed and independent
to $(b_{t_{1}},b_{t_{2}},\cdots,b_{t_{n}})$ for each $n\in \mathbb{N}$ and $0\leq t_{1}%
,\cdots,t_{n}\leq t$. \newline\textsl{(iii)} }$b_{0}=0$ and
$\lim_{t\downarrow0}\widetilde{\mathbb{E}}${$[b_{t}^{2}]t^{-1}=0$.
\newline Then }$b_{t}$ is
$\mathcal{U}_{[\underline{\mu}t,\overline{\mu}t]}$-distributed{{
with }}$\overline{\mu}${{$=$}$\widetilde{\mathbb{E}}${$[b_{1}]$ and
$\underline{\mu}=-$}$\widetilde{\mathbb{E}}${$[-b_{1}]$. } }
\end{theorem}

\begin{proof}
We need only to prove that $b_{t}$ is
$\mathcal{U}_{[\underline{\mu}t,\overline{\mu}t]}$-distributed. We
first prove that
\[
\widetilde{\mathbb{E}}[b_{t}]=\overline{\mu}t\text{
\ and }-\widetilde{\mathbb{E}}[-b_{t}%
]=\underline{\mu}t.
\]
We set $\varphi(t):=\widetilde{\mathbb{E}}[b_{t}]$. Then $\varphi(0)=0$ and $\lim_{t\downarrow0}\varphi(t)$%
{$\rightarrow0$.} Since for each $t,s\geq0$%
\begin{align*}
\varphi(t+s)  &  =\widetilde{\mathbb{E}}[b_{t+s}]=\widetilde{\mathbb{E}}[(b_{t+s}%
-b_{s})+b_{s}]\\
&  =\varphi(t)+\varphi(s).
\end{align*}
Thus $\varphi(t)$ is linear and uniformly continuous in $t$ which
means $\widetilde{\mathbb{E}}[b_{t}]=\overline{\mu}t$. Similarly
$-\widetilde{\mathbb{E}}[-b_{t}]=\underline{\mu}t$.

We now prove that $b_{t}$ is $\mathcal{U}
_{[\underline{\mu}t,\overline{\mu}t]}$-distributed. We just need to prove that for each fixed $\varphi \in C_{l.Lip}%
(\mathbb{R)}$, the function
\[
u(t,x):=\widetilde{\mathbb{E}}[\varphi(x+b_{t})],\  \ (t,x)\in \lbrack0,\infty)\times \mathbb{R}%
\]
is the viscosity solution of the following ${G_{\underline{\mu}%
,\overline{\mu}}}$-drift equation
\begin{equation}
\partial_{t}u-2G(\partial_{x}u)=0,\  \text{for }t>0,\  \ \  \ u|_{t=0}=\varphi
\label{G-mean}\end{equation}
with ${G(a)=G_{\underline{\mu},\overline{\mu}}(a)=\frac{1}{2}(}$%
{$\overline{\mu}$}${a^{+}-}${$\underline{\mu}$}${a^{-})}$. Since the
function
\[
\widehat{u}(t,x)=\max_{v\in
[\underline{\mu},\overline{\mu}]}\varphi(x+vt)
\]
is the unique solution of the PDE (\ref{G-mean}).

We first prove that $u$ is locally Lipschitz in $(t,x)$. In fact,
{for each fixed }$t$,{ }$u(t,\cdot)\in ${$C_{l.Lip}(\mathbb{R)}$
}since
\begin{align*}
|\widetilde{\mathbb{E}}[\varphi(x+b_{t}%
)]-\widetilde{\mathbb{E}}[\varphi(y+b_{t})]|  &
\leq|\widetilde{\mathbb{E}}[|\varphi(x+b_{t})-\varphi(y+b_{t})|]|\\
&  \leq|\widetilde{\mathbb{E}}[C(1+|x|^{m}+|y|^{m}+|b_t|^m)|x-y|]|\\
&  \leq C_1(1+|x|^{m}+|y|^{m})|x-y|.
\end{align*}
{For each $\delta \in \lbrack0,t]$, since $b_{t}-b_{\delta}$ is
independent to $b_{\delta}$, we also have
\begin{align*}
u(t,x)  &
=\widetilde{\mathbb{E}}[\mathbb{\varphi(}x+b_{\delta}+(b_{t}-
b_{\delta})]\\
&
=\widetilde{\mathbb{E}}[\widetilde{\mathbb{E}}[\varphi(y+(b_{t}-b_{\delta}))]_{y=x+b_{\delta}}]
\end{align*}
hence}%
\begin{equation}
u(t,x)=\widetilde{\mathbb{E}}[u(t-\delta,x+b_{\delta})].
\label{eq4.21}\end{equation}
{Thus}%
\begin{align*}
|u(t,x)-u(t-\delta,x)|  &
=|\widetilde{\mathbb{E}}[u(t-\delta,x+b_{\delta})-u(t-\delta,x)]|\\
&  \leq \widetilde{\mathbb{E}}[|u(t-\delta,x+b_{\delta})-u(t-\delta,x)|]\\
&  \leq \widetilde{\mathbb{E}}[C(1+|x|^{m}+|b_{\delta}|^{m})|b_{\delta}|]\\
&  \leq C_{1}(1+|x|^{m})\delta.
\end{align*}

To prove that $u$ is a viscosity solution of the PDE (\ref{G-mean})
we fix a time-space point $(t,x)\in(0,\infty)\times \mathbb{R}$ and
let $v\in C_b^{2,2}([0,\infty)\times \mathbb{R})$ be such that
$v\geq u$ and $v(t,x)=u(t,x)$. From (\ref{eq4.21}) we have
\[
v(t,x)=\widetilde{\mathbb{E}}[u(t-\delta,x+b_{\delta})]\leq
\widetilde{\mathbb{E}}[v(t-\delta,x+b_{\delta})]
\]
Therefore from Taylor's expansion
\begin{align*}
0  &  \leq \widetilde{\mathbb{E}}[v(t-\delta,x+b_{\delta})-v(t,x)]\\
&  =\widetilde{\mathbb{E}}[v(t-\delta,x+b_{\delta
})-v(t,x+b_{\delta})+(v(t,x+b_{\delta})-v(t,x))]\\
&  =\widetilde{\mathbb{E}}[-\partial_{t}v(t,x)\delta+\partial_{x}%
v(t,x)b_{\delta}+I_{\delta}]\\
&  \leq-\partial_{t}v(t,x)\delta+\widetilde{\mathbb{E}}[\partial
_{x}v(t,x)b_{\delta}]+\widetilde{\mathbb{E}%
}[I_{\delta}]\\
&  =-\partial_{t}v(t,x)\delta+2G(\partial_{x}(t,x))\delta+\widetilde
{\mathbb{E}}[I_{\delta}].
\end{align*}
where%
\begin{align*}
I_{\delta}  & =\int_{0}^{1}[-\partial_{t}v(t-\beta \delta,x+\beta
b_{\delta
})+\partial_{t}v(t,x)]d\beta \delta \\
& +\int_{0}^{1}[\partial_{x}v(t-\beta \delta,x+\beta
b_{\delta})-\partial
_{x}v(t,x)]d\beta b_{\delta}.%
\end{align*}
With the assumption that
{$\lim_{t\downarrow0}$}$\widetilde{\mathbb{E}}[b_{t}^{2}]t^{-1}=0$
we can check that \[\lim_{\delta
\downarrow0}\widetilde{\mathbb{E}}[|I_{\delta}|]\delta^{-1}=0;\]
from which we get $\partial_{t}v(t,x)-2G(\partial_{x}(t,x))\leq0$;
hence $u$ is a viscosity supersolution of (\ref{G-mean}). We can
analogously prove that $u$ is a viscosity subsolution. It follows
that $b_1$ is
$\mathcal{U}_{[\underline{\mu},\overline{\mu}]}$--distributed. The
proof is complete.
\end{proof}

\begin{corollary}
For each $t\leq T<\infty$, we have
\[
\underline{\mu}(T-t)\leq b_{T}-b_{t}%
\leq \overline{\mu}(T-t),\  \  \  \text{in
}\mathbb{L}_{G}^{1}(\mathcal{F}).\ \ \
\]

\end{corollary}

\begin{proof}
It is a direct consequence of
\[
\mathbb{\hat{E}[(}b_{T}-b_{t}-(T-t))^{+}]=\sup_{\underline{\mu}\leq \eta \leq \overline{\mu}}%
\mathbb{\hat{E}}[(\eta-\overline{\mu})^{+}(T-t)]=0
\]
and%
\[
\mathbb{\hat{E}[(}b_{T}-b_{t}-\sigma^{2}(T-t))^{-}]=\sup_{\underline{\mu}\leq \eta \leq \overline{\mu}%
}\mathbb{\hat{E}}[(\eta-\underline{\mu}T)^{-}(T-t)]=0.
\]

\end{proof}

\begin{corollary}
{ { { \label{Lem-Qua2 copy(1)} We have
\begin{equation}
\mathbb{\hat{E}}[\left \langle B^{s}\right \rangle _{t}^{2}|\mathcal{H}%
_{s}]=\mathbb{\hat{E}}[\left \langle B\right \rangle _{t}^{2}]=t^{2}%
\end{equation}
as well as%
\[
\mathbb{\hat{E}}[\left \langle B^{s}\right \rangle _{t}^{3}|\mathcal{H}%
_{s}]=\mathbb{\hat{E}}[\left \langle B\right \rangle _{t}^{2}]=t^{3}%
,\  \  \  \mathbb{\hat{E}}[\left \langle B^{s}\right \rangle _{t}^{4}%
|\mathcal{H}_{s}]=\mathbb{\hat{E}}[\left \langle B\right \rangle _{t}^{4}%
]=t^{4}.
\]
}}}
\end{corollary}

\begin{theorem}
For each $x\in \mathbb{R}$, $Z\in M_{G}^{2}(0,T)$ and $\eta \in
M_{G}^{1}(0,T)$
the process%
\[
M_{t}=x+\int_{0}^{t}Z_{s}dB_{t}+\int_{0}^{t}\eta_{s}d\left \langle
B\right \rangle _{s}-\int_{0}^{t}2G(\eta_{s})ds,\  \ t\in
\lbrack0,T]
\]
is a martingale:
\[
\mathbb{\hat{E}}[M_{t}|\mathcal{H}_{s}]=M_{s},\  \  \ 0\leq s\leq
t\leq T.
\]

\end{theorem}

\begin{proof}
Since
$\mathbb{\hat{E}}[\int_{s}^{t}Z_{r}dB_{t}|\mathcal{H}_{s}]=\mathbb{\hat
{E}}[-\int_{s}^{t}Z_{r}dB_{t}|\mathcal{H}_{s}]=0$ we only need to
prove that
\[
\bar{M}_{t}=\int_{0}^{t}\eta_{s}d\left \langle B\right \rangle _{s}-\int_{0}%
^{t}2G(\eta_{s})ds
\]
is a $G$-martingale. It suffices to consider the case where $\eta$
is a
simple process, i.e., $\eta_{t}=\sum_{k=0}^{N-1}\xi_{k}\mathbf{1}%
_{[t_{k},t_{k+1})}(t)$. In fact we only need to consider one-step
case in which we have
\begin{align*}
\mathbb{\hat{E}}[\bar{M}_{t_{k+1}}-\bar{M}_{t_{k}}|\mathcal{H}_{t_{k}}]
& =\mathbb{\hat{E}}[\xi_{k}(\left \langle B\right \rangle
_{t_{k+1}}-\left \langle
B\right \rangle _{t_{k}})-2G(\xi_{k})(t_{k+1}-t_{k})|\mathcal{H}_{t_{k}}]\\
&  =\mathbb{\hat{E}}[\xi_{k}(\left \langle B\right \rangle _{t_{k+1}%
}-\left \langle B\right \rangle
_{t_{k}})|\mathcal{H}_{t_{k}}]-2G(\xi
_{k})(t_{k+1}-t_{k})\\
&  =\xi_{k}^{+}\mathbb{\hat{E}}[\left \langle B\right \rangle _{t_{k+1}%
}-\left \langle B\right \rangle _{t_{k}}]+\xi_{k}^{-}\mathbb{\hat{E}%
}[-(\left \langle B\right \rangle _{t_{k+1}}-\left \langle B\right
\rangle
_{t_{k}})]\\
& \ \ \ -2G(\xi_{k})(t_{k+1}-t_{k})\\
&  =0.
\end{align*}

\end{proof}

\begin{problem}
For each
$\xi=\varphi(B_{t_{1}}-B_{t_{0}},\cdots,B_{t_{N}}-B_{t_{N-1}})$ we
have the following representation:%
\[
\xi=\mathbb{\hat{E}[\xi]+}\int_{0}^{T}z_{s}dB_{s}+\int_{0}^{T}\eta
_{s}d\left \langle B\right \rangle _{s}-\int_{0}^{T}2G(\eta_{s})ds.
\]
\end{problem}

\section{It\^{o}'s formula for $G$--Brownian motion}

{ { { We have the corresponding It\^{o}'s formula of $\Phi(X_{t})$
for a \textquotedblleft$G$-It\^{o} process\textquotedblright \ $X$.
For simplification we only treat the case where the function $\Phi$
is sufficiently regular. We first consider a simple situation. } } }

\begin{lemma} \label{Lem-26}Let $\Phi \in C^{2}(\mathbb{R}^{n})$ be
bounded with bounded derivatives and $\{
\partial_{x^{\mu}x^{\nu}}^{2}\Phi \}_{\mu,\nu =1}^{n}$ are uniformly
Lipschitz. Let $s\in \lbrack0,T]$ be fixed and let
$X=(X^{1},\cdots,X^{n})^{T}$ be an $n$--dimensional process on
$[s,T]$ of the form
\[
X_{t}^{\nu}=X_{s}^{\nu}+\alpha^{\nu}(t-s)+\eta^{\nu}(\left \langle
B\right \rangle _{t}-\left \langle B\right \rangle _{s})+\beta^{\nu}(B_{t}%
-B_{s}),
\]
where, for $\nu=1,\cdots,n$, $\alpha^{\nu}$, $\eta^{\nu}$ and
$\beta^{\nu}$,
are bounded elements of $L_{G}^{2}(\mathcal{F}_{s})$ and $X_{s}=(X_{s}%
^{1},\cdots,X_{s}^{n})^{T}$ is a given $\mathbb{R}^{n}$--vector in $L_{G}%
^{2}(\mathcal{F}_{s})$. Then we have
\begin{align}
\Phi(X_{t})-\Phi(X_{s})  &  =\int_{s}^{t}\partial_{x^{\nu}}\Phi(X_{u}%
)\beta^{\nu}dB_{u}+\int_{s}^{t}\partial_{x_{\nu}}\Phi(X_{u})\alpha^{\nu
}du\label{B-Ito}\\
&
+\int_{s}^{t}[D_{x^{\nu}}\Phi(X_{u})\eta^{\nu}+\frac{1}{2}\partial_{x^{\mu
}x^{\nu}}^{2}\Phi(X_{u})\beta^{\mu}\beta^{\nu}]d\left \langle
B\right \rangle _{u}.\nonumber
\end{align}
Here we use the Einstein convention, i.e., each single term with
repeated indices $\mu$ and/or $\nu$ implies the summation.
\end{lemma}

\begin{proof}
{ { { For each positive integer $N$ we set $\delta=(t-s)/N$ and take
the partition
\[
\pi_{\lbrack
s,t]}^{N}=\{t_{0}^{N},t_{1}^{N},\cdots,t_{N}^{N}\}=\{s,s+\delta
,\cdots,s+N\delta=t\}.
\]
We have
\begin{align}
\Phi(X_{t})  &  =\Phi(X_{s})+\sum_{k=0}^{N-1}[\Phi(X_{t_{k+1}^{N}}%
)-\Phi(X_{t_{k}^{N}})]\nonumber \\
&  =\Phi(X_{s})+\sum_{k=0}^{N-1}[\partial_{x^{\mu}}\Phi(X_{t_{k}^{N}%
})(X_{t_{k+1}^{N}}^{\mu}-X_{t_{k}^{N}}^{\mu})\nonumber \\
&  +\frac{1}{2}[\partial_{x^{\mu}x^{\nu}}^{2}\Phi(X_{t_{k}^{N}})(X_{t_{k+1}%
^{N}}^{\mu}-X_{t_{k}^{N}}^{\mu})(X_{t_{k+1}^{N}}^{\nu}-X_{t_{k}^{N}}^{\nu
})+\eta_{k}^{N}]] \label{Ito}%
\end{align}
where
\[
\eta_{k}^{N}=[\partial_{x^{\mu}x^{\nu}}^{2}\Phi(X_{t_{k}^{N}}+\theta
_{k}(X_{t_{k+1}^{N}}-X_{t_{k}^{N}}))-\partial_{x^{\mu}x^{\nu}}^{2}%
\Phi(X_{t_{k}^{N}})](X_{t_{k+1}^{N}}^{\mu}-X_{t_{k}^{N}}^{\mu})(X_{t_{k+1}%
^{N}}^{\nu}-X_{t_{k}^{N}}^{\nu})
\]
with $\theta_{k}\in \lbrack0,1]$. We have%
\begin{align*}
\mathbb{\hat{E}}[|\eta_{k}^{N}|]  &
=\mathbb{\hat{E}}[|[\partial_{x^{\mu
}x^{\nu}}^{2}\Phi(X_{t_{k}^{N}}+\theta_{k}(X_{t_{k+1}^{N}}-X_{t_{k}^{N}%
}))-\partial_{x^{\mu}x^{\nu}}^{2}\Phi(X_{t_{k}^{N}})](X_{t_{k+1}^{N}}^{\mu
}-X_{t_{k}^{N}}^{\mu})\times \\ & \ \ \ \times (X_{t_{k+1}^{N}}^{\nu}-X_{t_{k}^{N}}^{\nu})|]\\
&  \leq c\mathbb{\hat{E}[}|X_{t_{k+1}^{N}}-X_{t_{k}^{N}}|^{3}]\leq
C[\delta^{3}+\delta^{3/2}],
\end{align*}
where $c$ is the Lipschitz constant of $\{ \partial_{x^{\mu}x^{\nu}}^{2}%
\Phi \}_{\mu,\nu=1}^{n}$. Thus $\sum_{k}\mathbb{\hat{E}}[|\eta_{k}%
^{N}|]\rightarrow0$. The rest terms in the summation of the right
side of
(\ref{Ito}) are $\xi_{t}^{N}+\zeta_{t}^{N}$ with%
\begin{align*}
\xi_{t}^{N}  &  =\sum_{k=0}^{N-1}\{ \partial_{x^{\mu}}\Phi(X_{t_{k}^{N}%
})[\alpha^{\mu}(t_{k+1}^{N}-t_{k}^{N})+\eta^{\mu}(\left \langle
B\right \rangle _{t_{k+1}^{N}}-\left \langle B\right \rangle
_{t_{k}^{N}})+\beta^{\mu
}(B_{t_{k+1}^{N}}-B_{t_{k}^{N}})]\\
&
+\frac{1}{2}\partial_{x^{\mu}x^{\nu}}^{2}\Phi(X_{t_{k}^{N}})\beta^{\mu
}\beta^{\nu}(B_{t_{k+1}^{N}}-B_{t_{k}^{N}})(B_{t_{k+1}^{N}}-B_{t_{k}^{N}})\}
\end{align*}
and
\begin{align*}
\zeta_{t}^{N}  &  =\frac{1}{2}\sum_{k=0}^{N-1}\partial_{x^{\mu}x^{\nu}}%
^{2}\Phi(X_{t_{k}^{N}})[\alpha^{\mu}(t_{k+1}^{N}-t_{k}^{N})+\eta^{\mu
}(\left \langle B\right \rangle _{t_{k+1}^{N}}-\left \langle B\right
\rangle
_{t_{k}^{N}})]\\
&  \times \lbrack
\alpha^{\nu}(t_{k+1}^{N}-t_{k}^{N})+\eta^{\nu}(\left \langle
B\right \rangle _{t_{k+1}^{N}}-\left \langle B\right \rangle _{t_{k}^{N}})]\\
&  +\beta^{\nu}[\alpha^{\mu}(t_{k+1}^{N}-t_{k}^{N})+\eta^{\mu}(\left
\langle
B\right \rangle _{t_{k+1}^{N}}-\left \langle B\right \rangle _{t_{k}^{N}%
})](B_{t_{k+1}^{N}}-B_{t_{k}^{N}}).
\end{align*}
We observe that, for each $u\in \lbrack t_{k}^{N},t_{k+1}^{N})$
\begin{align*}
&  \mathbb{\hat{E}}[|\partial_{x^{\mu}}\Phi(X_{u})-\sum_{k=0}^{N-1}%
\partial_{x^{\mu}}\Phi(X_{t_{k}^{N}})\mathbf{I}_{[t_{k}^{N},t_{k+1}^{N}%
)}(u)|^{2}]\\
&  =\mathbb{\hat{E}}[|\partial_{x^{\mu}}\Phi(X_{u})-\partial_{x^{\mu}}%
\Phi(X_{t_{k}^{N}})|^{2}]\\
&  \leq c^{2}\mathbb{\hat{E}}[|X_{u}-X_{t_{k}^{N}}|^{2}]\leq
C[\delta +\delta^{2}].
\end{align*}
Thus $\sum_{k=0}^{N-1}\partial_{x^{\mu}}\Phi(X_{t_{k}^{N}})\mathbf{I}%
_{[t_{k}^{N},t_{k+1}^{N})}(\cdot)$ tends to
$\partial_{x^{\mu}}\Phi(X_{\cdot })$ in $M_{G}^{2}(0,T)$. Similarly,
\[
\sum_{k=0}^{N-1}\partial_{x^{\mu}x^{\nu}}^{2}\Phi(X_{t_{k}^{N}})\mathbf{I}%
_{[t_{k}^{N},t_{k+1}^{N})}(\cdot)\rightarrow \partial_{x^{\mu}x^{\nu}}^{2}%
\Phi(X_{\cdot})\text{ in \ }M_{G}^{2}(0,T).
\]
Let $N\rightarrow \infty$. From the definitions of the integrations
with respect to $dt$, $dB_{t}$ and $d\left \langle B\right \rangle
_{t}$ the limit of $\xi_{t}^{N}$ in $L_{G}^{2}(\mathcal{F}_{t})$ is
just the right hand of (\ref{B-Ito}). By the estimates of the next
remark we also have $\zeta _{t}^{N}\rightarrow0$ in
$L_{G}^{1}(\mathcal{F}_{t})$. We then have proved (\ref{B-Ito}). } }
}
\end{proof}

\begin{remark}
{ { { We have the following estimates: For $\psi^{N}\in
M_{G}^{1,0}(0,T)$ such
that $\psi_{t}^{N}=\sum_{k=0}^{N-1}\xi_{t_{k}}^{N}\mathbf{I}_{[t_{k}%
^{N},t_{k+1}^{N})}(t)$ and $\pi_{T}^{N}=\{0\leq
t_{0},\cdots,t_{N}=T\}$ with
$\lim_{N\rightarrow \infty}\mu(\pi_{T}^{N})=0$ and $\sum_{k=0}^{N-1}%
\mathbb{\hat{E}}[|\xi_{t_{k}}^{N}|](t_{k+1}^{N}-t_{k}^{N})\leq C$
for all $N=1,2,\cdots$, we have
\[
\mathbb{\hat{E}}[|\sum_{k=0}^{N-1}\xi_{k}^{N}(t_{k+1}^{N}-t_{k}^{N}%
)^{2}|]\rightarrow0,
\]
and, thanks to Lemma \ref{Lem-Qua2},%
\begin{align*}
\mathbb{\hat{E}}[|\sum_{k=0}^{N-1}\xi_{k}^{N}(\left \langle B\right
\rangle _{t_{k+1}^{N}}-\left \langle B\right \rangle
_{t_{k}^{N}})^{2}|]  &  \leq
\sum_{k=0}^{N-1}\mathbb{\hat{E}[}|\xi_{k}^{N}|\cdot \mathbb{\hat{E}%
}[(\left \langle B\right \rangle _{t_{k+1}^{N}}-\left \langle
B\right \rangle
_{t_{k}^{N}})^{2}|\mathcal{H}_{t_{k}^{N}}]]\\
&  =\sum_{k=0}^{N-1}\mathbb{\hat{E}[}|\xi_{k}^{N}|](t_{k+1}^{N}-t_{k}^{N}%
)^{2}\rightarrow0,
\end{align*}
as well as%
\begin{align*}
&  \mathbb{\hat{E}}[|\sum_{k=0}^{N-1}\xi_{k}^{N}(\left \langle
B\right \rangle
_{t_{k+1}^{N}}-\left \langle B\right \rangle _{t_{k}^{N}})(B_{t_{k+1}^{N}%
}-B_{t_{k}^{N}})|]\\
&  \leq \sum_{k=0}^{N-1}\mathbb{\hat{E}[}|\xi_{k}^{N}|]\mathbb{\hat{E}%
[}(\left \langle B\right \rangle _{t_{k+1}^{N}}-\left \langle
B\right \rangle
_{t_{k}^{N}})|B_{t_{k+1}^{N}}-B_{t_{k}^{N}}|]\\
&  \leq \sum_{k=0}^{N-1}\mathbb{\hat{E}[}|\xi_{k}^{N}|]\mathbb{\hat{E}%
[}(\left \langle B\right \rangle _{t_{k+1}^{N}}-\left \langle
B\right \rangle
_{t_{k}^{N}})^{2}]^{1/2}\mathbb{\hat{E}[}|B_{t_{k+1}^{N}}-B_{t_{k}^{N}}%
|^{2}]^{1/2}\\
&  =\sum_{k=0}^{N-1}\mathbb{\hat{E}[}|\xi_{k}^{N}|](t_{k+1}^{N}-t_{k}%
^{N})^{3/2}\rightarrow0.
\end{align*}
We also have
\begin{align*}
&  \mathbb{\hat{E}}[|\sum_{k=0}^{N-1}\xi_{k}^{N}(\left \langle
B\right \rangle
_{t_{k+1}^{N}}-\left \langle B\right \rangle _{t_{k}^{N}})(t_{k+1}^{N}-t_{k}%
^{N})|]\\
&  \leq \sum_{k=0}^{N-1}\mathbb{\hat{E}[}|\xi_{k}^{N}|(t_{k+1}^{N}-t_{k}%
^{N})\cdot \mathbb{\hat{E}}[(\left \langle B\right \rangle _{t_{k+1}^{N}%
}-\left \langle B\right \rangle _{t_{k}^{N}})|\mathcal{H}_{t_{k}^{N}}]]\\
&  =\sum_{k=0}^{N-1}\mathbb{\hat{E}[}|\xi_{k}^{N}|](t_{k+1}^{N}-t_{k}^{N}%
)^{2}\rightarrow0
\end{align*}
and
\begin{align*}
\mathbb{\hat{E}}[|\sum_{k=0}^{N-1}\xi_{k}^{N}(t_{k+1}^{N}-t_{k}^{N}%
)(B_{t_{k+1}^{N}}-B_{t_{k}^{N}})|]  &  \leq \sum_{k=0}^{N-1}\mathbb{\hat{E}%
[}|\xi_{k}^{N}|](t_{k+1}^{N}-t_{k}^{N})\mathbb{\hat{E}}[|B_{t_{k+1}^{N}%
}-B_{t_{k}^{N}}|]\\
&  =\sqrt{\frac{2}{\pi}}\sum_{k=0}^{N-1}\mathbb{\hat{E}[}|\xi_{k}%
^{N}|](t_{k+1}^{N}-t_{k}^{N})^{3/2}\rightarrow0.\
\end{align*}
} } }\endproof
\end{remark}

{ { { We now consider a more general form of It\^{o}'s formula. Consider%
\[
X_{t}^{\nu}=X_{0}^{\nu}+\int_{0}^{t}\alpha_{s}^{\nu}ds+\int_{0}^{t}\eta
_{s}^{\nu}d\left \langle B,B\right \rangle
_{s}+\int_{0}^{t}\beta_{s}^{\nu }dB_{s}.
\]
} } }

\begin{proposition}
{ { { \label{Prop-Ito}Let $\alpha^{\nu}$, $\beta^{\nu}$ and
$\eta^{\nu}$, $\nu=1,\cdots,n$, are bounded processes of
$M_{G}^{2}(0,T)$. Then for each
$t\geq0$ and $\Phi$ in $L_{G}^{2}(\mathcal{F}_{t})$ we have%
\begin{align}
\Phi(X_{t})-\Phi(X_{s})  &  =\int_{s}^{t}\partial_{x^{\nu}}\Phi(X_{u}%
)\beta_{u}^{\nu}dB_{u}+\int_{s}^{t}\partial_{x_{\nu}}\Phi(X_{u})\alpha
_{u}^{\nu}du\label{Ito-form1}\\
&  +\int_{s}^{t}[\partial_{x^{\nu}}\Phi(X_{u})\eta_{u}^{\nu}+\frac{1}%
{2}\partial_{x^{\mu}x^{\nu}}^{2}\Phi(X_{u})\beta_{u}^{\mu}\beta_{u}^{\nu
}]d\left \langle B\right \rangle _{u}.\nonumber
\end{align}
} } }
\end{proposition}

\begin{proof}
{ { { We first consider the case where $\alpha$, $\eta$ and $\beta$
are step
processes of the form%
\[
\eta_{t}(\omega)=\sum_{k=0}^{N-1}\xi_{k}(\omega)\mathbf{I}_{[t_{k},t_{k+1}%
)}(t).
\]
From Lemma \ref{Lem-26} it is clear that (\ref{Ito-form1}) holds
true. Now let
\[
X_{t}^{\nu,N}=X_{0}^{\nu}+\int_{0}^{t}\alpha_{s}^{\nu,N}ds+\int_{0}^{t}%
\eta_{s}^{\nu,N}d\left \langle B\right \rangle _{s}+\int_{0}^{t}\beta_{s}%
^{\nu,N}dB_{s},%
\]
where $\alpha^{N}$, $\eta^{N}$ and $\beta^{N}$ are uniformly bounded
step processes that converge to $\alpha$, $\eta$ and $\beta$ in
$M_{G}^{2}(0,T)$ as
$N\rightarrow \infty$, respectively. From Lemma \ref{Lem-26}%
\begin{align}
\Phi(X_{t}^{\nu,N})-\Phi(X_{0})  &
=\int_{s}^{t}\partial_{x^{\nu}}\Phi
(X_{u}^{N})\beta_{u}^{\nu,N}dB_{u}+\int_{s}^{t}\partial_{x_{\nu}}\Phi
(X_{u}^{N})\alpha_{u}^{\nu,N}du\label{N-Ito}\\
&  +\int_{s}^{t}[\partial_{x^{\nu}}\Phi(X_{u}^{N})\eta_{u}^{\nu,N}+\frac{1}%
{2}\partial_{x^{\mu}x^{\nu}}^{2}\Phi(X_{u}^{N})\beta_{u}^{\mu,N}\beta_{u}%
^{\nu,N}]d\left \langle B\right \rangle _{u}.\nonumber
\end{align}
Since%
\begin{align*}
\mathbb{\hat{E}[}|X_{t}^{\nu,N}-X_{t}^{\nu}|^{2}]  &  \leq3\mathbb{\hat{E}%
[}|\int_{0}^{t}(\alpha_{s}^{N}-\alpha_{s})ds|^{2}]+3\mathbb{\hat{E}[}|\int
_{0}^{t}(\eta_{s}^{\nu,N}-\eta_{s}^{\nu})d\left \langle B\right
\rangle
_{s}|^{2}]\\
+3\mathbb{\hat{E}[}|\int_{0}^{t}(\beta_{s}^{\nu,N}-\beta_{s}^{\nu})dB_{s}%
|^{2}]  &
\leq3\int_{0}^{T}\mathbb{\hat{E}}[(\alpha_{s}^{\nu,N}-\alpha
_{s}^{\nu})^{2}]ds+3\int_{0}^{T}\mathbb{\hat{E}}[|\eta_{s}^{\nu,N}-\eta
_{s}^{\nu}|^{2}]ds\\
&
+3\int_{0}^{T}\mathbb{\hat{E}}[(\beta_{s}^{\nu,N}-\beta_{s}^{\nu})^{2}]ds.
\end{align*}
Furthermore
\begin{align*}
\partial_{x^{\nu}}\Phi(X_{\cdot}^{N})\eta_{\cdot}^{\nu,N}+\partial_{x^{\mu
}x^{\nu}}^{2}\Phi(X_{\cdot}^{N})\beta_{\cdot}^{\mu,N}\beta_{\cdot}^{\nu,N}
& \rightarrow
\partial_{x^{\nu}}\Phi(X_{\cdot})\eta_{\cdot}^{\nu}+\partial
_{x^{\mu}x^{\nu}}^{2}\Phi(X_{\cdot})\beta_{\cdot}^{\mu}\beta_{\cdot}^{\nu};\\
\partial_{x_{\nu}}\Phi(X_{\cdot}^{N})\alpha_{\cdot}^{\nu,N}  &  \rightarrow
\partial_{x_{\nu}}\Phi(X_{\cdot})\alpha_{\cdot}^{\nu};\\
\partial_{x^{\nu}}\Phi(X_{\cdot}^{N})\beta_{\cdot}^{\nu,N}  &  \rightarrow
\partial_{x^{\nu}}\Phi(X_{\cdot})\beta_{\cdot}^{\nu}.%
\end{align*}
We then can pass limit in both sides of (\ref{N-Ito}) and get (\ref{Ito-form1}%
). } } }\newline \newpage
\end{proof}

\section{{{{Stochastic differential equations}}}}

{ { { We consider the following SDE defined on $M_{G}^{2}(0,T;\mathbb{R}^{n}%
)$:
\begin{equation}
X_{t}=X_{0}+\int_{0}^{t}b(X_{s})ds+\int_{0}^{t}h(X_{s})d\left
\langle B\right \rangle _{s}+\int_{0}^{t}\sigma(X_{s})dB_{s},\ t\in
\lbrack0,T].
\label{SDE}%
\end{equation}
where the initial condition $X_{0}\in \mathbb{R}^{n}$ is given and
$b,h,\sigma:\mathbb{R}^{n}\mapsto \mathbb{R}^{n}$ are given
Lipschitz functions, i.e., $|\varphi(x)-\varphi(x^{\prime})|\leq
K|x-x^{\prime}|$, for each $x$, $x^{\prime}\in \mathbb{R}^{n}$,
$\varphi=b$, $h$ and $\sigma$, respectively. Here the horizon
$[0,T]$ can be arbitrarily large. The solution is a process $X\in
M_{G}^{2}(0,T;\mathbb{R}^{n})$ satisfying the above SDE. We first
introduce
the following mapping on a fixed interval $[0,T]$:%
\[
\Lambda_{\cdot}(Y):=Y\in M_{G}^{2}(0,T;\mathbb{R}^{n})\longmapsto M_{G}%
^{2}(0,T;\mathbb{R}^{n})\  \
\]
by setting $\Lambda_{t}$ with
\[
\Lambda_{t}(Y)=X_{0}+\int_{0}^{t}b(Y_{s})ds+\int_{0}^{t}h(Y_{s})d\left
\langle B\right \rangle _{s}+\int_{0}^{t}\sigma(Y_{s})dB_{s},\ t\in
\lbrack0,T].
\]
} } }

{ { { We immediately have } } }

\begin{lemma}
{ { { For each $Y,Y^{\prime}\in M_{G}^{2}(0,T;\mathbb{R}^{n})$ we
have the
following estimate:%
\[
\mathbb{\hat{E}}[|\Lambda_{t}(Y)-\Lambda_{t}(Y^{\prime})|^{2}]\leq
C\int _{0}^{t}\mathbb{\hat{E}}[|Y_{s}-Y_{s}^{\prime}|^{2}]ds,\ t\in
\lbrack0,T],
\]
where $C=3K^{2}$. } } }
\end{lemma}

\begin{proof}
{ { { This is a direct consequence of the inequalities
(\ref{ine-dt}), (\ref{e2}) and (\ref{qua-ine}). } } }
\end{proof}

{ { { We now prove that SDE (\ref{SDE}) has a unique solution. By
multiplying $e^{-2Ct}$ on both sides of the above inequality and
then integrating them on
$[0,T]$ it follows that%
\begin{align*}
\int_{0}^{T}\mathbb{\hat{E}}[|\Lambda_{t}(Y)-\Lambda_{t}(Y^{\prime}%
)|^{2}]e^{-2Ct}dt  &  \leq C\int_{0}^{T}e^{-2Ct}\int_{0}^{t}\mathbb{\hat{E}%
}[|Y_{s}-Y_{s}^{\prime}|^{2}]dsdt\\
&
=C\int_{0}^{T}\int_{s}^{T}e^{-2Ct}dt\mathbb{\hat{E}}[|Y_{s}-Y_{s}^{\prime
}|^{2}]ds\\
&  =(2C)^{-1}C\int_{0}^{T}(e^{-2Cs}-e^{-2CT})\mathbb{\hat{E}}[|Y_{s}%
-Y_{s}^{\prime}|^{2}]ds.
\end{align*}
We then have
\[
\int_{0}^{T}\mathbb{\hat{E}}[|\Lambda_{t}(Y)-\Lambda_{t}(Y^{\prime}%
)|^{2}]e^{-2Ct}dt\leq \frac{1}{2}\int_{0}^{T}\mathbb{\hat{E}}[|Y_{t}%
-Y_{t}^{\prime}|^{2}]e^{-2Ct}dt.
\]
We observe that the following two norms are equivalent in $M_{G}%
^{2}(0,T;\mathbb{R}^{n})$:
\[
\int_{0}^{T}\mathbb{\hat{E}}[|Y_{t}|^{2}]dt\thicksim \int_{0}^{T}%
\mathbb{\hat{E}}[|Y_{t}|^{2}]e^{-2Ct}dt.
\]
From this estimate we can obtain that $\Lambda(Y)$ is a contraction
mapping. Consequently, we have } } }

\begin{theorem}
{ { { There exists a unique solution $X\in
M_{G}^{2}(0,T;\mathbb{R}^{n})$ of the stochastic differential
equation (\ref{SDE}). } } }
\end{theorem}

\section{Backward SDE}

We consider the following type of BSDE%
\begin{equation}
Y_{t}=\mathbb{\hat{E}}[\xi+\int_{t}^{T}f(s,Y_{s})ds|\mathcal{H}_{t}%
],\  \ t\in \lbrack0,T]. \label{BSDE}%
\end{equation}
where $\xi \in \mathbb{L}_{G}^{1}(\mathcal{F}_{T};\mathbb{R}^{m})$,
$f(t,y)\in \mathbb{M}_{G}^{1}(0,T;\mathbb{R}^{m})$, $t\in
\lbrack0,T]$,
$y\in \mathbb{R}^{m}$, are given such that%
\[
|f(t,y)-f(t,y^{\prime})|\leq k|y-y^{\prime}|.
\]
We have

\begin{theorem}
There exists a unique solution $(Y_{t})_{t\in \lbrack0,T]}\in \mathbb{M}%
^{1}(0,T)$.
\[
\Lambda_{t}(Y)=\mathbb{\hat{E}}[\xi+\int_{t}^{T}f(s,Y_{s})ds|\mathcal{H}%
_{t}]\
\]%
\begin{align*}
\int_{0}^{T}\mathbb{\hat{E}}[|\Lambda_{t}(Y)-\Lambda_{t}(Y^{\prime})|]e^{\beta
t}dt  &  \leq \int_{0}^{T}\mathbb{\hat{E}}[\int_{t}^{T}|f(s,Y_{s}%
)-f(s,Y_{s}^{\prime})|ds]e^{\beta t}dt\\
&  \leq
C\int_{0}^{T}\int_{t}^{T}\mathbb{\hat{E}}[|Y_{s}-Y_{s}^{\prime
}|]e^{\beta t}dsdt\\
&
=C\int_{0}^{T}\mathbb{\hat{E}}[|Y_{s}-Y_{s}^{\prime}|]\int_{0}^{s}e^{\beta
t}dtds\\
&  =\frac{C}{\beta}\int_{0}^{T}\mathbb{\hat{E}}[|Y_{s}-Y_{s}^{\prime
}|](e^{\beta s}-1)ds\\
&  \leq
\frac{C}{\beta}\int_{0}^{T}\mathbb{\hat{E}}[|Y_{s}-Y_{s}^{\prime
}|]e^{\beta s}ds.
\end{align*}
We choose $\beta=2C$, then%
\[
\int_{0}^{T}\mathbb{\hat{E}}[|\Lambda_{t}(Y)-\Lambda_{t}(Y^{\prime})|]e^{\beta
t}dt\leq
\frac{1}{2}\int_{0}^{T}\mathbb{\hat{E}}[|Y_{s}-Y_{s}^{\prime
}|]e^{\beta s}ds.
\]

\end{theorem}

\noindent\textbf{Excercise.} We define a deterministic function
$u=u(t,x)$ on $(t,x)\in [0,T]\times\mathbb{R}^d$ by
\[
dX_{s}^{x,t}=f(X_{s}^{t,x})ds+h(X_{s}^{t,x})d\left \langle B\right
\rangle _{s}+\sigma(X_{s}^{t,x})dB_{s},\ \ s\in[t,T], \ \ \ X_{t}^{x,t}=x%
\]%
\[
Y_{s}^{t,x}=\mathbb{\hat{E}}[\Phi(X_{T}^{t,x})+\int_{s}^{T}g(X_{r}^{t.x}%
,Y_{r}^{t,x})dr|\mathcal{H}_{s}],\  \ s\in \lbrack t,T].
\]
where $\Phi$ and $g$ are given $\mathbb{R}^m$-valued Lipschitz
functions on $\mathbb{R}^d$ and $\mathbb{R}^d\times \mathbb{R}^m$,
respectively. $u$ is a viscosity solution of a fully nonlinear
parabolic PDE. Try to find this PDE and prove this interpretation.

\chapter{Vector Valued $G$-Brownian Motion}

\section{Multidimensional $G$--normal distributions}

For a given positive integer $n$ we will denote by $(x,y)$ the
scalar product of $x$, $y\in \mathbb{R}^{n}$ and by $\left \vert
x\right \vert =(x,x)^{1/2}$ the Euclidean norm of $x$. We denote by
$\mathbb{S}(n)$ the space of all
$d\times d$ symmetric matrices. $\mathbb{S}(n)$ is a $\frac{d(d+1)}{2}%
$-dimensional Euclidean space of the norm $\left(  A,B\right)
:=$tr$[AB]$. We also denote $\mathbb{S}(n)^{+}$ subset of
non-negative matrices in $\mathbb{S}(n)$.

For a given $d$-dimensional random vector
$\xi=(\xi_{1},\cdots,\xi_{d})$ in a sublinear expectation space
$(\Omega,\mathcal{H},\mathbb{\hat{E}})$ with zero
mean uncertainty:%
\[
\mathbb{\hat{E}}[\xi]=(\mathbb{\hat{E}}[\xi_{1}],\cdots,\mathbb{\hat{E}}%
[\xi_{d}])=0,\  \mathbb{\hat{E}}[-\xi]=0.
\]
We set%
\[
G(A)=G_{\xi}(A):=\frac{1}{2}\mathbb{\hat{E}}[\left(  A\xi,\xi
\right) ],\  \ A\in \mathbb{S}(d).
\]
This function $G(\cdot):\mathbb{S}(n)\longmapsto \mathbb{R}$
characterizes the variance-uncertainty of $\xi$. Since $\xi \in
\mathcal{H}$ implies $\mathbb{\hat{E}}[|\xi|^{2}]<\infty$, there
exists a constant $C$ such that
\[
|G(A)|\leq C|A|,\  \  \  \forall A\in \mathbb{S}(n).
\]%
\[
\  \
\]

\begin{proposition}
The function $G(\cdot):\mathbb{S}(n)\longmapsto \mathbb{R}$ is a
monotonic and
sublinear mapping:\newline(a)\ $A\geq B$ $\implies \ G(A)\geq G(B)$%
;\newline(b)\ $G(\lambda
A)=\lambda^{+}G(A)+\lambda^{-}G(-A)$;\newline(c) $G(A+B)\leq
G(A)+G(B).$\newline Moreover, there exists a bounded, convex and
closed subset $\Gamma \subset \mathbb{S}^{+}(d)$ (the non-negative
elements of $\mathbb{S}(d)$) such that
\begin{equation}
G(A)=\frac{1}{2}\sup_{B\in \Gamma}\left(  A,B\right)  . \label{Gamma}%
\end{equation}

\end{proposition}

We introduce $d$-dimensional $G$--normal distribution.

\begin{definition}
In a sublinear expectation space
$(\Omega,\mathbb{\hat{E}},\mathcal{H})$ a $d$-dimensional random
vector $\xi=(\xi_{1},\cdots,\xi_{d})\in \mathcal{H}^{d}$ with is
said to be $G$-normal {distributed if for any random vector }$\zeta$
in $\mathcal{H}^{d}$ independent to $\xi$ such that $\zeta \sim \xi$ we have%
\[
a\xi+b\zeta \sim \sqrt{a^{2}+b^{2}}\xi,\  \  \  \forall a,b\geq0.
\]
Here $G(\cdot)$ is a real sublinear function defined on $\mathbb{S}(n)$ by%
\[
G(A):=\frac{1}{2}\mathbb{\hat{E}}[(A\xi,\xi)],\  \  \ A\in
\mathbb{S}(n).
\]
Since $G$ is uniquely determined by $\Gamma \subset \mathbb{S}(n)$
via (\ref{Gamma}) $\xi$ is also called $\mathcal{N}(0,\Gamma)$
distributed, denoted by $\xi \sim \mathcal{N}(0,\Gamma)$.
\end{definition}

Proposition \ref{Prop-2-10} tells us how to construct the above
$(\xi,\zeta)$ based on the distribution of $\xi$.

\begin{remark}
By the above definition, we have $\sqrt{2}\mathbb{\hat{E}}[\xi]=\mathbb{\hat{E}%
}[\xi+\zeta]=2\mathbb{\hat{E}}[\xi]$ and
$\sqrt{2}\mathbb{\hat{E}}[-\xi]=\mathbb{\hat
{E}}[-\xi-\zeta]=2\mathbb{\hat{E}}[-\xi]$ it follows that
\[
\mathbb{\hat{E}}[\xi]=\mathbb{\hat{E}}[-\xi]=0,
\]
Namely $G$--normal distributed random variable $\xi$ has no mean
uncertainty.
\end{remark}

\begin{remark}
If $\xi$ is independent to $\zeta$ and $\xi\sim \zeta$, such that
(\ref{G-normal}) satisfies, then $-\xi$ is independent to $-\zeta$
and $-\xi\sim-\zeta$. We also have $a(-\xi)+b(-\zeta)\sim
\sqrt{a^{2}+b^{2}}(-\xi)$, $a,b\geq0$. Thus
\[
\xi\sim \mathcal{N}(0;\Gamma)\  \  \  \hbox{\textit { iff }  }\  \ \
-\xi\sim \mathcal{N}(0;\Gamma).
\]

\end{remark}
 The following proposition and corollary show that
$\mathcal{N}(0;\Gamma)$ is a
uniquely defined sublinear distribution on $(\mathbb{R}^{n},C_{l.Lip}%
(\mathbb{R}^{n}))$. We will show that an {$\mathcal{N}(0;\Gamma)$
distribution is characterized, or generated, by the following
parabolic PDE }defined on $[0,\infty)\times \mathbb{R}$:

\begin{proposition}
For a given $\varphi \in C_{l.Lip}(\mathbb{R}^{d})$ the function
\[
u(t,x)=\mathbb{\hat{E}}[\varphi(x+\sqrt{t}\xi)],\ (t,x)\in \lbrack
0,\infty)\times \mathbb{R}^{n}%
\]
satisfies
\[
u(t+s,x)=\mathbb{\hat{E}}[u(t,x+\sqrt{s}\xi)]
\]
$u(t,\cdot)\in C_{l.Lip}(\mathbb{R}^{n})$ for each $t$ and $u$ is
locally H\"{o}lder in $t$. Moreover $u$ is the unique viscosity
solution of the following PDE
\begin{equation}
\frac{\partial u}{\partial t}-G(D^{2}u)=0,\  \ u|_{t=0}=\varphi
,\  \label{d-eq-heat}%
\end{equation}
here $D^{2}u$ is the Hessian matrix of $u$, i.e., $D^{2}u=(\partial
_{x^{i}x^{j}}^{2}u)_{i,j=1}^{d}$.
\end{proposition}

\begin{remark}
In the case where $\xi$ has zero variance--uncertainty, i.e., there
exists
$\gamma_{0}\in \mathbb{S}(n)^{+}$, such that%
\[
G(A)=\frac{1}{2}\left(  A,\gamma_{0}\right).
\]
The above PDE becomes a standard linear heat equation and thus, for
$G^{0}=G_{\{ \gamma_{0}^{2}\}}$, the corresponding
$G_{\xi_{0}}$--distribution is just the $d$--dimensional classical
normal distribution with zero mean and covariance $\gamma_{0}$,
i.e., $\xi_{0}\sim \mathcal{N}(0,\gamma_{0})$. In a typical case
where $\gamma_{0}=I_{d}\in \Gamma$ we have
\[
\mathbb{\hat{E}}[\varphi(\xi)]=\frac{1}{(2\pi)^{d/2}}\int_{\mathbb{R}^{d}}%
\exp[-\sum_{i=1}^{d}\frac{(x^{i})^{2}}{2}]\varphi(x)dx,\  \varphi
\in C_{l.Lip}(\mathbb{R}^{d}).
\]

\end{remark}

In the case where $\gamma_{0}\in \Gamma$ from comparison theorem of
PDE
\begin{equation}
\mathbb{\hat{F}}_{\xi}[\varphi]\geq
\mathbb{\hat{F}}_{\xi_{0}}[\varphi
],\  \forall \varphi \in C_{l.Lip}(\mathbb{R}^{d})\text{.}\  \label{d-compar}%
\end{equation}
More generally, for each subset $\Gamma^{\prime}\subset \Gamma$ let%
\[
G^{\prime}(A):=\sup_{\gamma \in \Gamma^{\prime}}\left \langle
A,\gamma \right \rangle
\]
the corresponding $G$-normal distribution
$\mathcal{N}(0,\Gamma^{\prime})$ is
dominated by $P^{G}$ in the following sense:%
\[
\mathbb{\hat{F}}_{G_{\Gamma^{\prime}}}[\varphi]-\mathbb{\hat{F}}%
_{G_{\Gamma^{\prime}}}[\psi]\leq
\mathbb{\hat{F}}_{G_{\Gamma}}[\varphi -\psi],\  \  \forall
\varphi,\psi \in C_{l.Lip}(\mathbb{R}^{d}).
\]

\begin{remark}
In the previous chapters we have discussed $1$--dimensional case
which
corresponds to $d=1$ and $\Gamma=[\underline{\sigma}^{2},\overline{\sigma}%
^{2}]\subset \mathbb{R}$. In this case the nonlinear heat equation
(\ref{d-eq-heat}) becomes%
\[
\frac{\partial u}{\partial t}-G(\partial_{xx}^{2}u)=0,\  \
u|_{t=0}=\varphi \in C_{l.Lip}(\mathbb{R}),
\]
with $G(a)=\frac{1}{2}(\overline{\sigma}^{2}a^{+}-\underline{\sigma}^{2}%
a^{-})$. In the multidimensional cases we also have the following
typical
nonlinear heat equation:%
\[
\frac{\partial u}{\partial t}-\sum_{i=1}^{d}G_{i}(\partial_{x^{i}x^{i}}%
^{2}u)=0,\  \  \ u|_{t=0}=\varphi \in C_{l.Lip}(\mathbb{R}^{d}),
\]
where
$G_{i}(a)=\frac{1}{2}(\overline{\sigma}_{i}^{2}a^{+}-\underline{\sigma
}_{i}^{2}a^{-})$ and $0\leq \underline{\sigma}_{i}\leq
\overline{\sigma}_{i}$ are given constants. This corresponds to:
\[
\Gamma=\{diag[\gamma_{1},\cdots,\gamma_{d}]:\gamma_{i}\in \lbrack
\underline{\sigma}_{i}^{2},\overline{\sigma}_{i}^{2}],\ \
i=1,\cdots,d\}.
\]

\end{remark}

$x+\sqrt{t}\xi$ is $\mathcal{N}(x,t\Gamma)$-distributed.

\begin{definition}
\label{d-Def-2}We denote
\begin{equation}
P_{t}^{G}(\varphi)(x)=\mathbb{\hat{E}}[\varphi(x+\sqrt{t}\times \xi
))=u(t,x),\  \ (t,x)\in \lbrack0,\infty)\times \mathbb{R}^{d}. \label{d-P_tx}%
\end{equation}
Since for each $\varphi \in C_{l.Lip}(\mathbb{R}^{d})$ we have the
Markov-Nisio chain rule:
\[
P_{t+s}^{G}(\varphi)(x)=P_{s}^{G}(P_{t}^{G}(\varphi)(\cdot))(x).
\]

\end{definition}

\begin{remark}
This chain rule was first established  by Nisio \cite{Nisio1} and
\cite{Nisio2} in
terms of \textquotedblleft envelope of Markovian semi-groups\textquotedblright%
. See \cite{Peng2005} for an application to generates a sublinear
expectation space.
\end{remark}

\begin{lemma}
\label{d-a-x}Let $\xi\sim\mathcal{N}(x,t\Gamma)$. Then for each
$\mathbf{a}\in \mathbb{R}^{d}$, the random variable $\left(
\mathbf{a},\xi \right)  $ is $\mathcal{N}(0,[\underline{\sigma}^{2}%
,\overline{\sigma}^{2}])$ distributed, where
\[
\overline{\sigma}^{2}=\mathbb{\hat{E}}[(\mathbf{a},\xi)^{2}],\  \
\underline {\sigma}^{2}=-\mathbb{\hat{E}}[-(\mathbf{a},\xi)^{2}].
\]

\end{lemma}

\begin{proof}
If $\zeta$ is independent to $\xi$ and $\zeta \sim \xi$ then $(\mathbf{a}%
,\zeta)$ is also independent to $(\mathbf{a},\zeta)$ and $(\mathbf{a}%
,\zeta)\sim(\mathbf{a},\xi)$. It follows from $\alpha \xi+\beta
\zeta \sim
\sqrt{\alpha^{2}+\beta^{2}}\xi$ that%
\[
\alpha(\mathbf{a},\xi)+\beta(\mathbf{a},\zeta)=(\mathbf{a},\alpha
\xi +\beta \zeta)\sim \sqrt{\alpha^{2}+\beta^{2}}(\mathbf{a},\xi).
\]
Thus $(\mathbf{a},\xi)\sim
\mathcal{N}(0,[\underline{\sigma}^{2},\overline {\sigma}^{2}])$.
\end{proof}

\section{$G$--Brownian motions under $G$--expectations}

In the rest of this paper, we set
$\Omega=C_{0}^{d}(\mathbb{R}^{+})$, the space of all
$\mathbb{R}^{d}$--valued continuous paths $(\omega_{t})_{t\in
\mathbb{R}^{+}}$, with $\omega_{0}=0$. $\Omega$ is the classical
canonical space and $\omega=(\omega _{t})_{t\geq0}$ is the
corresponding canonical process. It is well--known that in this
canonical space there exists a Wiener measure $(\mathbf{\Omega
},\mathcal{F},P)$ under which the canonical process
$B_{t}(\omega)=\omega_{t}$ is a $d$--dimensional Brownian motion.

For each fixed $T\geq0$ we consider the following space of random
variables:
\begin{align*}
L_{ip}^{0}(\mathcal{F}_{T}):=&\{X(\omega)=\varphi(\omega_{t_{1}},\cdots
,\omega_{t_{m}}):\ \forall m\geq1,\;t_{1},\cdots,t_{m}\in
\lbrack0,T],\\
& \varphi \in C_{l.Lip}(\mathbb{R}^{d\times m}) \}.
\end{align*}
It is clear that $\{\mathcal{H}_t^0\}_{t\geq0}$ constitute a family
of sub-lattices such that $\mathcal{H}_t^0\subseteq
L_{ip}^{0}(\mathcal{F}_{T})$, for $t\leq T<\infty$. $L_{ip}^{0}(\mathcal{F}%
_{t})$ representing the past history of $\omega$ at the time $t$.
It's
completion will play the same role of Brownian filtration $\mathcal{F}_{t}%
^{B}$ as in classical stochastic analysis. We also denote%
\[
L_{ip}^{0}(\mathcal{F}):=%
{\displaystyle \bigcup \limits_{n=1}^{\infty}}
L_{ip}^{0}(\mathcal{F}_{n}).
\]

\begin{remark}
It is clear that $C_{l.Lip}(\mathbb{R}^{d\times m})$ and then $L_{ip}%
^{0}(\mathcal{F}_{T})$, $L_{ip}^{0}(\mathcal{F})$ are vector
lattices. Moreover, since $\varphi,\psi \in
C_{l.Lip}(\mathbb{R}^{d\times m})$ imply $\varphi \cdot \psi \in
C_{l.Lip}(\mathbb{R}^{d\times m})$ thus $X$, $Y\in
L_{ip}^{0}(\mathcal{F}_{T})$ imply $X\cdot Y\in L_{ip}^{0}(\mathcal{F}_{T}%
)$; $X$, $Y\in L_{ip}^{0}(\mathcal{F})$ imply $X\cdot Y\in L_{ip}%
^{0}(\mathcal{F})$.
\end{remark}

We will consider the canonical space and set
$B_{t}(\omega)=\omega_{t}$, $t\in \lbrack0,\infty)$, for $\omega \in
\Omega$.

\begin{definition}
\label{d-Def-3}The $d$-dimensional process $(B_{t})_{t\geq0}$ is
called a $G$\textbf{--Brownian} \textbf{motion} under a sublinear
expectation $(\Omega,\mathcal{H},\mathbb{\hat{E}})$ defined on
$L_{ip}^{0}(\mathcal{F})$ if $B_{0}=0$ and \newline \textsl{(i) }For
each $s,t\geq0$ and $\psi \in
C_{l.Lip}(\mathbb{R}^{d})$, $B_{t}\sim B_{t+s}-B_{s}\sim \mathcal{N}%
(0,G)$;\newline \textsl{(ii) }For each $m=1,2,\cdots$, $0=t_{0}<t_{1}%
<\cdots<t_{m}<\infty$ the increment \newline \text{\ \ \ \ \
}$B_{t_{m}}-B_{t_{m-1}}$ is independent to
$B_{t_{1}}$,$\cdots,B_{t_{m-1}}$.
\end{definition}

\begin{definition}\label{d-Def-3-1}
The related conditional expectation of $\varphi(B_{t_{1}}-B_{t_{0}}%
,\cdots,B_{t_{m}}-B_{t_{m-1}})$ under $\mathcal{H}_{t_{k}}$ is defined by%
\begin{equation}
\mathbb{E}[\varphi(B_{t_{1}}-B_{t_{0}},\cdots,B_{t_{m}}-B_{t_{m}}%
)|\mathcal{H}_{t_{k}}]=\varphi_{m-k}(B_{t_{1}},\cdots,B_{t_{k}}),
\label{d-Condition}%
\end{equation}
where
$\varphi_{m-k}(x^{1},\cdots,x^{k})=\mathbb{E}[\varphi(x^{1},\cdots
,x^{k},B_{t_{k+1}}-B_{t_{k}},\cdots,B_{t_{m}}-B_{t_{k}})].$
\end{definition}
 It is proved in \cite{Peng2005} that $\mathbb{E}[\cdot]$
consistently defines a nonlinear expectation satisfying (a)--(d) on
the vector lattice $L_{ip}^{0}(\mathcal{F}_{T})$ as well as on
$L_{ip}^{0}(\mathcal{F})$. It follows that
$\mathbb{E}[|X|]$, $X\in L_{ip}^{0}(\mathcal{F}_{T})$ (respectively $L_{ip}%
^{0}(\mathcal{F})$) is a norm and thus $L_{ip}^{0}(\mathcal{F}_{T})$
(respectively, $L_{ip}^{0}(\mathcal{F})$) can be extended, under
this norm, to a Banach space. We denote this space by
\[
\mathcal{H}_{T}=L_{G}^{1}(\mathcal{F}_{T}),\  \ T\in [0,\infty),\ \ \ \text{(respectively, } \mathcal{H}=L_{G}%
^{1}(\mathcal{F})).
\]
 For each $0\leq t\leq T<\infty$, we have $L_{G}%
^{1}(\mathcal{F}_{t})\subseteq L_{G}^{1}(\mathcal{F}_{T})\subset L_{G}%
^{1}(\mathcal{F})$. In $L_{G}^{1}(\mathcal{F}_{T})$ (respectively, $L_{G}%
^{1}(\mathcal{F}_{T})$), $\mathbb{E}[\cdot]$ still satisfies
(a)--(d) in Definition \ref{Def-1}.

\begin{remark}
It is suggestive to denote $L_{ip}^0(\mathcal{F}_t)$ by $\mathcal{H}%
_{t}^{0}$,  $L_{G}^{1}(\mathcal{F}_{t})$ by $\mathcal{H}_{t}$ and $L_{G}%
^{1}(\mathcal{F})$ by $\mathcal{H}$ and thus consider the
conditional expectation $\mathbb{E}[\cdot|\mathcal{H}_{t}]$ as a
projective mapping from $\mathcal{H}$ to $\mathcal{H}_{t}$. The
notation $L_{G}^{1}(\mathcal{F}_{t})$ is due to the similarity with
$L^{1}(\Omega,\mathcal{F}_{t},P)$ in classical stochastic analysis.
\end{remark}

\begin{definition}
The expectation $\mathbb{E}[\cdot]:L_{G}^{1}(\mathcal{F})\mapsto
\mathbb{R}$ introduced through the above procedure is called
$G$\textbf{--expectation}, or $G$--Brownian expectation. The
corresponding canonical process $B$ is said to be a $G$--Brownian
motion under $\mathbb{\hat{E}}[\cdot]$.
\end{definition}

For a given $p>1$ we also denote $L_{G}^{p}(\mathcal{F})=\{X\in L_{G}%
^{1}(\mathcal{F}):\ |X|^{p}\in L_{G}^{1}(\mathcal{F})\}$. $L_{G}%
^{p}(\mathcal{F})$ is also a Banach space under the norm $\left
\Vert X\right \Vert _{p}:=(\mathbb{\hat{E}}[|X|^{p}])^{1/p}$. We
have (see Appendix)
\[
\left \Vert X+Y\right \Vert _{p}\leq \left \Vert X\right \Vert
_{p}+\left \Vert
Y\right \Vert _{p}%
\]
and, for each $X\in L_{G}^{p}$, $Y\in L_{G}^{q}(Q)$ with
$\frac{1}{p}+\frac
{1}{q}=1$,%
\[
\left \Vert XY\right \Vert =\mathbb{\hat{E}}[|XY|]\leq \left \Vert
X\right \Vert _{p}\left \Vert X\right \Vert _{q}.
\]
With this we have $\left \Vert X\right \Vert _{p}\leq \left \Vert
X\right \Vert _{p^{\prime}}$ if $p\leq p^{\prime}$.

We now consider the conditional expectation introduced in
Definition\ref{d-Def-3-1} (see (\ref{d-Condition})).
For each fixed $t=t_{k}\leq T$ the conditional expectation $\mathbb{\hat{E}%
}[\cdot|\mathcal{H}_{t}]:L_{ip}^{0}(\mathcal{F}_{T})\mapsto L_{ip}%
^{0}(\mathcal{F}_{t})$ is a continuous mapping under $\left \Vert
\cdot \right \Vert $. Indeed, we have $\mathbb{\hat{E}}[\mathbb{\hat{E}%
}[X|\mathcal{H}_{t}]]=\mathbb{\hat{E}}[X]$, $X\in
L_{ip}^{0}(\mathcal{F}_{T})$ and since $P_{t}^{G}$ is subadditive
\[
\mathbb{\hat{E}}[X|\mathcal{H}_{t}]-\mathbb{\hat{E}}[Y|\mathcal{H}_{t}%
]\leq \mathbb{\hat{E}}[X-Y|\mathcal{H}_{t}]\leq \mathbb{\hat{E}}%
[|X-Y||\mathcal{H}_{t}]
\]
We thus obtain%
\[
\mathbb{\hat{E}[\hat{E}}[X|\mathcal{H}_{t}]-\mathbb{\hat{E}}[Y|\mathcal{H}%
_{t}]]\leq \mathbb{\hat{E}}[X-Y]
\]
and
\[
\left \Vert \mathbb{\hat{E}}[X|\mathcal{H}_{t}]-\mathbb{\hat{E}}[Y|\mathcal{H}%
_{t}]\right \Vert \leq \left \Vert X-Y\right \Vert .
\]
It follows that $\mathbb{\hat{E}}[\cdot|\mathcal{H}_{t}]$ can be
also extended
as a continuous mapping
\[\mathbb{\hat{E}}[\cdot|\mathcal{H}_{t}]:L_{G}^{1}(\mathcal{F}_{T})\mapsto L_{G}%
^{1}(\mathcal{F}_{t})\]. If the above $T$ is not fixed then we can
obtain
$\mathbb{\hat{E}}[\cdot|\mathcal{H}_{t}]:L_{G}^{1}(\mathcal{F})\mapsto
L_{G}^{1}(\mathcal{F}_{t})$.

\begin{proposition}
\label{d-Prop-1-7}We list the properties of $\mathbb{\hat{E}}[\cdot
|\mathcal{H}_{t}]$, $t\in \lbrack0,T]$ that hold in $L_{ip}^{0}(\mathcal{F}%
_{T})$ and still hold for $X$, $Y\in$ $L_{G}^{1}(\mathcal{F}_{T})$%
:\newline \newline \textsl{(i)
}$\mathbb{\hat{E}}[X|\mathcal{H}_{t}]=X$, for $X\in
L_{G}^{1}(\mathcal{F}_{t})$, $t\leq T$.\newline \textsl{(ii) }If
$X\geq
Y$ then $\mathbb{\hat{E}}[X|\mathcal{H}_{t}]\geq \mathbb{\hat{E}%
}[Y|\mathcal{H}_{t}]$.\newline \textsl{(iii) }$\mathbb{\hat{E}}[X|\mathcal{H}%
_{t}]-\mathbb{\hat{E}}[Y|\mathcal{H}_{t}]\leq \mathbb{\hat{E}}[X-Y|\mathcal{H}%
_{t}].$\newline \textsl{(iv) }$\mathbb{\hat{E}}[\mathbb{\hat{E}}[X|\mathcal{H}%
_{t}]|\mathcal{H}_{s}]=\mathbb{\hat{E}}[X|\mathcal{H}_{t\wedge s}]$,
$\mathbb{\hat{E}}[\mathbb{\hat{E}}[X|\mathcal{H}_{t}]]=\mathbb{\hat{E}}%
[X].$\newline \textsl{(v)} $\mathbb{\hat{E}}[X+\eta|\mathcal{H}_{t}%
]=\mathbb{\hat{E}}[X|\mathcal{H}_{t}]+\eta$, $\eta \in L_{G}^{1}(\mathcal{F}%
_{t})$.\newline \textsl{(vi)} $\mathbb{\hat{E}}[\eta
X|\mathcal{H}_{t}]=\eta
^{+}\mathbb{\hat{E}}[X|\mathcal{H}_{t}]+\eta^{-}\mathbb{\hat{E}}%
[-X|\mathcal{H}_{t}]$ for bounded $\eta \in L_{G}^{1}(\mathcal{F}_{t}%
).$\newline \textsl{(vii)} We have the following independence:
\[
\mathbb{\hat{E}}[X|\mathcal{H}_{t}]=\mathbb{\hat{E}}[X],\medskip \
\  \forall X\in L_{G}^{1}(\mathcal{F}_{T}^{t}),\  \forall T\geq0,
\]
where $L_{G}^{1}(\mathcal{F}_{T}^{t})$ is the extension under $\left
\Vert \cdot \right \Vert $ of $L_{ip}^{0}(\mathcal{F}_{T}^{t})$
which
consists of random variables of the form $\varphi(B_{t_{1}}^{t},B_{t_{2}}%
^{t},\cdots,B_{t_{m}}^{t})$, $\varphi \in
C_{l.Lip}(\mathbb{R}^{m})$,
$t_{1},\cdots,t_{m}\in \lbrack0,T]$, $m=1,2,\cdots$. Here we denote%
\[
B_{s}^{t}=B_{t+s}-B_{t},\  \ s\geq0.
\]
\newline \textsl{(viii)} The increments of $B$ are identically distributed:%
\[
\mathbb{\hat{E}}[\varphi(B_{t_{1}}^{t},B_{t_{2}}^{t},\cdots,B_{t_{m}}%
^{t})]=\mathbb{\hat{E}}[\varphi(B_{t_{1}},B_{t_{2}},\cdots,B_{t_{m}})].
\]

\end{proposition}

The meaning of the independence in \textbf{(vii)} is similar to the
classical one:

\begin{definition}
An $\mathbb{R}^{n}$ valued random variable
$Y\in(L_{G}^{1}(\mathcal{F}))^{n}$ is said to be independent to
$\mathcal{H}_{t}$ for some given $t$ if for each
$\varphi \in C_{l.Lip}(\mathbb{R}^{n})$ we have%
\[
\mathbb{\hat{E}}[\varphi(Y)|\mathcal{H}_{t}]=\mathbb{\hat{E}}[\varphi(Y)].
\]

\end{definition}

It is seen that the above property \textbf{(vii)} also holds for the
situation $X\in$ $L_{G}^{1}(\mathcal{F}^{t})$ where
$L_{G}^{1}(\mathcal{F}^{t})$ is the completion of the sub-lattice
$\cup_{T\geq0}L_{G}^{1}(\mathcal{F}_{T}^{t})$ under $\left \Vert
\cdot \right \Vert $.

From the above results we have

\begin{proposition}
For each fixed $t\geq0$ $(B_{s}^{t})_{s\geq0}$ is a $G$--Brownian
motion in
$L_{G}^{1}(\mathcal{F}^{t})$ under the same $G$--expectation $\mathbb{\hat{E}%
}[\cdot]$.
\end{proposition}

\begin{remark}
We can also prove that the time scaling of $B$, i.e.,
$\tilde{B}=(\sqrt{\lambda}B_{t/\lambda})_{t\geq0}$ also constitutes
a $G$--Brownian motion.
\end{remark}

The following property is very useful.

\begin{proposition}
\label{d-E-x+y}Let $X,Y\in L_{G}^{1}(\mathcal{F})$ be such that
$\mathbb{\hat
{E}}[Y|\mathcal{H}_{t}]=-\mathbb{\hat{E}}[-Y|\mathcal{H}_{t}]$, for
some
$t\in \lbrack0,T]$. Then we have%
\[
\mathbb{\hat{E}}[X+Y|\mathcal{H}_{t}]=\mathbb{\hat{E}}[X|\mathcal{H}%
_{t}]+\mathbb{\hat{E}}[Y|\mathcal{H}_{t}].
\]
In particular, if $\mathbb{\hat{E}}[Y|\mathcal{H}_{t}]=\mathbb{\hat{E}%
}[-Y|\mathcal{H}_{t}]=0$ then $\mathbb{\hat{E}}[X+Y|\mathcal{H}%
_{t}]=\mathbb{\hat{E}}[X|\mathcal{H}_{t}]$.
\end{proposition}

\begin{proof}
This follows from the two inequalities $\mathbb{\hat{E}}[X+Y|\mathcal{H}_{t}%
]\leq
\mathbb{\hat{E}}[X|\mathcal{H}_{t}]+\mathbb{\hat{E}}[Y|\mathcal{H}_{t}]$
and
\[
\mathbb{\hat{E}}[X+Y|\mathcal{H}_{t}]\geq \mathbb{\hat{E}}[X|\mathcal{H}%
_{t}]-\mathbb{\hat{E}}[-Y|\mathcal{H}_{t}]=\mathbb{\hat{E}}[X|\mathcal{H}%
_{t}]+\mathbb{\hat{E}}[Y|\mathcal{H}_{t}]\text{.}%
\]

\end{proof}

\begin{example}
\label{d-Exm-GBM-14a}
We have%
\[
\mathbb{\hat{E}}[(AB_{t},B_{t})]=\sigma_{A}t=2G(A)t,\  \  \forall
A\in \mathbb{S}(d).
\]
More general, for each $s\leq t$ and
$\eta=(\eta^{ij})_{i,j=1}^{d}\in
L_{G}^{2}(\mathcal{F}_{s};\mathbb{S}(d))$
\begin{equation}
\mathbb{\hat{E}}[(\eta
B_{t}^{s},B_{t}^{s})|\mathcal{H}_{s}]=\sigma_{\eta
}t=2G(\eta)t,\ s,t\geq0. \label{d-eq-GMB-14a}%
\end{equation}

\end{example}

\begin{definition}
We will denote in the rest of this paper
\begin{equation}
B_{t}^{\mathbf{a}}=(\mathbf{a},B_{t}),\  \  \  \text{for each }\mathbf{a}%
=(a_{1},\cdots,a_{d})^{T}\in \mathbb{R}^{d}. \label{d-a-Brown}%
\end{equation}
Since $(\mathbf{a},B_{t})$ is normal distributed $(\mathbf{a},B_{t}
)\sim
\mathcal{N}(0,[\sigma_{\mathbf{aa}^{T}},\sigma_{-\mathbf{aa}^{T}}])$:
\[
\mathbb{\hat{E}}[\varphi(B_{t}^{\mathbf{a}})]=\mathbb{\hat{E}}[\varphi
((\mathbf{a},\sqrt{t}\xi)))
\]
Thus, according to Definition \ref{d-Def-3} for $d$--dimensional
$G$--Brownian
motion $B^{\mathbf{a}}$ forms a $1$--dimensional $G_{\mathbf{a}}%
$--\textbf{Brownian motion} for which the
$G_{\mathbf{a}}$--expectation coincides with
$\mathbb{\hat{E}}[\cdot]$.
\end{definition}

\begin{example}
\label{d-Exam-1}For each $0\leq s-t$ we have
\[
\mathbb{\hat{E}}[\psi(B_{t}-B_{s})|\mathcal{H}_{s}]=\mathbb{\hat{E}}%
[\psi(B_{t}-B_{s})].
\]
If $\varphi$ is a real convex function on $\mathbb{R}$ and at least
not
growing too fast then%
\begin{align*}
&  \mathbb{\hat{E}}[X\varphi(B_{T}^{\mathbf{a}}-B_{t}^{\mathbf{a}%
})|\mathcal{H}_{t}]\\
&  =X^{+}\mathbb{\hat{E}}[\varphi(B_{T}^{\mathbf{a}}-B_{t}^{\mathbf{a}%
})|\mathcal{H}_{t}]+X^{-}\mathbb{\hat{E}}[-\varphi(B_{T}^{\mathbf{a}}%
-B_{t}^{\mathbf{a}})|\mathcal{H}_{t}]\\
&
=\frac{X^{+}}{\sqrt{2\pi(T-t)\sigma_{\mathbf{aa}^{T}}}}\int_{-\infty
}^{\infty}\varphi(x)\exp(-\frac{x^{2}}{2(T-t)\sigma_{\mathbf{aa}^{T}}})dx\\
&
-\frac{X^{-}}{\sqrt{2\pi(T-t)|\sigma_{-\mathbf{aa}^{T}}|}}\int_{-\infty
}^{\infty}\varphi(x)\exp(-\frac{x^{2}}{2(T-t)|\sigma_{-\mathbf{aa}^{T}}|})dx.
\end{align*}
In particular, for $n=1,2,\cdots,$%
\begin{align*}
\mathbb{\hat{E}}[|B_{t}^{\mathbf{a}}-B_{s}^{\mathbf{a}}|^{n}|\mathcal{H}_{s}]
&  =\mathbb{\hat{E}}[|B_{t-s}^{\mathbf{a}}|^{n}]\\
&
=\frac{1}{\sqrt{2\pi(t-s)\sigma_{\mathbf{aa}^{T}}}}\int_{-\infty}^{\infty
}|x|^{n}\exp(-\frac{x^{2}}{2(t-s)\sigma_{\mathbf{aa}^{T}}})dx.
\end{align*}
But we have $\mathbb{\hat{E}}[-|B_{t}^{\mathbf{a}}-B_{s}^{\mathbf{a}}%
|^{n}|\mathcal{H}_{s}]=\mathbb{\hat{E}}[-|B_{t-s}^{\mathbf{a}}|^{n}]$
which is $0$ when $\sigma_{-\mathbf{aa}^{T}}=0$ and equals to
\[
\frac{-1}{\sqrt{2\pi(t-s)|\sigma_{-\mathbf{aa}^{T}}|}}\int_{-\infty}^{\infty
}|x|^{n}\exp(-\frac{x^{2}}{2(t-s)|\sigma_{-\mathbf{aa}^{T}}|})dx,\
\  \text{if }\sigma_{-\mathbf{aa}^{T}}<0.\
\]
Exactly as in the classical cases we have $\mathbb{\hat{E}}[B_{t}^{\mathbf{a}%
}-B_{s}^{\mathbf{a}}|\mathcal{H}_{s}]=0$ and
\begin{align*}
\mathbb{\hat{E}}[(B_{t}^{\mathbf{a}}-B_{s}^{\mathbf{a}})^{2}|\mathcal{H}_{s}]
&  =\sigma_{\mathbf{aa}^{T}}(t-s),\  \  \  \mathbb{\hat{E}}[(B_{t}^{\mathbf{a}%
}-B_{s}^{\mathbf{a}})^{4}|\mathcal{H}_{s}]=3\sigma_{\mathbf{aa}^{T}}%
^{2}(t-s)^{2},\\
\mathbb{\hat{E}}[(B_{t}^{\mathbf{a}}-B_{s}^{\mathbf{a}})^{6}|\mathcal{H}_{s}]
&  =15\sigma_{\mathbf{aa}^{T}}^{3}(t-s)^{3},\  \  \mathbb{\hat{E}}%
[(B_{t}^{\mathbf{a}}-B_{s}^{\mathbf{a}})^{8}|\mathcal{H}_{s}]=105\sigma
_{\mathbf{aa}^{T}}^{4}(t-s)^{4},\\
\mathbb{\hat{E}}[|B_{t}^{\mathbf{a}}-B_{s}^{\mathbf{a}}||\mathcal{H}_{s}]
&
=\frac{\sqrt{2(t-s)\sigma_{\mathbf{aa}^{T}}}}{\sqrt{\pi}},\  \  \mathbb{\hat{E}%
}[|B_{t}^{\mathbf{a}}-B_{s}^{\mathbf{a}}|^{3}|\mathcal{H}_{s}]=\frac{2\sqrt
{2}[(t-s)\sigma_{\mathbf{aa}^{T}}]^{3/2}}{\sqrt{\pi}},\\
\mathbb{\hat{E}}[|B_{t}^{\mathbf{a}}-B_{s}^{\mathbf{a}}|^{5}|\mathcal{H}_{s}]
&
=8\frac{\sqrt{2}[(t-s)\sigma_{\mathbf{aa}^{T}}]^{5/2}}{\sqrt{\pi}}.
\end{align*}

\end{example}

\begin{example}
\label{d-Exam-2}For each $n=1,2,\cdots,$ $0\leq t\leq T$ and $X\in L_{G}%
^{1}(\mathcal{F}_{t})$, we have%
\[
\mathbb{\hat{E}}[X(B_{T}^{\mathbf{a}}-B_{t}^{\mathbf{a}})|\mathcal{H}%
_{t}]=X^{+}\mathbb{\hat{E}}[(B_{T}^{\mathbf{a}}-B_{t}^{\mathbf{a}%
})|\mathcal{H}_{t}]+X^{-}\mathbb{\hat{E}}[-(B_{T}^{\mathbf{a}}-B_{t}%
^{\mathbf{a}})|\mathcal{H}_{t}]=0.
\]
This, together with Proposition \ref{d-E-x+y}, yields%
\[
\mathbb{\hat{E}}[Y+X(B_{T}^{\mathbf{a}}-B_{t}^{\mathbf{a}})|\mathcal{H}%
_{t}]=\mathbb{\hat{E}}[Y|\mathcal{H}_{t}],\  \ Y\in
L_{G}^{1}(\mathcal{F}).
\]
We also have,
\begin{align*}
\mathbb{\hat{E}}[X(B_{T}^{\mathbf{a}}-B_{t}^{\mathbf{a}})^{2}|\mathcal{H}%
_{t}]  &  =X^{+}\mathbb{\hat{E}}[(B_{T}^{\mathbf{a}}-B_{t}^{\mathbf{a}}%
)^{2}|\mathcal{H}_{t}]+X^{-}\mathbb{\hat{E}}[-(B_{T}^{\mathbf{a}}%
-B_{t}^{\mathbf{a}})^{2}|\mathcal{H}_{t}]\\
&
=[X^{+}\sigma_{\mathbf{aa}^{T}}+X^{-}\sigma_{-\mathbf{aa}^{T}}](T-t).
\end{align*}

\end{example}

\begin{remark}
\label{d-Rem-17}It is clear that we can define an expectation
$E^{0}[\cdot]$ on $L_{ip}^{0}(\mathcal{F})$ in the same way as in
Definition \ref{d-Def-3} with the standard normal distribution
$\mathcal{N}(0,1)$ in place of $\mathcal{N}(0,\Gamma)$. If $I_{d}\in
\Gamma$ then it follows from (\ref{d-compar}) that
$\mathcal{N}(0,1)$ is dominated by $\mathcal{N}(0,\Gamma )$:
\[
E^{0}[\varphi(B_{t})]-E^{0}[\psi(B_{t})]\leq
\mathbb{\hat{E}}[(\varphi -\psi)(B_{t})].
\]
Then $E^{0}[\cdot]$ can be continuously extended to
$L_{G}^{1}(\mathcal{F})$. $E^{0}[\cdot]$ is a linear expectation
under which $(B_{t})_{t\geq0}$ behaves as a classical Brownian
motion. We have
\begin{equation}
-\mathbb{\hat{E}}[-X]\leq E^{0}[X]\leq \mathbb{\hat{E}}[X],\  \
-\mathbb{\hat
{E}}[-X|\mathcal{H}_{t}]\leq E^{0}[X|\mathcal{H}_{t}]\leq \mathbb{\hat{E}%
}[X|\mathcal{H}_{t}].
\end{equation}
More generally, if $\Gamma^{\prime}\subset \Gamma$ since the
corresponding $P^{\prime}=P^{G_{\Gamma^{\prime}}}$ is dominated by
$P^{G}=P^{G_{\Gamma}}$, thus the corresponding expectation
$\mathbb{\hat{E}}^{\prime}$ is well--defined in
$L_{G}^{1}(\mathcal{F})$ and $\mathbb{\hat{E}}^{\prime}$ is
dominated by $\mathbb{\hat{E}}$:
\[
\mathbb{\hat{E}}^{\prime}[X]-\mathbb{\hat{E}}^{\prime}[Y]\leq \mathbb{\hat{E}%
}[X-Y],\  \ X,Y\in L_{G}^{1}(\mathcal{F}).
\]
Such kind of extension through the above type of domination
relations was discussed in details in \cite{Peng2005}. With this
domination we then can introduce a large kind of time consistent
linear or nonlinear expectations and the corresponding conditional
expectations, not necessarily to be positive
homogeneous and/or subadditive, as continuous functionals in $L_{G}%
^{1}(\mathcal{F})$. See Example \ref{d-Exa-AB} for a further
discussion.
\end{remark}

\begin{example}
\label{d-Exam-B2}Since
\[
\mathbb{\hat{E}}[2B_{s}^{\mathbf{a}}(B_{t}^{\mathbf{a}}-B_{s}^{\mathbf{a}%
})|\mathcal{H}_{s}]=\mathbb{\hat{E}}[-2B_{s}^{\mathbf{a}}(B_{t}^{\mathbf{a}%
}-B_{s}^{\mathbf{a}})|\mathcal{H}_{s}]=0
\]
we have
\begin{align*}
\mathbb{\hat{E}}[(B_{t}^{\mathbf{a}})^{2}-(B_{s}^{\mathbf{a}})^{2}%
|\mathcal{H}_{s}]  &  =\mathbb{\hat{E}}[(B_{t}^{\mathbf{a}}-B_{s}^{\mathbf{a}%
}+B_{s}^{\mathbf{a}})^{2}-(B_{s}^{\mathbf{a}})^{2}|\mathcal{H}_{s}]\\
&  =\mathbb{\hat{E}}[(B_{t}^{\mathbf{a}}-B_{s}^{\mathbf{a}})^{2}%
+2(B_{t}^{\mathbf{a}}-B_{s}^{\mathbf{a}})B_{s}^{\mathbf{a}}|\mathcal{H}_{s}]\\
&  =\sigma_{\mathbf{aa}^{T}}(t-s)
\end{align*}
and%
\begin{align*}
\mathbb{\hat{E}}[((B_{t}^{\mathbf{a}})^{2}-(B_{s}^{\mathbf{a}})^{2}%
)^{2}|\mathcal{H}_{s}]  &  =\mathbb{\hat{E}}[\{(B_{t}^{\mathbf{a}}%
-B_{s}^{\mathbf{a}}+B_{s}^{\mathbf{a}})^{2}-(B_{s}^{\mathbf{a}})^{2}%
\}^{2}|\mathcal{H}_{s}]\\
&  =\mathbb{\hat{E}}[\{(B_{t}^{\mathbf{a}}-B_{s}^{\mathbf{a}})^{2}%
+2(B_{t}^{\mathbf{a}}-B_{s}^{\mathbf{a}})B_{s}^{\mathbf{a}}\}^{2}%
|\mathcal{H}_{s}]\\
&  =\mathbb{\hat{E}}[(B_{t}^{\mathbf{a}}-B_{s}^{\mathbf{a}})^{4}%
+4(B_{t}^{\mathbf{a}}-B_{s}^{\mathbf{a}})^{3}B_{s}^{\mathbf{a}}+4(B_{t}%
^{\mathbf{a}}-B_{s}^{\mathbf{a}})^{2}(B_{s}^{\mathbf{a}})^{2}|\mathcal{H}%
_{s}]\\
&  \leq \mathbb{\hat{E}}[(B_{t}^{\mathbf{a}}-B_{s}^{\mathbf{a}})^{4}%
]+4\mathbb{\hat{E}}[|B_{t}^{\mathbf{a}}-B_{s}^{\mathbf{a}}|^{3}]|B_{s}%
^{\mathbf{a}}|+4\sigma_{\mathbf{aa}^{T}}(t-s)(B_{s}^{\mathbf{a}})^{2}\\
&  =3\sigma_{\mathbf{aa}^{T}}^{2}(t-s)^{2}+8\sqrt{\frac{2}{\pi}}%
[\sigma_{\mathbf{aa}^{T}}(t-s)]^{3/2}|B_{s}^{\mathbf{a}}|+4\sigma
_{\mathbf{aa}^{T}}(t-s)(B_{s}^{\mathbf{a}})^{2}%
\end{align*}

\end{example}

\section{It\^{o}'s integral of $G$--Brownian motion}

\subsection{Bochner's integral}

\begin{definition}
\label{d-Def-4}For $T\in \mathbb{R}_{+}$ a partition $\pi_{T}$ of
$[0,T]$ is a
finite ordered subset $\pi=\{t_{1},\cdots,t_{N}\}$ such that $0=t_{0}%
<t_{1}<\cdots<t_{N}=T$.
\[
\mu(\pi_{T})=\max \{|t_{i+1}-t_{i}|\ :\ i=0,1,\cdots,N-1\} \text{.}%
\]
We use $\pi_{T}^{N}=\{t_{0}^{N},t_{1}^{N},\cdots,t_{N}^{N}\}$ to
denote a sequence of partitions of $[0,T]$ such that
$\lim_{N\rightarrow \infty}\mu (\pi_{T}^{N})=0$.
\end{definition}

Let $p\geq1$ be fixed. We consider the following type of simple
processes: For
a given partition $\{t_{0},\cdots,t_{N}\}=\pi_{T}$ of $[0,T]$ we set%
\[
\eta_{t}(\omega)=\sum_{k=0}^{N-1}\xi_{k}(\omega)\mathbf{I}_{[t_{k},t_{k+1}%
)}(t),
\]
where $\xi_{k}\in L_{G}^{p}(\mathcal{F}_{t_{k}})$,
$k=0,1,2,\cdots,N-1$ are
given. The collection of these processes is denoted by $M_{G}%
^{p,0}(0,T)$.

\begin{definition}
\label{d-Def-5}For an $\eta \in M_{G}^{p,0}(0,T)$ with $\eta_{t}=\sum_{k=0}%
^{N-1}\xi_{k}(\omega)\mathbf{I}_{[t_{k},t_{k+1})}(t)$, the related
Bochner integral is
\[
\int_{0}^{T}\eta_{t}(\omega)dt=\sum_{k=0}^{N-1}\xi_{k}(\omega)(t_{k+1}%
-t_{k}).
\]

\end{definition}

\begin{remark}
We set for each $\eta \in M_{G}^{p,0}(0,T)$
\[
\mathbb{\tilde{E}}_{T}[\eta]:=\frac{1}{T}\int_{0}^{T}\mathbb{\hat{E}}[\eta
_{t}]dt=\frac{1}{T}\sum_{k=0}^{N-1}\mathbb{\hat{E}}\xi_{k}(\omega
)(t_{k+1}-t_{k}).
\]
It is easy to check that
$\mathbb{\tilde{E}}_{T}:M_{G}^{1,0}(0,T)\longmapsto \mathbb{R}$
forms a nonlinear expectation satisfying (a)--(d) of Definition
\ref{Def-1}. We then can introduce a natural norm
\[
\left \Vert \eta \right \Vert _{T}^{1}=\mathbb{\tilde{E}}_{T}[|\eta|]=\frac{1}%
{T}\int_{0}^{T}\mathbb{\hat{E}}[|\eta_{t}|]dt.
\]
Under this norm $M_{G}^{p,0}(0,T)$ can be extended to
$M_{G}^{p}(0,T)$ which is a Banach space.
\end{remark}

\begin{definition}
For each $p\geq1$ we denote by $M_{G}^{p}(0,T)$ the completion of
$M_{G}^{p,0}(0,T)$ under the norm%
\[
(\frac{1}{T}\int_{0}^{T}\left \Vert |\eta_{t}|^{p}\right \Vert
dt)^{1/p}=\left(
\frac{1}{T}\sum_{k=0}^{N-1}\mathbb{\hat{E}[}|\xi_{k}(\omega)|^{p}%
](t_{k+1}-t_{k})\right)  ^{1/p}.
\]
\end{definition}

We observe that,
\begin{equation}
\mathbb{\hat{E}}[|\frac{1}{T}\int_{0}^{T}\eta_{t}(\omega)dt|]\leq \frac{1}{T}\sum_{k=0}%
^{N-1}\left \Vert \xi_{k}(\omega)\right \Vert (t_{k+1}-t_{k})=\frac{1}{T}\int_{0}%
^{T}\mathbb{\hat{E}}[|\eta_{t}|]dt. \label{d-Bohner}%
\end{equation}
We then have

\begin{proposition}
The linear mapping
$\int_{0}^{T}\eta_{t}(\omega)dt:M_{G}^{1,0}(0,T)\mapsto
L_{G}^{1}(\mathcal{F}_{T})$ is continuous and thus can be
continuously extended to $M_{G}^{1}(0,T)\mapsto
L_{G}^{1}(\mathcal{F}_{T})$. We still denote this extended mapping
by $\int_{0}^{T}\eta_{t}(\omega)dt$, $\eta \in M_{G}^{1}(0,T)$.
\end{proposition}

Since $M_{G}^{p}(0,T)\subset M_{G}^{1}(0,T)$ for $p\geq1$, this
definition makes sense for $\eta \in M_{G}^{p}(0,T)$.

\subsection{It\^{o}'s integral of $G$--Brownian motion}

We still use $B_{t}^{\mathbf{a}}:= (\mathbf{a},B_{t})$ as in
(\ref{d-a-Brown}).

\begin{definition}
For each $\eta \in M_{G}^{2,0}(0,T)$ with the form
\[\eta_{t}(\omega)=\sum
_{k=0}^{N-1}\xi_{k}(\omega)\mathbf{I}_{[t_{k},t_{k+1})}(t)\]
 we
define
\[
I(\eta)=\int_{0}^{T}\eta(s)dB_{s}^{\mathbf{a}}:=\sum_{k=0}^{N-1}\xi
_{k}(B_{t_{k+1}}^{\mathbf{a}}-B_{t_{k}}^{\mathbf{a}})\mathbf{.}%
\]

\end{definition}

\begin{lemma}
\label{d-bdd} We have, for each $\eta \in M_{G}^{2,0}(0,T)$,
\begin{align}
\mathbb{\hat{E}}[\int_{0}^{T}\eta(s)dB_{s}^{\mathbf{a}}]  &  =0,\  \  \label{d-e1}%
\\
\mathbb{\hat{E}}[(\int_{0}^{T}\eta(s)dB_{s}^{\mathbf{a}})^{2}]  &
\leq
\sigma_{\mathbf{aa}^{T}}\int_{0}^{T}\mathbb{\hat{E}}[\eta^{2}(s)]ds.
\label{d-e2}%
\end{align}
Consequently, the linear mapping $I:M_{G}^{2,0}(0,T)\longmapsto L_{G}%
^{2}(\mathcal{F}_{T})$ is continuous and thus can be extended to
$I:M_{G}^{2}(0,T)\longmapsto L_{G}^{2}(\mathcal{F}_{T})$.
\end{lemma}

\begin{definition}
We define, for a fixed $\eta \in M_{G}^{2}(0,T)$, the stochastic
calculus
\[
\int_{0}^{T}\eta(s)dB_{s}^{\mathbf{a}}:=I(\eta).
\]
It is clear that (\ref{d-e1}) and (\ref{d-e2}) still hold for $\eta \in M_{G}%
^{2}(0,T)$.
\end{definition}

\noindent\textbf{Proof of Lemma \ref{d-bdd}. }From Example
\ref{d-Exam-2} for each $k$
\[
\mathbb{\hat{E}}\mathbf{[}\xi_{k}(B_{t_{k+1}}^{\mathbf{a}}-B_{t_{k}%
}^{\mathbf{a}})|\mathcal{H}_{t_{k}}]=0.
\]
We have%
\begin{align*}
\mathbb{\hat{E}}[\int_{0}^{T}\eta(s)dB_{s}^{\mathbf{a}}]  &
=\mathbb{\hat
{E}[}\int_{0}^{t_{N-1}}\eta(s)dB_{s}^{\mathbf{a}}+\xi_{N-1}(B_{t_{N}%
}^{\mathbf{a}}-B_{t_{N-1}}^{\mathbf{a}})]\\
&  =\mathbb{\hat{E}[}\int_{0}^{t_{N-1}}\eta(s)dB_{s}^{\mathbf{a}}%
+\mathbb{\hat{E}}\mathbf{[}\xi_{N-1}(B_{t_{N}}^{\mathbf{a}}-B_{t_{N-1}%
}^{\mathbf{a}})|\mathcal{H}_{t_{N-1}}]]\\
&  =\mathbb{\hat{E}[}\int_{0}^{t_{N-1}}\eta(s)dB_{s}^{\mathbf{a}}].
\end{align*}
We then can repeat this procedure to obtain (\ref{d-e1}). We now
prove (\ref{d-e2})
\begin{align*}
&  \mathbb{\hat{E}}[\left(
\int_{0}^{T}\eta(s)dB_{s}^{\mathbf{a}}\right)
^{2}]=\mathbb{\hat{E}[}\left(  \int_{0}^{t_{N-1}}\eta(s)dB_{s}^{\mathbf{a}%
}+\xi_{N-1}(B_{t_{N}}^{\mathbf{a}}-B_{t_{N-1}}^{\mathbf{a}})\right)  ^{2}]\\
&  =\mathbb{\hat{E}[}\left(  \int_{0}^{t_{N-1}}\eta(s)dB_{s}^{\mathbf{a}%
}\right)  ^{2}\\
&  +\mathbb{\hat{E}}[2\left(  \int_{0}^{t_{N-1}}\eta(s)dB_{s}^{\mathbf{a}%
}\right)
\xi_{N-1}(B_{t_{N}}^{\mathbf{a}}-B_{t_{N-1}}^{\mathbf{a}})+\xi
_{N-1}^{2}(B_{t_{N}}^{\mathbf{a}}-B_{t_{N-1}}^{\mathbf{a}})^{2}|\mathcal{H}%
_{t_{N-1}}]]\\
&  =\mathbb{\hat{E}[}\left(  \int_{0}^{t_{N-1}}\eta(s)dB_{s}^{\mathbf{a}%
}\right)
^{2}+\xi_{N-1}^{2}\sigma_{\mathbf{aa}^{T}}(t_{N}-t_{N-1})].
\end{align*}
Thus $\mathbb{\hat{E}}[(\int_{0}^{t_{N}}\eta(s)dB_{s}^{\mathbf{a}})^{2}%
]\leq \mathbb{\hat{E}[}\left(  \int_{0}^{t_{N-1}}\eta(s)dB_{s}^{\mathbf{a}%
}\right)  ^{2}]+\mathbb{\hat{E}}[\xi_{N-1}^{2}]\sigma_{\mathbf{aa}^{T}}%
(t_{N}-t_{N-1})$. We then repeat this procedure to deduce
\[
\mathbb{\hat{E}}[(\int_{0}^{T}\eta(s)dB_{s})^{2}]\leq \sigma_{\mathbf{aa}^{T}%
}\sum_{k=0}^{N-1}\mathbb{\hat{E}}[(\xi_{k})^{2}](t_{k+1}-t_{k})=\int_{0}%
^{T}\mathbb{\hat{E}}[(\eta(t))^{2}]dt.
\]
$\blacksquare$

We list some main properties of the It\^{o}'s integral of
$G$--Brownian motion. We denote for  $0\leq s\leq t\leq T$
\[
\int_{s}^{t}\eta_{u}dB_{u}^{\mathbf{a}}:=\int_{0}^{T}\mathbf{I}_{[s,t]}%
(u)\eta_{u}dB_{u}^{\mathbf{a}}.
\]
We have

\begin{proposition}
\label{d-Prop-Integ}Let $\eta,\theta \in M_{G}^{2}(0,T)$ and $0\leq
s\leq r\leq t\leq T$. Then in $L_{G}^{1}(\mathcal{F}_{T})$ we
have\newline\textsl{(i)}
$\int_{s}^{t}\eta_{u}dB_{u}^{\mathbf{a}}=\int_{s}^{r}\eta_{u}dB_{u}%
^{\mathbf{a}}+\int_{r}^{t}\eta_{u}dB_{u}^{\mathbf{a}}.$\newline\textsl{(ii)}
$\int
_{s}^{t}(\alpha \eta_{u}+\theta_{u})dB_{u}^{\mathbf{a}}=\alpha \int_{s}^{t}%
\eta_{u}dB_{u}^{\mathbf{a}}+\int_{s}^{t}\theta_{u}dB_{u}^{\mathbf{a}}%
,\ $if$\  \alpha$ is bounded and in
$L_{G}^{1}(\mathcal{F}_{s})$.\newline\textsl{(iii)}
$\mathbb{\hat{E}[}X+\int_{r}^{T}\eta_{u}dB_{u}^{\mathbf{a}}|\mathcal{H}%
_{s}]=\mathbb{\hat{E}[}X|\mathcal{H}_{s}]$, $\forall X\in L_{G}^{1}%
(\mathcal{F}).$\newline\textsl{(iv)} $\mathbb{\hat{E}[}(\int_{r}^{T}\eta_{u}%
dB_{u}^{\mathbf{a}})^{2}|\mathcal{H}_{s}]\leq
\sigma_{\mathbf{aa}^{T}}\int
_{r}^{T}\mathbb{\hat{E}[}|\eta_{u}|^{2}|\mathcal{H}_{s}]du.$
\end{proposition}

\subsection{Quadratic variation process of $G$--Brownian motion}

We now consider the quadratic variation process of $G$--Brownian
motion. It concentrically reflects the characteristic of the
`uncertainty' part of the $G$-Brownian motion $B$. This makes a
major difference from the classical Brownian motion.

Let $\pi_{t}^{N}$, $N=1,2,\cdots$, be a sequence of partitions of $[0,t]$. We consider%

\begin{align*}
(B_{t}^{\mathbf{a}})^{2}  &  =\sum_{k=0}^{N-1}[(B_{t_{k+1}^{N}}^{\mathbf{a}%
})^{2}-(B_{t_{k}^{N}}^{\mathbf{a}})^{2}]\\
&  =\sum_{k=0}^{N-1}2B_{t_{k}^{N}}^{\mathbf{a}}(B_{t_{k+1}^{N}}^{\mathbf{a}%
}-B_{t_{k}^{N}}^{\mathbf{a}})+\sum_{k=0}^{N-1}(B_{t_{k+1}^{N}}^{\mathbf{a}%
}-B_{t_{k}^{N}}^{\mathbf{a}})^{2}%
\end{align*}
As $\mu(\pi_{t}^{N})=\max_{0\leq k\leq N-1}(t_{k+1}^{N}-t_{k}^{N}%
)\rightarrow0$, the first term of the right side tends to $2\int_{0}^{t}%
B_{s}^{\mathbf{a}}dB_{s}^{\mathbf{a}}$. The second term must
converge. We denote its limit by $\left \langle B^{\mathbf{a}}\right
\rangle _{t}$, i.e.,
\begin{equation}
\left \langle B^{\mathbf{a}}\right \rangle _{t}=\lim_{\mu(\pi_{t}^{N}%
)\rightarrow0}\sum_{k=0}^{N-1}(B_{t_{k+1}^{N}}^{\mathbf{a}}-B_{t_{k}^{N}%
}^{\mathbf{a}})^{2}=(B_{t}^{\mathbf{a}})^{2}-2\int_{0}^{t}B_{s}^{\mathbf{a}%
}dB_{s}^{\mathbf{a}}. \label{d-quadra-def}%
\end{equation}
By the above construction $\left \langle B^{\mathbf{a}}\right
\rangle _{t}$,
$t\geq0$, is an increasing process with $\left \langle B^{\mathbf{a}%
}\right \rangle _{0}=0$. We call it the \textbf{quadratic variation
process} of the $G$--Brownian motion $B^{\mathbf{a}}$. Clearly
$\left \langle B^{\mathbf{a}}\right \rangle $ is an increasing
process. It is also clear that for each $0\leq s\leq t$ and  each
smooth real function $\psi$ such that
$\psi(\left \langle B^{\mathbf{a}}\right \rangle _{t-s})\in L_{G}^{1}%
(\mathcal{F}_{t-s})$ we have $\mathbb{\hat{E}}[\psi(\left \langle
B^{\mathbf{a}}\right \rangle _{t-s})]=\mathbb{\hat{E}}[\psi(\left
\langle B^{\mathbf{a}}\right \rangle _{t}-\left \langle
B^{\mathbf{a}}\right \rangle
_{s})]$. We also have%
\[
\left \langle B^{\mathbf{a}}\right \rangle _{t}=\left \langle B^{-\mathbf{a}%
}\right \rangle _{t}=\left \langle -B^{\mathbf{a}}\right \rangle
_{t}.
\]
It is important to keep in mind that $\left \langle
B^{\mathbf{a}}\right \rangle
_{t}$ is not a deterministic process except in the case $\sigma_{\mathbf{aa}^{T}%
}=-\sigma_{-\mathbf{aa}^{T}}$ and thus $B^{\mathbf{a}}$ becomes a
classical Brownian motion. In fact we have

\begin{lemma}
\label{d-Lem-Q1}For each $0\leq s\leq t<\infty$%
\begin{align}
\mathbb{\hat{E}}[\left \langle B^{\mathbf{a}}\right \rangle
_{t}-\left \langle
B^{\mathbf{a}}\right \rangle _{s}|\mathcal{H}_{s}]  &  =\sigma_{\mathbf{aa}%
^{T}}(t-s),\  \  \label{d-quadra}\\
\mathbb{\hat{E}}[-(\left \langle B^{\mathbf{a}}\right \rangle
_{t}-\left \langle
B^{\mathbf{a}}\right \rangle _{s})|\mathcal{H}_{s}]  &  =\sigma_{-\mathbf{aa}%
^{T}}(t-s). \label{d-quadra1}%
\end{align}

\end{lemma}

\begin{proof}
From the definition of $\left \langle B^{\mathbf{a}}\right \rangle
$,  Proposition \ref{d-Prop-Integ}-(iii) and Example \ref{d-Exam-B2}
\begin{align*}
\mathbb{\hat{E}}[\left \langle B^{\mathbf{a}}\right \rangle
_{t}-\left \langle
B^{\mathbf{a}}\right \rangle _{s}|\mathcal{H}_{s}]  &  =\mathbb{\hat{E}}%
[(B_{t}^{\mathbf{a}})^{2}-(B_{s}^{\mathbf{a}})^{2}-2\int_{s}^{t}%
B_{u}^{\mathbf{a}}dB_{u}^{\mathbf{a}}|\mathcal{H}_{s}]\\
&  =\mathbb{\hat{E}}[(B_{t}^{\mathbf{a}})^{2}-(B_{s}^{\mathbf{a}}%
)^{2}|\mathcal{H}_{s}]=\sigma_{\mathbf{aa}^{T}}(t-s).
\end{align*}
We then have (\ref{d-quadra}). (\ref{d-quadra1}) can be proved
analogously by
using the equality $\mathbb{\hat{E}}[-((B_{t}^{\mathbf{a}})^{2}-(B_{s}%
^{\mathbf{a}})^{2})|\mathcal{H}_{s}]=\sigma_{-\mathbf{aa}^{T}}(t-s)$.
\end{proof}

An interesting new phenomenon of our $G$-Brownian motion is that its
quadratic process $\left \langle B\right \rangle $ also has
independent increments. In fact, we have

\begin{lemma}
An increment of $\left \langle B^{\mathbf{a}}\right \rangle $ is the
quadratic variation of the corresponding increment of
$B^{\mathbf{a}}$, i.e., for each
fixed $s\geq0$%
\[
\left \langle B^{\mathbf{a}}\right \rangle _{t+s}-\left \langle B^{\mathbf{a}%
}\right \rangle _{s}=\left \langle (B^{s})^{\mathbf{a}}\right \rangle _{t},%
\]
where $B_{t}^{s}=B_{t+s}-B_{s}$, $t\geq0$ and $(B^{s})_{t}^{\mathbf{a}%
}=(\mathbf{a},B_{s}^{t})$.
\end{lemma}

\begin{proof}%
\begin{align*}
\left \langle B^{\mathbf{a}}\right \rangle _{t+s}-\left \langle B^{\mathbf{a}%
}\right \rangle _{s}  &  =(B_{t+s}^{\mathbf{a}})^{2}-2\int_{0}^{t+s}%
B_{u}^{\mathbf{a}}dB_{u}^{\mathbf{a}}-\left(  (B_{s}^{\mathbf{a}})^{2}%
-2\int_{0}^{s}B_{u}^{\mathbf{a}}dB_{u}^{\mathbf{a}}\right) \\
&  =(B_{t+s}^{\mathbf{a}}-B_{s}^{\mathbf{a}})^{2}-2\int_{s}^{t+s}%
(B_{u}^{\mathbf{a}}-B_{s}^{\mathbf{a}})dB_{u}^{\mathbf{a}}\\
&  =(B_{t+s}^{\mathbf{a}}-B_{s}^{\mathbf{a}})^{2}-2\int_{0}^{t}(B_{s+u}%
^{\mathbf{a}}-B_{s}^{\mathbf{a}})d(B_{s+u}^{\mathbf{a}}-B_{s}^{\mathbf{a}})\\
&  =\left \langle (B^{s})^{\mathbf{a}}\right \rangle _{t}.
\end{align*}

\end{proof}

\begin{lemma}
We have
\begin{equation}
\mathbb{\hat{E}}[\left \langle B^{\mathbf{a}}\right \rangle _{t}^{2}%
]=\mathbb{\hat{E}}[(\left \langle B^{\mathbf{a}}\right \rangle _{t+s}%
-\left \langle B^{\mathbf{a}}\right \rangle _{s})^{2}|\mathcal{H}_{s}%
]=\sigma_{\mathbf{aa}^{T}}^{2}t^{2},\  \ s,t\geq0. \label{d-Qua2}%
\end{equation}

\end{lemma}

\begin{proof}
We set $\varphi(t):=\mathbb{\hat{E}}[\left \langle
B^{\mathbf{a}}\right \rangle _{t}^{2}]$.
\begin{align*}
\varphi(t)  &  =\mathbb{\hat{E}}[\{(B_{t}^{\mathbf{a}})^{2}-2\int_{0}^{t}%
B_{u}^{\mathbf{a}}dB_{u}^{\mathbf{a}}\}^{2}]\\
&
\leq2\mathbb{\hat{E}}[(B_{t}^{\mathbf{a}})^{4}]+8\mathbb{\hat{E}}[(\int
_{0}^{t}B_{u}^{\mathbf{a}}dB_{u}^{\mathbf{a}})^{2}]\\
&
\leq6\sigma_{\mathbf{aa}^{T}}^{2}t^{2}+8\sigma_{\mathbf{aa}^{T}}\int
_{0}^{t}\mathbb{\hat{E}[(}B_{u}^{\mathbf{a}})^{2}]du\\
&  =10\sigma_{\mathbf{aa}^{T}}^{2}t^{2}.
\end{align*}
This also implies $\mathbb{\hat{E}}[(\left \langle
B^{\mathbf{a}}\right \rangle _{t}-\left \langle B^{\mathbf{a}}\right
\rangle _{s})^{2}]=\varphi
(t-s)\leq10\sigma_{\mathbf{aa}^{T}}^{2}(t-s)^{2}$. For each $s\in
\lbrack 0,t)$
\begin{align*}
\varphi(t)  &  =\mathbb{\hat{E}}[(\left \langle B^{\mathbf{a}}\right
\rangle
_{s}+\left \langle B^{\mathbf{a}}\right \rangle _{t}-\left \langle B^{\mathbf{a}%
}\right \rangle _{s})^{2}]\\
&  \leq \mathbb{\hat{E}}[(\left \langle B^{\mathbf{a}}\right \rangle _{s}%
)^{2}]+\mathbb{\hat{E}}[(\left \langle B^{\mathbf{a}}\right \rangle
_{t}-\left \langle B^{\mathbf{a}}\right \rangle _{s})^{2}]+2\mathbb{\hat{E}%
}[(\left \langle B^{\mathbf{a}}\right \rangle _{t}-\left \langle B^{\mathbf{a}%
}\right \rangle _{s})\left \langle B^{\mathbf{a}}\right \rangle _{s}]\\
&
=\varphi(s)+\varphi(t-s)+2\mathbb{\hat{E}}[\mathbb{\hat{E}}[(\left
\langle B^{\mathbf{a}}\right \rangle _{t}-\left \langle
B^{\mathbf{a}}\right \rangle
_{s})|\mathcal{H}_{s}]\left \langle B^{\mathbf{a}}\right \rangle _{s}]\\
&  =\varphi(s)+\varphi(t-s)+2\sigma_{\mathbf{aa}^{T}}^{2}s(t-s).
\end{align*}
We set $\delta_{N}=t/N$, $t_{k}^{N}=kt/N=k\delta_{N}$ for a positive
integer
$N$. From the above inequalities%
\begin{align*}
\varphi(t_{N}^{N})  &  \leq \varphi(t_{N-1}^{N})+\varphi(\delta_{N}%
)+2\sigma_{\mathbf{aa}^{T}}^{2}t_{N-1}^{N}\delta_{N}\\
&  \leq \varphi(t_{N-2}^{N})+2\varphi(\delta_{N})+2\sigma_{\mathbf{aa}^{T}}%
^{2}(t_{N-1}^{N}+t_{N-2}^{N})\delta_{N}\\
&  \cdots.\\
&
\end{align*}
We then have
\[
\varphi(t)\leq
N\varphi(\delta_{N})+2\sigma_{\mathbf{aa}^{T}}^{2}\sum
_{k=0}^{N-1}t_{k}^{N}\delta_{N}\leq10t^{2}\sigma_{\mathbf{aa}^{T}}%
^{2}/N+2\sigma_{\mathbf{aa}^{T}}^{2}\sum_{k=0}^{N-1}t_{k}^{N}\delta_{N}.
\]
Let $N\rightarrow \infty$ we have $\varphi(t)\leq2\sigma_{\mathbf{aa}^{T}}%
^{2}\int_{0}^{t}sds=\sigma_{\mathbf{aa}^{T}}^{2}t^{2}$. Thus $\mathbb{\hat{E}%
}[\left \langle B^{\mathbf{a}}\right \rangle _{t}^{2}]\leq \sigma_{\mathbf{aa}%
^{T}}^{2}t^{2}$. This, together with $\mathbb{\hat{E}}[\left \langle B^{\mathbf{a}%
}\right \rangle _{t}^{2}]\geq E^{0}[\left \langle
B^{\mathbf{a}}\right \rangle
_{t}^{2}]=\sigma_{\mathbf{aa}^{T}}^{2}t^{2}$, implies
(\ref{d-Qua2}). In the last
step, the classical normal distribution $P_{1}^{0}$, or $N(0,\gamma_{0}%
\gamma_{0}^{T})$, $\gamma_{0}\in \Gamma$, is chosen such that%
\[
tr[\gamma_{0}\gamma_{0}^{T}\mathbf{aa}^{T}]=\sigma_{\mathbf{aa}^{T}}^{2}%
=\sup_{\gamma \in \Gamma}tr[\gamma \gamma^{T}\mathbf{aa}^{T}].
\]

\end{proof}

Similarly we have
\begin{align}
\mathbb{\hat{E}}[(\left \langle B^{\mathbf{a}}\right \rangle
_{t}-\left \langle B^{\mathbf{a}}\right \rangle
_{s})^{3}|\mathcal{H}_{s}]  &  =\sigma
_{\mathbf{aa}^{T}}^{3}(t-s)^{3},\label{d-Quad3}\\
\mathbb{\hat{E}}[(\left \langle B^{\mathbf{a}}\right \rangle
_{t}-\left \langle B^{\mathbf{a}}\right \rangle
_{s})^{4}|\mathcal{H}_{s}]  &  =\sigma
_{\mathbf{aa}^{T}}^{4}(t-s)^{4}.\nonumber
\end{align}

\begin{proposition}
\label{d-Prop-temp}Let $0\leq s\leq t$, $\xi \in
L_{G}^{1}(\mathcal{F}_{s})$ and
$X\in L_{G}^{1}(\mathcal{F})$. Then%
\begin{align*}
\mathbb{\hat{E}}[X+\xi((B_{t}^{\mathbf{a}})^{2}-(B_{s}^{\mathbf{a}})^{2})]
&
=\mathbb{\hat{E}}[X+\xi(B_{t}^{\mathbf{a}}-B_{s}^{\mathbf{a}})^{2}]\\
&  =\mathbb{\hat{E}}[X+\xi(\left \langle B^{\mathbf{a}}\right
\rangle _{t}-\left \langle B^{\mathbf{a}}\right \rangle _{s})].
\end{align*}

\end{proposition}

\begin{proof}
From (\ref{d-quadra-def}) and Proposition \ref{d-E-x+y} we have%
\begin{align*}
\mathbb{\hat{E}}[X+\xi((B_{t}^{\mathbf{a}})^{2}-(B_{s}^{\mathbf{a}})^{2})]
& =\mathbb{\hat{E}}[X+\xi(\left \langle B^{\mathbf{a}}\right \rangle
_{t}-\left \langle B^{\mathbf{a}}\right \rangle _{s}+2\int_{s}^{t}%
B_{u}^{\mathbf{a}}dB_{u}^{\mathbf{a}})]\\
&  =\mathbb{\hat{E}}[X+\xi(\left \langle B^{\mathbf{a}}\right
\rangle _{t}-\left \langle B^{\mathbf{a}}\right \rangle _{s})].
\end{align*}
We also have
\begin{align*}
\mathbb{\hat{E}}[X+\xi((B_{t}^{\mathbf{a}})^{2}-(B_{s}^{\mathbf{a}})^{2})]
&
=\mathbb{\hat{E}}[X+\xi \{(B_{t}^{\mathbf{a}}-B_{s}^{\mathbf{a}})^{2}%
+2(B_{t}^{\mathbf{a}}-B_{s}^{\mathbf{a}})B_{s}^{\mathbf{a}}\}]\\
&
=\mathbb{\hat{E}}[X+\xi(B_{t}^{\mathbf{a}}-B_{s}^{\mathbf{a}})^{2}].
\end{align*}

\end{proof}

\begin{example}
\label{d-Exa-AB}We assume that in a financial market a stock price
$(S_{t})_{t\geq0}$ is observed. Let $B_{t}=\log(S_{t})$, $t\geq0$,
be a
$1$-dimensional $G$-Brownian motion $(d=1)$ with $\Gamma=[\sigma_{\ast}%
,\sigma^{\ast}]$, with fixed $\sigma_{\ast}\in
\lbrack0,\frac{1}{2})$ and $\sigma^{\ast}\in \lbrack1,\infty)$. Two
traders $\mathbf{a}$ and $\mathbf{b}$ in a same bank are using their
own statistic to price a contingent claim $X=\left \langle B\right
\rangle _{T}$ with maturity $T$. Suppose, for example, under the
probability measure $\mathbb{P}_{\mathbf{a}}$ of $\mathbf{a}$,
$B_{t}(\omega)_{t\geq0}$ is a (classical) Brownian motion whereas
under $\mathbb{P}_{\mathbf{b}}$ of $\mathbf{b}$,
$\frac{1}{2}B_{t}(\omega)_{t\geq0}$
is a Brownian motion, here $\mathbb{P}_{\mathbf{a}}$ (respectively, $\mathbb{P}%
_{\mathbf{b}}$) is a classical probability measure with its linear
expectation $\mathbb{E}_{\mathbf{a}}$ (respectively,
$\mathbb{E}_{\mathbf{b}}$ ) generated by the heat equation
$\partial_{t}u=\frac{1}{2}\partial_{xx}^{2}u$ (respectively, $\partial_{t}u=\frac{1}{4}\partial_{xx}^{2}u$). Since $\mathbb{E}_{\mathbf{a}%
}$ and $\mathbb{E}_{\mathbf{b}}$ are both dominated by
$\mathbb{\hat{E}}$ in the sense of (3), they can be both
well--defined as a linear bounded functional in
$L_{G}^{1}(\mathcal{F})$. This framework cannot be provided by just
using a classical probability space because it is known that $\left
\langle B\right \rangle _{T}=T$, $\mathbb{P}_{\mathbf{a}}$--a.s.,
and
$\left \langle B\right \rangle _{T}=\frac{T}{4}$, $\mathbb{P}_{\mathbf{b}}%
$--a.s. Thus there is no probability measure on $\Omega$ with
respect to which $\mathbb{P}_{\mathbf{a}}$ and
$\mathbb{P}_{\mathbf{b}}$ are both absolutely continuous.
Practically this sublinear expectation $\mathbb{\hat{E}}$ provides a
realistic tool of dynamic risk measure for a risk supervisor of the
traders
$\mathbf{a}$ and $\mathbf{b}$: Given a risk position $X\in L_{G}%
^{1}(\mathcal{F}_{T})$ we always have $\mathbb{\hat{E}}[-X|\mathcal{H}%
_{t}]\geq \mathbb{E}_{\mathbf{a}}[-X|\mathcal{H}_{t}]\vee \mathbb{E}%
_{\mathbf{b}}[-X|\mathcal{H}_{t}]$ for the loss $-X$ of this
position. The meaning is that the supervisor uses a more sensitive
risk measure. Clearly no linear expectation can play this role. The
subset $\Gamma$ represents the uncertainty of the volatility model
of a risk regulator. The larger the subset $\Gamma$, the bigger the
uncertainty, thus the stronger the corresponding $\mathbb{\hat{E}}$.
\newline It is worth considering to create a hierarchic and dynamic
risk control system for a bank, or a banking system, in which the
Chief Risk Officer (CRO) uses
$\mathbb{\hat{E}}=\mathbb{\hat{E}}^{G}$ for his risk measure and the
Risk Officer of the $i$th division of the bank uses
$\mathbb{\hat{E}}^{i}=\mathbb{\hat{E}}^{G_{i}}$ for his one, where
\[
G(A)=\frac{1}{2}\sup_{\gamma \in \Gamma}\text{tr}[\gamma \gamma^{T}%
A],\  \ G_{i}(A)=\frac{1}{2}\sup_{\gamma \in
\Gamma_{i}}\text{tr}[\gamma \gamma^{T}A],\  \Gamma_{i}\subset \Gamma
\text{,\ }\ i=1,\cdots,I.
\]
Thus $\mathbb{\hat{E}}^{i}$ is dominated by $\mathbb{\hat{E}}$ for
each $i$.
For a large banking system we can even consider to create $\mathbb{\hat{E}%
}^{ij}=\mathbb{\hat{E}}^{G_{ij}}$ for its $(i,j)$th sub-division.
The reason is: In general, a risk regulator's statistics and
knowledge of a specific risk position $X$ are less than a trader who
is concretely involved in the business of the financial loss $X$.
\end{example}

To define the integration of a process $\eta \in M_{G}^{1}(0,T)$
with respect
to $d\left \langle B^{\mathbf{a}}\right \rangle $ we first define a mapping:%
\[
Q_{0,T}(\eta)=\int_{0}^{T}\eta(s)d\left \langle B^{\mathbf{a}}\right
\rangle _{s}:=\sum_{k=0}^{N-1}\xi_{k}(\left \langle
B^{\mathbf{a}}\right \rangle
_{t_{k+1}}-\left \langle B^{\mathbf{a}}\right \rangle _{t_{k}}):M_{G}%
^{1,0}(0,T)\mapsto L^{1}(\mathcal{F}_{T}).
\]

\begin{lemma}
\label{d-Lem-Q2}For each $\eta \in M_{G}^{1,0}(0,T)$
\begin{equation}
\mathbb{\hat{E}}[|Q_{0,T}(\eta)|]\leq \sigma_{\mathbf{aa}^{T}}\int_{0}%
^{T}\mathbb{\hat{E}}[|\eta_{s}|]ds.\  \label{d-dA}%
\end{equation}
Thus $Q_{0,T}:M_{G}^{1,0}(0,T)\mapsto L^{1}(\mathcal{F}_{T})$ is a
continuous
linear mapping. Consequently, $Q_{0,T}$ can be uniquely extended to $M_{G}%
^{1}(0,T)$. We still denote this mapping \ by%
\[
\int_{0}^{T}\eta(s)d\left \langle B^{\mathbf{a}}\right \rangle _{s}=Q_{0,T}%
(\eta),\  \  \eta \in M_{G}^{1}(0,T)\text{.}%
\]
We still have
\begin{equation}
\mathbb{\hat{E}}[|\int_{0}^{T}\eta(s)d\left \langle
B^{\mathbf{a}}\right \rangle _{s}|]\leq
\sigma_{\mathbf{aa}^{T}}\int_{0}^{T}\mathbb{\hat{E}}[|\eta
_{s}|]ds,\  \  \forall \eta \in M_{G}^{1}(0,T)\text{.} \label{d-qua-ine}%
\end{equation}

\end{lemma}

\begin{proof}
With the help of Lemma \ref{d-Lem-Q1} (\ref{d-dA}) can be checked as follows:%
\begin{align*}
\mathbb{\hat{E}}[|\sum_{k=0}^{N-1}\xi_{k}(\left \langle B^{\mathbf{a}%
}\right \rangle _{t_{k+1}}-\left \langle B^{\mathbf{a}}\right \rangle _{t_{k}%
})|]  &  \leq \sum_{k=0}^{N-1}\mathbb{\hat{E}[}|\xi_{k}|\cdot \mathbb{\hat{E}%
}[\left \langle B^{\mathbf{a}}\right \rangle _{t_{k+1}}-\left
\langle
B^{\mathbf{a}}\right \rangle _{t_{k}}|\mathcal{H}_{t_{k}}]]\\
&  =\sum_{k=0}^{N-1}\mathbb{\hat{E}[}|\xi_{k}|]\sigma_{\mathbf{aa}^{T}%
}(t_{k+1}-t_{k})\\
&
=\sigma_{\mathbf{aa}^{T}}\int_{0}^{T}\mathbb{\hat{E}}[|\eta_{s}|]ds.
\end{align*}

\end{proof}

We have the following isometry.

\begin{proposition}
Let $\eta \in M_{G}^{2}(0,T)$. We have%
\begin{equation}
\mathbb{\hat{E}}[(\int_{0}^{T}\eta(s)dB_{s}^{\mathbf{a}})^{2}]=\mathbb{\hat
{E}}[\int_{0}^{T}\eta^{2}(s)d\left \langle B^{\mathbf{a}}\right
\rangle _{s}].
\label{d-isometry}%
\end{equation}

\end{proposition}

\begin{proof}
We first consider $\eta \in M_{G}^{2,0}(0,T)$ with the form
\[
\eta_{t}(\omega)=\sum_{k=0}^{N-1}\xi_{k}(\omega)\mathbf{I}_{[t_{k},t_{k+1}%
)}(t)
\]
and thus
$\int_{0}^{T}\eta(s)dB_{s}^{\mathbf{a}}:=\sum_{k=0}^{N-1}\xi
_{k}(B_{t_{k+1}}^{\mathbf{a}}-B_{t_{k}}^{\mathbf{a}})$\textbf{.} By
Proposition \ref{d-E-x+y} we have
\[
\mathbb{\hat{E}}[X+2\xi_{k}(B_{t_{k+1}}^{\mathbf{a}}-B_{t_{k}}^{\mathbf{a}%
})\xi_{l}(B_{t_{l+1}}^{\mathbf{a}}-B_{t_{l}}^{\mathbf{a}})]=\mathbb{\hat{E}%
}[X]\text{, for }X\in L_{G}^{1}(\mathcal{F)}\text{, }l\not =k.
\]
Therefore%
\[
\mathbb{\hat{E}}[(\int_{0}^{T}\eta(s)dB_{s}^{\mathbf{a}})^{2}]=\mathbb{\hat
{E}[}\left(  \sum_{k=0}^{N-1}\xi_{k}(B_{t_{k+1}}^{\mathbf{a}}-B_{t_{k}%
}^{\mathbf{a}})\right)  ^{2}]=\mathbb{\hat{E}[}\sum_{k=0}^{N-1}\xi_{k}%
^{2}(B_{t_{k+1}}^{\mathbf{a}}-B_{t_{k}}^{\mathbf{a}})^{2}].
\]
This, together with Proposition \ref{d-Prop-temp}, yields that
\[
\mathbb{\hat{E}}[(\int_{0}^{T}\eta(s)dB_{s}^{\mathbf{a}})^{2}]=\mathbb{\hat
{E}[}\sum_{k=0}^{N-1}\xi_{k}^{2}(\left \langle B^{\mathbf{a}}\right
\rangle _{t_{k+1}}-\left \langle B^{\mathbf{a}}\right \rangle
_{t_{k}})]=\mathbb{\hat {E}[}\int_{0}^{T}\eta^{2}(s)d\left \langle
B^{\mathbf{a}}\right \rangle _{s}].
\]
Thus (\ref{d-isometry}) holds for $\eta \in M_{G}^{2,0}(0,T)$. Then
we can continuously extend this equality to the case $\eta \in
M_{G}^{2}(0,T)$ and obtain (\ref{d-isometry}).
\end{proof}

\subsection{Mutual variation processes for $G$--Brownian motion}

Let $\mathbf{a}=(a_{1},\cdots,a_{d})^{T}$ and $\mathbf{\bar{a}}=(\bar{a}%
_{1},\cdots,\bar{a}_{d})^{T}$ be two given vectors in
$\mathbb{R}^{d}$. We
then have their quadratic variation processes $\left \langle B^{\mathbf{a}%
}\right \rangle $ and $\left \langle B^{\mathbf{\bar{a}}}\right
\rangle $. We
then can define their mutual variation processes by%
\begin{align*}
\left \langle B^{\mathbf{a}},B^{\mathbf{\bar{a}}}\right \rangle _{t}
& :=\frac{1}{4}[\left \langle
B^{\mathbf{a}}+B^{\mathbf{\bar{a}}}\right \rangle
_{t}-\left \langle B^{\mathbf{a}}-B^{\mathbf{\bar{a}}}\right \rangle _{t}]\\
&  =\frac{1}{4}[\left \langle B^{\mathbf{a}+\mathbf{\bar{a}}}\right
\rangle _{t}-\left \langle B^{\mathbf{a}-\mathbf{\bar{a}}}\right
\rangle _{t}].
\end{align*}
Since $\left \langle B^{\mathbf{a}-\mathbf{\bar{a}}}\right \rangle
=\left \langle
B^{\mathbf{\bar{a}}-\mathbf{a}}\right \rangle =\left \langle -B^{\mathbf{a}%
-\mathbf{\bar{a}}}\right \rangle $ we see that $\left \langle B^{\mathbf{a}%
},B^{\mathbf{\bar{a}}}\right \rangle _{t}=\left \langle B^{\mathbf{\bar{a}}%
},B^{\mathbf{a}}\right \rangle _{t}$. In particular we have $\left
\langle
B^{\mathbf{a}},B^{\mathbf{a}}\right \rangle =\left \langle B^{\mathbf{a}%
}\right \rangle $. Let $\pi_{t}^{N}$, $N=1,2,\cdots$, be a sequence
of
partitions of $[0,t]$. We observe that%
\[
\sum_{k=0}^{N-1}(B_{t_{k+1}^{N}}^{\mathbf{a}}-B_{t_{k}^{N}}^{\mathbf{a}%
})(B_{t_{k+1}^{N}}^{\mathbf{\bar{a}}}-B_{t_{k}^{N}}^{\mathbf{\bar{a}}}%
)=\frac{1}{4}\sum_{k=0}^{N-1}[(B_{t_{k+1}}^{\mathbf{a}+\mathbf{\bar{a}}%
}-B_{t_{k}}^{\mathbf{a}+\mathbf{\bar{a}}})^{2}-(B_{t_{k+1}}^{\mathbf{a}%
-\mathbf{\bar{a}}}-B_{t_{k}}^{\mathbf{a}-\mathbf{\bar{a}}})^{2}].
\]
Thus as $\mu(\pi_{t}^{N})\rightarrow0$ we have%
\[
\lim_{N\rightarrow0}\sum_{k=0}^{N-1}(B_{t_{k+1}^{N}}^{\mathbf{a}}-B_{t_{k}%
^{N}}^{\mathbf{a}})(B_{t_{k+1}^{N}}^{\mathbf{\bar{a}}}-B_{t_{k}^{N}%
}^{\mathbf{\bar{a}}})=\left \langle B^{\mathbf{a}},B^{\mathbf{\bar{a}}%
}\right \rangle _{t}.
\]
We also have
\begin{align*}
\left \langle B^{\mathbf{a}},B^{\mathbf{\bar{a}}}\right \rangle _{t}
& =\frac{1}{4}[\left \langle B^{\mathbf{a}+\mathbf{\bar{a}}}\right
\rangle
_{t}-\left \langle B^{\mathbf{a}-\mathbf{\bar{a}}}\right \rangle _{t}]\\
&  =\frac{1}{4}[(B_{t}^{\mathbf{a}+\mathbf{\bar{a}}})^{2}-2\int_{0}^{t}%
B_{s}^{\mathbf{a}+\mathbf{\bar{a}}}dB_{s}^{\mathbf{a}+\mathbf{\bar{a}}}%
-(B_{t}^{\mathbf{a}-\mathbf{\bar{a}}})^{2}+2\int_{0}^{t}B_{s}^{\mathbf{a}%
-\mathbf{\bar{a}}}dB_{s}^{\mathbf{a}-\mathbf{\bar{a}}}]\\
&  =B_{t}^{\mathbf{a}}B_{t}^{\mathbf{\bar{a}}}-\int_{0}^{t}B_{s}^{\mathbf{a}%
}dB_{s}^{\mathbf{\bar{a}}}-\int_{0}^{t}B_{s}^{\mathbf{\bar{a}}}dB_{s}%
^{\mathbf{a}}.
\end{align*}
Now for each $\eta \in M_{G}^{1}(0,T)$ we can consistently define
\[
\int_{0}^{T}\eta_{s}d\left \langle B^{\mathbf{a}},B^{\mathbf{\bar{a}}%
}\right \rangle _{s}=\frac{1}{4}\int_{0}^{T}\eta_{s}d\left \langle
B^{\mathbf{a}+\mathbf{\bar{a}}}\right \rangle _{s}-\frac{1}{4}\int_{0}^{T}%
\eta_{s}d\left \langle B^{\mathbf{a}-\mathbf{\bar{a}}}\right \rangle
_{s}.
\]

\begin{lemma}
\label{d-Lem-mutual}Let $\eta^{N}\in M_{G}^{1,0}(0,T)$,
$N=1,2,\cdots,$ be of form
\[
\eta_{t}^{N}(\omega)=\sum_{k=0}^{N-1}\xi_{k}^{N}(\omega)\mathbf{I}_{[t_{k}%
^{N},t_{k+1}^{N})}(t)
\]
with $\mu(\pi_{T}^{N})\rightarrow0$ and $\eta^{N}\rightarrow \eta$
in $M_{G}^{1}(0,T)$ as $N\rightarrow \infty$. Then we have the
following
convergence in $L_{G}^{1}(\mathcal{F}_{T})$:%
\begin{align*}
\int_{0}^{T}\eta^{N}(s)d\left \langle B^{\mathbf{a}},B^{\mathbf{\bar{a}}%
}\right \rangle _{s}  &  :=\sum_{k=0}^{N-1}\xi_{k}^{N}(B_{t_{k+1}^{N}%
}^{\mathbf{a}}-B_{t_{k}^{N}}^{\mathbf{a}})(B_{t_{k+1}^{N}}^{\mathbf{\bar{a}}%
}-B_{t_{k}^{N}}^{\mathbf{\bar{a}}})\\
&  \rightarrow \int_{0}^{T}\eta(s)d\left \langle
B^{\mathbf{a}},B^{\mathbf{\bar {a}}}\right \rangle _{s}.
\end{align*}

\end{lemma}

For notational simplification we denote $B^{i}=B^{\mathbf{e}_{i}}$,
the $i$-th coordinate of the $G$--Brownian motion $B$, under a given
orthonormal basis $(\mathbf{e}_{1},\cdots,\mathbf{e}_{d})$ of
$\mathbb{R}^{d}$. We denote
by%
\[
(\left \langle B\right \rangle _{t})_{ij}=\left \langle
B^{i},B^{j}\right \rangle
_{t}.%
\]
$\left \langle B\right \rangle _{t}$, $t\geq0$, is an
$\mathbb{S}(n)$-valued
process. Since%
\[
=
\]
We have%
\[
\mathbb{\hat{E}}[(\left \langle B\right \rangle _{t},A)]=2G(A)t,\  \
A\in \mathbb{S}(n).
\]
Now we set a function
\[
v(t,X)=\mathbb{\hat{E}}[\varphi(X+\left \langle B\right \rangle _{t}%
)],\  \ (t,X)\in \lbrack0,\infty)\times \mathbb{S}(n).
\]

\begin{proposition}
$v$ solves the following first order PDE:
\[
\partial_{t}v=2G(Dv),\ v|_{t=0}=\varphi,\  \
\]
where $Dv=(\partial_{x_{ij}})_{i,j=1}^{d}$. We also have%
\[
v(t,X)=\sup_{\Lambda \in \Gamma}\varphi(X+t\Lambda).
\]

\end{proposition}

\begin{proof}
(Sketch) We have%
\begin{align*}
v(t+\delta,X)  &  =\mathbb{\hat{E}}[\varphi(X+\left \langle B\right
\rangle _{\delta}+\left \langle B\right \rangle _{t+\delta}-\left
\langle B\right \rangle
_{\delta})]\\
&  =\mathbb{\hat{E}}[v(t,X+\left \langle B\right \rangle
_{\delta})].
\end{align*}
The rest part of the proof is as in one dimensional case.
\end{proof}

\begin{corollary}
We have%
\[
\left \langle B\right \rangle _{t}\in t\Gamma:=\{t\times
\gamma:\gamma \in \Gamma\}.
\]
or, equivalently $d_{t\Gamma}(\left \langle B\right \rangle
_{t})=0$, where $d_{U}(x)=\inf \{|x-y|:y\in U\}.$
\end{corollary}

\begin{proof}
Since
\[
\mathbb{\hat{E}}[d_{t\Gamma}(\left \langle B\right \rangle
_{t})]=\sup
_{\Lambda \in \Gamma}d_{t\Gamma}(t\Lambda)=0\text{. }%
\]
It follows that $d_{t\Gamma}(\left \langle B\right \rangle _{t})=0$
in $L_{G}^{1}(\mathcal{F})$.
\end{proof}

\subsection{It\^{o}'s formula for $G$--Brownian motion}

We have the corresponding It\^{o}'s formula of $\Phi(X_{t})$ for a
\textquotedblleft$G$-It\^{o} process\textquotedblright \ $X$. For
simplification we only treat the case where the function $\Phi$ is
sufficiently regular.

\begin{lemma}
\label{d-Lem-26}Let $\Phi \in C^{2}(\mathbb{R}^{n})$ be bounded with
bounded derivatives and $\{ \partial_{x^{\mu}x^{\nu}}^{2}\Phi
\}_{\mu,\nu=1}^{n}$ are
uniformly Lipschitz. Let $s\in \lbrack0,T]$ be fixed and let $X=(X^{1}%
,\cdots,X^{n})^{T}$ be an $n$--dimensional process on $[s,T]$ of the
form
\[
X_{t}^{\nu}=X_{s}^{\nu}+\alpha^{\nu}(t-s)+\eta^{\nu ij}(\left
\langle B^{i},B^{j}\right \rangle _{t}-\left \langle
B^{i},B^{j}\right \rangle _{s})+\beta^{\nu j}(B_{t}^{j}-B_{s}^{j}),
\]
where, for $\nu=1,\cdots,n$, $i,j=1,\cdots,d$, $\alpha^{\nu}$,
$\eta^{\nu ij}$ and $\beta^{\nu ij}$ are bounded elements of
$L_{G}^{2}(\mathcal{F}_{s})$ and
$X_{s}=(X_{s}^{1},\cdots,X_{s}^{n})^{T}$ is a given
$\mathbb{R}^{n}$--vector in $L_{G}^{2}(\mathcal{F}_{s})$. Then we
have
\begin{align}
\Phi(X_{t})-\Phi(X_{s})  &  =\int_{s}^{t}\partial_{x^{\nu}}\Phi(X_{u}%
)\beta^{\nu
j}dB_{u}^{j}+\int_{s}^{t}\partial_{x_{\nu}}\Phi(X_{u})\alpha^{\nu
}du\label{d-B-Ito}\\
&  +\int_{s}^{t}[\partial_{x^{\nu}}\Phi(X_{u})\eta^{\nu ij}+\frac{1}%
{2}\partial_{x^{\mu}x^{\nu}}^{2}\Phi(X_{u})\beta^{\nu i}\beta^{\nu
j}]d\left \langle B^{i},B^{j}\right \rangle _{u}.\nonumber
\end{align}
Here we use the Einstein convention, i.e., the above repeated
indices $\mu ,\nu$, $i$ and $j$ (but not $k$) imply the summation.
\end{lemma}

\begin{proof}
For each positive integer $N$ we set $\delta=(t-s)/N$ and take the
partition
\[
\pi_{\lbrack
s,t]}^{N}=\{t_{0}^{N},t_{1}^{N},\cdots,t_{N}^{N}\}=\{s,s+\delta
,\cdots,s+N\delta=t\}.
\]
We have
\begin{align}
\Phi(X_{t})-\Phi(X_{s})  &  =\sum_{k=0}^{N-1}[\Phi(X_{t_{k+1}^{N}}%
)-\Phi(X_{t_{k}^{N}})]\nonumber \\
&  =\sum_{k=0}^{N-1}[\partial_{x^{\mu}}\Phi(X_{t_{k}^{N}})(X_{t_{k+1}^{N}%
}^{\mu}-X_{t_{k}^{N}}^{\mu})\nonumber \\
&  +\frac{1}{2}[\partial_{x^{\mu}x^{\nu}}^{2}\Phi(X_{t_{k}^{N}})(X_{t_{k+1}%
^{N}}^{\mu}-X_{t_{k}^{N}}^{\mu})(X_{t_{k+1}^{N}}^{\nu}-X_{t_{k}^{N}}^{\nu
})+\eta_{k}^{N}]] \label{d-Ito}%
\end{align}
where
\[
\eta_{k}^{N}=[\partial_{x^{\mu}x^{\nu}}^{2}\Phi(X_{t_{k}^{N}}+\theta
_{k}(X_{t_{k+1}^{N}}-X_{t_{k}^{N}}))-\partial_{x^{\mu}x^{\nu}}^{2}%
\Phi(X_{t_{k}^{N}})](X_{t_{k+1}^{N}}^{\mu}-X_{t_{k}^{N}}^{\mu})(X_{t_{k+1}%
^{N}}^{\nu}-X_{t_{k}^{N}}^{\nu})
\]
with $\theta_{k}\in \lbrack0,1]$. We have%
\begin{align*}
\mathbb{\hat{E}}[|\eta_{k}^{N}|]  &
=\mathbb{\hat{E}}[|[\partial_{x^{\mu
}x^{\nu}}^{2}\Phi(X_{t_{k}^{N}}+\theta_{k}(X_{t_{k+1}^{N}}-X_{t_{k}^{N}%
}))-\partial_{x^{\mu}x^{\nu}}^{2}\Phi(X_{t_{k}^{N}})]\\
&
\times(X_{t_{k+1}^{N}}^{\mu}-X_{t_{k}^{N}}^{\mu})(X_{t_{k+1}^{N}}^{\nu
}-X_{t_{k}^{N}}^{\nu})|]\\
&  \leq c\mathbb{\hat{E}[}|X_{t_{k+1}^{N}}-X_{t_{k}^{N}}|^{3}]\leq
C[\delta^{3}+\delta^{3/2}],
\end{align*}
where $c$ is the Lipschitz constant of $\{ \partial_{x^{\mu}x^{\nu}}^{2}%
\Phi \}_{\mu,\nu=1}^{d}$. In the last step we use Example
\ref{d-Exam-1} and (\ref{d-Quad3}). Thus
$\sum_{k}\mathbb{\hat{E}}[|\eta_{k}^{N}|]\rightarrow0$. The rest
terms in the summation of the right side of (\ref{d-Ito}) are $\xi
_{t}^{N}+\zeta_{t}^{N}$ with%
\begin{align*}
\xi_{t}^{N}  &  =\sum_{k=0}^{N-1}\{ \partial_{x^{\mu}}\Phi(X_{t_{k}^{N}%
})[\alpha^{\mu}(t_{k+1}^{N}-t_{k}^{N})+\eta^{\mu ij}(\left \langle B^{i}%
,B^{j}\right \rangle _{t_{k+1}^{N}}-\left \langle B^{i},B^{j}\right
\rangle
_{t_{k}^{N}})\\
&  +\beta^{\mu j}(B_{t_{k+1}^{N}}^{j}-B_{t_{k}^{N}}^{j})]+\frac{1}{2}%
\partial_{x^{\mu}x^{\nu}}^{2}\Phi(X_{t_{k}^{N}})\beta^{\mu i}\beta^{\nu
j}(B_{t_{k+1}^{N}}^{i}-B_{t_{k}^{N}}^{i})(B_{t_{k+1}^{N}}^{j}-B_{t_{k}^{N}%
}^{j})\} \\
&
\end{align*}
and
\begin{align*}
\zeta_{t}^{N}  &  =\frac{1}{2}\sum_{k=0}^{N-1}\partial_{x^{\mu}x^{\nu}}%
^{2}\Phi(X_{t_{k}^{N}})[\alpha^{\mu}(t_{k+1}^{N}-t_{k}^{N})+\eta^{\mu
ij}(\left \langle B^{i},B^{j}\right \rangle _{t_{k+1}^{N}}-\left
\langle
B^{i},B^{j}\right \rangle _{t_{k}^{N}})]\\
&  \times \lbrack \alpha^{\nu}(t_{k+1}^{N}-t_{k}^{N})+\eta^{\nu
lm}(\left \langle B^{l},B^{m}\right \rangle _{t_{k+1}^{N}}-\left
\langle B^{l},B^{m}\right \rangle
_{t_{k}^{N}})]\\
&  +[\alpha^{\mu}(t_{k+1}^{N}-t_{k}^{N})+\eta^{\mu ij}(\left \langle
B^{i},B^{j}\right \rangle _{t_{k+1}^{N}}-\left \langle
B^{i},B^{j}\right \rangle _{t_{k}^{N}})]\beta^{\nu
l}(B_{t_{k+1}^{N}}^{l}-B_{t_{k}^{N}}^{l}).
\end{align*}
We observe that, for each $u\in \lbrack t_{k}^{N},t_{k+1}^{N})$
\begin{align*}
&  \mathbb{\hat{E}}[|\partial_{x^{\mu}}\Phi(X_{u})-\sum_{k=0}^{N-1}%
\partial_{x^{\mu}}\Phi(X_{t_{k}^{N}})\mathbf{I}_{[t_{k}^{N},t_{k+1}^{N}%
)}(u)|^{2}]\\
&  =\mathbb{\hat{E}}[|\partial_{x^{\mu}}\Phi(X_{u})-\partial_{x^{\mu}}%
\Phi(X_{t_{k}^{N}})|^{2}]\\
&  \leq c^{2}\mathbb{\hat{E}}[|X_{u}-X_{t_{k}^{N}}|^{2}]\leq
C[\delta +\delta^{2}].
\end{align*}
Thus $\sum_{k=0}^{N-1}\partial_{x^{\mu}}\Phi(X_{t_{k}^{N}})\mathbf{I}%
_{[t_{k}^{N},t_{k+1}^{N})}(\cdot)$ tends to
$\partial_{x^{\mu}}\Phi(X_{\cdot })$ in $M_{G}^{2}(0,T)$. Similarly,
\[
\sum_{k=0}^{N-1}\partial_{x^{\mu}x^{\nu}}^{2}\Phi(X_{t_{k}^{N}})\mathbf{I}%
_{[t_{k}^{N},t_{k+1}^{N})}(\cdot)\rightarrow \partial_{x^{\mu}x^{\nu}}^{2}%
\Phi(X_{\cdot})\text{ in \ }M_{G}^{2}(0,T).
\]
Let $N\rightarrow \infty$. From Lemma \ref{d-Lem-mutual} as well as
the definitions of the integrations of $dt$, $dB_{t}$ and $d\left
\langle B\right \rangle _{t}$ the limit of $\xi_{t}^{N}$ in
$L_{G}^{2}(\mathcal{F}_{t})$ is just the right
hand side of (\ref{d-B-Ito}). By the next Remark we also have $\zeta_{t}%
^{N}\rightarrow0$ in $L_{G}^{2}(\mathcal{F}_{t})$. We then have
proved (\ref{d-B-Ito}).
\end{proof}

\begin{remark}
In the proof of $\zeta_{t}^{N}\rightarrow0$ in
$L_{G}^{2}(\mathcal{F}_{t})$, we use the following estimates: for
$\psi^{N}\in M_{G}^{1,0}(0,T)$ such that
$\psi_{t}^{N}=\sum_{k=0}^{N-1}\xi_{t_{k}}^{N}\mathbf{I}_{[t_{k}^{N}%
,t_{k+1}^{N})}(t)$, and $\pi_{T}^{N}=\{0\leq t_{0},\cdots,t_{N}=T\}$
with
$\lim_{N\rightarrow \infty}\mu(\pi_{T}^{N})=0$ and $\sum_{k=0}^{N-1}%
\mathbb{\hat{E}}[|\xi_{t_{k}}^{N}|](t_{k+1}^{N}-t_{k}^{N})\leq C$,
for all
$N=1,2,\cdots$, we have $\mathbb{\hat{E}}[|\sum_{k=0}^{N-1}\xi_{k}^{N}%
(t_{k+1}^{N}-t_{k}^{N})^{2}]\rightarrow0$ and, for any fixed
$\mathbf{a,\bar
{a}\in}\mathbb{R}^{d}$,%
\begin{align*}
\mathbb{\hat{E}}[|\sum_{k=0}^{N-1}\xi_{k}^{N}(\left \langle B^{\mathbf{a}%
}\right \rangle _{t_{k+1}^{N}}-\left \langle B^{\mathbf{a}}\right
\rangle
_{t_{k}^{N}})^{2}|]  &  \leq \sum_{k=0}^{N-1}\mathbb{\hat{E}[}|\xi_{k}%
^{N}|\cdot \mathbb{\hat{E}}[(\left \langle B^{\mathbf{a}}\right
\rangle
_{t_{k+1}^{N}}-\left \langle B^{\mathbf{a}}\right \rangle _{t_{k}^{N}}%
)^{2}|\mathcal{H}_{t_{k}^{N}}]]\\
&  =\sum_{k=0}^{N-1}\mathbb{\hat{E}[}|\xi_{k}^{N}|]\sigma_{\mathbf{aa}^{T}%
}^{2}(t_{k+1}^{N}-t_{k}^{N})^{2}\rightarrow0,
\end{align*}%
\begin{align*}
&  \mathbb{\hat{E}}[|\sum_{k=0}^{N-1}\xi_{k}^{N}(\left \langle B^{\mathbf{a}%
}\right \rangle _{t_{k+1}^{N}}-\left \langle B^{\mathbf{a}}\right
\rangle
_{t_{k}^{N}})(t_{k+1}^{N}-t_{k}^{N})|]\\
&  \leq \sum_{k=0}^{N-1}\mathbb{\hat{E}[}|\xi_{k}^{N}|(t_{k+1}^{N}-t_{k}%
^{N})\cdot \mathbb{\hat{E}}[(\left \langle B^{\mathbf{a}}\right
\rangle
_{t_{k+1}^{N}}-\left \langle B^{\mathbf{a}}\right \rangle _{t_{k}^{N}%
})|\mathcal{H}_{t_{k}^{N}}]]\\
&  =\sum_{k=0}^{N-1}\mathbb{\hat{E}[}|\xi_{k}^{N}|]\sigma_{\mathbf{aa}^{T}%
}(t_{k+1}^{N}-t_{k}^{N})^{2}\rightarrow0,
\end{align*}
as well as
\begin{align*}
\mathbb{\hat{E}}[|\sum_{k=0}^{N-1}\xi_{k}^{N}(t_{k+1}^{N}-t_{k}^{N}%
)(B_{t_{k+1}^{N}}^{\mathbf{a}}-B_{t_{k}^{N}}^{\mathbf{a}})|]  & \leq
\sum_{k=0}^{N-1}\mathbb{\hat{E}[}|\xi_{k}^{N}|](t_{k+1}^{N}-t_{k}%
^{N})\mathbb{\hat{E}}[|B_{t_{k+1}^{N}}^{\mathbf{a}}-B_{t_{k}^{N}}^{\mathbf{a}%
}|]\\
&  =\sqrt{\frac{2\sigma_{\mathbf{aa}^{T}}}{\pi}}\sum_{k=0}^{N-1}%
\mathbb{\hat{E}[}|\xi_{k}^{N}|](t_{k+1}^{N}-t_{k}^{N})^{3/2}\rightarrow0\
\end{align*}
and%
\begin{align*}
&  \mathbb{\hat{E}}[|\sum_{k=0}^{N-1}\xi_{k}^{N}(\left \langle B^{\mathbf{a}%
}\right \rangle _{t_{k+1}^{N}}-\left \langle B^{\mathbf{a}}\right
\rangle
_{t_{k}^{N}})(B_{t_{k+1}^{N}}^{\mathbf{\bar{a}}}-B_{t_{k}^{N}}^{\mathbf{\bar
{a}}})|]\\
&  \leq \sum_{k=0}^{N-1}\mathbb{\hat{E}[}|\xi_{k}^{N}|]\mathbb{\hat{E}%
[}(\left \langle B^{\mathbf{a}}\right \rangle _{t_{k+1}^{N}}-\left
\langle
B^{\mathbf{a}}\right \rangle _{t_{k}^{N}})|B_{t_{k+1}^{N}}^{\mathbf{\bar{a}}%
}-B_{t_{k}^{N}}^{\mathbf{\bar{a}}}|]\\
&  \leq \sum_{k=0}^{N-1}\mathbb{\hat{E}[}|\xi_{k}^{N}|]\mathbb{\hat{E}%
[}(\left \langle B^{\mathbf{a}}\right \rangle _{t_{k+1}^{N}}-\left
\langle
B^{\mathbf{a}}\right \rangle _{t_{k}^{N}})^{2}]^{1/2}\mathbb{\hat{E}%
[}|B_{t_{k+1}^{N}}^{\mathbf{\bar{a}}}-B_{t_{k}^{N}}^{\mathbf{\bar{a}}}%
|^{2}]^{1/2}\\
&  =\sum_{k=0}^{N-1}\mathbb{\hat{E}[}|\xi_{k}^{N}|]\sigma_{\mathbf{aa}^{T}%
}^{1/2}\sigma_{\mathbf{\bar{a}\bar{a}}^{T}}^{1/2}(t_{k+1}^{N}-t_{k}^{N}%
)^{3/2}\rightarrow0.
\end{align*}

\end{remark}

We now can claim our $G$--It\^{o}'s formula. Consider%
\[
X_{t}^{\nu}=X_{0}^{\nu}+\int_{0}^{t}\alpha_{s}^{\nu}ds+\int_{0}^{t}\eta
_{s}^{\nu ij}d\left \langle B^{i},B^{j}\right \rangle
_{s}+\int_{0}^{t}\beta
_{s}^{\nu j}dB_{s}^{j}%
\]

\begin{proposition}
\label{d-Prop-Ito}Let $\alpha^{\nu}$, $\beta^{\nu j}$ and $\eta^{\nu
ij}$, $\nu=1,\cdots,n$, $i,j=1,\cdots,d$ be bounded processes of
$M_{G}^{2}(0,T)$.
Then for each $t\geq0$ and $\Phi\in L_{G}^{2}(\mathcal{F}_{t})$ we have%
\begin{align}
\Phi(X_{t})-\Phi(X_{s})  &  =\int_{s}^{t}\partial_{x^{\nu}}\Phi(X_{u}%
)\beta_{u}^{\nu j}dB_{u}^{j}+\int_{s}^{t}\partial_{x_{\nu}}\Phi(X_{u}%
)\alpha_{u}^{\nu}du\label{d-Ito-form1}\\
&  +\int_{s}^{t}[\partial_{x^{\nu}}\Phi(X_{u})\eta_{u}^{\nu ij}+\frac{1}%
{2}\partial_{x^{\mu}x^{\nu}}^{2}\Phi(X_{u})\beta_{u}^{\nu
i}\beta_{u}^{\nu j}]d\left \langle B^{i},B^{j}\right \rangle
_{u}\nonumber
\end{align}

\end{proposition}

\begin{proof}
We first consider the case where $\alpha$, $\eta$ and $\beta$ are
step
processes of the form%
\[
\eta_{t}(\omega)=\sum_{k=0}^{N-1}\xi_{k}(\omega)\mathbf{I}_{[t_{k},t_{k+1}%
)}(t).
\]
From the above Lemma, it is clear that (\ref{d-Ito-form1}) holds
true. Now let
\[
X_{t}^{\nu,N}=X_{0}^{\nu}+\int_{0}^{t}\alpha_{s}^{\nu,N}ds+\int_{0}^{t}%
\eta_{s}^{\nu ij,N}d\left \langle B^{i},B^{j}\right \rangle _{s}+\int_{0}%
^{t}\beta_{s}^{\nu j,N}dB_{s}^{j}%
\]
where $\alpha^{N}$, $\eta^{N}$ and $\beta^{N}$ are uniformly bounded
step processes that converge to $\alpha$, $\eta$ and $\beta$ in
$M_{G}^{2}(0,T)$ as
$N\rightarrow \infty$. From Lemma \ref{d-Lem-26}%
\begin{align}
\Phi(X_{t}^{N})-\Phi(X_{0})  &  =\int_{0}^{t}\partial_{x^{\nu}}\Phi(X_{u}%
^{N})\beta_{u}^{\nu
j,N}dB_{u}^{j}+\int_{0}^{t}\partial_{x_{\nu}}\Phi
(X_{u}^{N})\alpha_{u}^{\nu,N}du\label{d-N-Ito}\\
&  +\int_{0}^{t}[\partial_{x^{\nu}}\Phi(X_{u}^{N})\eta_{u}^{\nu
ij,N}+\frac
{1}{2}\partial_{x^{\mu}x^{\nu}}^{2}\Phi(X_{u}^{N})\beta_{u}^{\mu
i,N}\beta _{u}^{\nu j,N}]d\left \langle B^{i},B^{j}\right \rangle
_{u}\nonumber
\end{align}
Since%
\begin{align*}
&  \mathbb{\hat{E}[}|X_{t}^{N,\mu}-X_{t}^{\mu}|^{2}]\\
&  \leq C\int_{0}^{T}\{
\mathbb{\hat{E}}[(\alpha_{s}^{\mu,N}-\alpha_{s}^{\mu
})^{2}]+\mathbb{\hat{E}}[|\eta_{s}^{\mu,N}-\eta_{s}^{\mu}|^{2}]+\mathbb{\hat
{E}}[(\beta_{s}^{\mu,N}-\beta_{s}^{\mu})^{2}]\}ds\\
&
\end{align*}
We then can prove that, in $M_{G}^{2}(0,T)$,
\begin{align*}
\partial_{x^{\nu}}\Phi(X_{\cdot}^{N})\eta_{\cdot}^{\nu ij,N}  &
\rightarrow \partial_{x^{\nu}}\Phi(X_{\cdot})\eta_{\cdot}^{\nu ij}\\
\partial_{x^{\mu}x^{\nu}}^{2}\Phi(X_{\cdot}^{N})\beta_{\cdot}^{\mu i,N}%
\beta_{\cdot}^{\nu j,N}  &  \rightarrow \partial_{x^{\mu}x^{\nu}}^{2}%
\Phi(X_{\cdot})\beta_{\cdot}^{\mu i}\beta_{\cdot}^{\nu j}\\
\partial_{x_{\nu}}\Phi(X_{\cdot}^{N})\alpha_{\cdot}^{\nu,N}  &  \rightarrow
\partial_{x_{\nu}}\Phi(X_{\cdot})\alpha_{\cdot}^{\nu}\\
\partial_{x^{\nu}}\Phi(X_{\cdot}^{N})\beta_{\cdot}^{\nu j,N}  &
\rightarrow \partial_{x^{\nu}}\Phi(X_{\cdot})\beta_{\cdot}^{\nu j}%
\end{align*}
We then can pass limit in both sides of (\ref{d-N-Ito}) to get
(\ref{d-Ito-form1}).
\end{proof}

\begin{example}%
\begin{align*}
(B_{t},AB_{t})  &  =\sum_{i,j}^{d}A_{ij}B_{t}^{i}B_{t}^{j}=2\sum_{i,j}%
^{d}[A_{ij}\int_{0}^{t}B_{t}^{i}dB^{j}+\ A_{ij}\left \langle B^{i}%
,B^{j}\right \rangle _{t}]\\
&  =2\sum_{i,j}^{d}[A_{ij}\int_{0}^{t}B_{t}^{i}dB^{j}+\ (A,\left
\langle B\right \rangle _{t})].
\end{align*}
Thus
\[
\mathbb{\hat{E}}[(A,\left \langle B\right \rangle _{t})]=\mathbb{\hat{E}}%
[(B_{t},AB_{t})]=2G(A)t.
\]

\end{example}

\section{$G$--martingales, $G$--convexity and Jensen's inequality}

\subsection{The notion of $G$--martingales}

We now give the notion of $G$--martingales:

\begin{definition}
A process $(M_{t})_{t\geq0}$ is called a $G$\textbf{--martingale}
(respectively, $G$\textbf{--supermartingale},
$G$\textbf{--submartingale}) if for each $0\leq s\leq t<\infty$, we
have $M_{t}\in L_{G}^{1}(\mathcal{F}_{t})$ and
\[
\mathbb{\hat{E}}[M_{t}|\mathcal{H}_{s}]=M_{s},\  \  \
\text{(resp.,\ }\leq M_{s},\  \  \geq M_{s}).
\]

\end{definition}

It is clear that for a fixed $X\in L_{G}^{1}(\mathcal{F})$ $\mathbb{\hat{E}%
}[X|\mathcal{H}_{t}]_{t\geq0}$ is a $G$--martingale. In general how
to characterize a $G$--martingale or a $G$--supermartingale is still
a very interesting problem. But the following example gives an
important characterization:

\begin{example}
Let $M_{0}\in \mathbb{R}$, $\varphi=(\varphi^{i})_{i=1}^{d}\in M_{G}%
^{2}(0,T;\mathbb{R}^{d})$ and $\eta=(\eta^{ij})_{i,j=1}^{d}\in M_{G}%
^{2}(0,T;\mathbb{S}(d))$ be given and let%
\[
M_{t}=M_{0}+\int_{0}^{t}\varphi_{u}^{i}dB_{s}^{j}+\int_{0}^{t}\eta_{u}%
^{ij}d\left \langle B^{i},B^{j}\right \rangle
_{u}-\int_{0}^{t}2G(\eta _{u})du,\ t\in \lbrack0,T].
\]
Then $M$ is a $G$--martingale on $[0,T]$. To consider this it
suffices to prove the case $\eta \in
M_{G}^{2,0}(0,T;\mathbb{S}(d))$, i.e.,
\[
\eta_{t}=\sum_{k=0}^{N-1}\eta_{t_{k}}I_{[t_{k}.t_{k+1})}(t).
\]
We have for $s\in \lbrack t_{N-1},t_{N}]$,%
\begin{align*}
\mathbb{\hat{E}}[M_{t}|\mathcal{H}_{s}]  &  =M_{s}+\mathbb{\hat{E}}%
[\eta_{t_{N-1}}^{ij}(\left \langle B^{i},B^{j}\right \rangle
_{t}-\left \langle
B^{i},B^{j}\right \rangle _{s})-2G(\eta_{t_{N-1}})(t-s)|\mathcal{H}_{s}]\\
&  =M_{s}+\mathbb{\hat{E}}[\eta_{t_{N-1}}^{ij}(B_{t}^{i}-B_{s}^{i})(B_{t}%
^{j}-B_{s}^{j})|\mathcal{H}_{s}]-2G(\eta_{t_{N-1}})(t-s)\\
&  =M_{s}.
\end{align*}
In the last step we apply the relation (\ref{d-eq-GMB-14a}). We then
can repeat this procedure, step by step backwardly, to prove the
result for any $s\in \lbrack0,t_{N-1}]$. \
\end{example}

\begin{remark}
It is worth mentioning that for a $G$--martingale, in general, $-M$
is not a $G$--martingale. But in the above example when $\eta
\equiv0$ then $-M$ is still a $G$--martingale. This makes an
essential difference of the $dB$ part and the $d\left \langle
B\right \rangle $ part of a $G$--martingale.
\end{remark}

\subsection{$G$--convexity and Jensen's inequality for\\ $G$--expectation}

A very interesting question is whether the well--known Jensen's
inequality still holds for $G$--expectation. In the framework of
$g$--expectation this problem was investigated in \cite{BCHMP1} in
which a counterexample is given to indicate that, even for a linear
function which is obviously convex, Jensen's inequality for
$g$-expectation generally does not hold. Stimulated by this example
\cite{JC1} proved that Jensen's inequality holds for any convex
function under a $g$--expectation if and only if the corresponding
generating function $g=g(t,z)$ is super-homogeneous in $z$. Here we
will discuss this problem from a quite different point of view. We
will define a new notion of convexity:

\begin{definition}
A $C^{2}$-function $h:\mathbb{R\longmapsto R}$ is called
$G$\textbf{--convex} if the following condition holds for each
$(y,z,A)\in \mathbb{R}\times \mathbb{R}^{d}\times \mathbb{S}(d)$:
\begin{equation}
G(h^{\prime}(y)A+h^{\prime \prime}(y)zz^{T})-h^{\prime}(y)G(A)\geq
0,\  \label{d-G-conv}%
\end{equation}
where $h^{\prime}$ and $h^{\prime \prime}$ denote the first and the
second derivatives of $h$, respectively.
\end{definition}

It is clear that  in the special situation where
$G(D^{2}u)=\frac{1}{2}\Delta u$  a $G$-convex function becomes is a
convex function in the classical sense.

\begin{lemma}
The following two conditions are equivalent:\newline
\textbf{\textsl{(i)}} The function $h$ is $G$--convex.\newline
\textbf{\textsl{(ii)}} The following Jensen inequality holds: For
each $T\geq0$,
\begin{equation}
\mathbb{\hat{E}}[h(\varphi(B_{T}))]\geq
h(\mathbb{\hat{E}}[\varphi(B_{T})]),
\label{d-gg-Jensen}%
\end{equation}
for each $C^{2}$--function $\varphi$ such that $h(\varphi(B_{T}))$
and $\varphi(B_{T})\in L_{G}^{1}(\mathcal{F}_{T})$.
\end{lemma}

\begin{proof}
(i) $\Longrightarrow$ (ii): From the definition
$u(t,x):=P_{t}^{G}[\varphi ](x)=\mathbb{\hat{E}}[\varphi(x+B_{t})]$
solves the nonlinear heat equation
(\ref{d-eq-heat}). Here we only consider the case where $u$ is a $C^{1,2}%
$-function. Otherwise we can use the language of viscosity solution.
By simple calculation we have
\[
\partial_{t}h(u(t,x))=h^{\prime}(u)\partial_{t}u=h^{\prime}(u(t,x))G(D^{2}%
u(t,x)),
\]
or%
\[
\partial_{t}h(u(t,x))-G(D^{2}h(u(t,x)))-f(t,x)=0,\ h(u(0,x))=h(\varphi(x)),
\]
where we denote%
\[
f(t,x)=h^{\prime}(u(t,x))G(D^{2}u(t,x))-G(D^{2}h(u(t,x))).
\]
Since $h$ is $G$--convex it follows that $f\leq0$ and thus $h(u)$ is
a
$G$-subsolution. It follows from the maximum principle that $h(P_{t}%
^{G}(\varphi)(x))\leq P_{t}^{G}(h(\varphi))(x)$. In particular
(\ref{d-gg-Jensen}) holds. Thus we have (ii).\newline(ii)
$\Longrightarrow$(i): For a fixed $(y,z,A)\in \mathbb{R\times
R}^{d}\times \mathbb{S}(d)$ we set $\varphi(x):=y+\left \langle
x,z\right \rangle +\frac{1}{2}\left \langle Ax,x\right \rangle $.
From the definition of $P_{t}^{G}$ we have $\partial
_{t}(P_{t}^{G}(\varphi)(x))|_{t=0}=G(D^{2}\varphi)(x)$. With (ii) we have%
\[
h(P_{t}^{G}(\varphi)(x))\leq P_{t}^{G}(h(\varphi))(x).
\]
Thus, for $t>0$,%
\[
\frac{1}{t}[h(P_{t}^{G}(\varphi)(x))-h(\varphi(x))]\leq \frac{1}{t}[P_{t}%
^{G}(h(\varphi))(x)-h(\varphi(x))]
\]
We then let $t$ tend to $0$:%
\[
h^{\prime}(\varphi(x))G(D^{2}\varphi(x))\leq
G(D_{xx}^{2}h(\varphi(x))).
\]
Since $D_{x}\varphi(x)=z+Ax$ and $D_{xx}^{2}\varphi(x)=A$ we then
set $x=0$ and obtain (\ref{d-G-conv}).
\end{proof}

\begin{proposition}
The following two conditions are equivalent:\newline
\textbf{\textsl{(i)}} the function $h$ is $G$--convex.\newline
\textbf{\textsl{(ii)}} The following Jensen inequality holds:
\begin{equation}
\mathbb{\hat{E}}[h(X)|\mathcal{H}_{t}]\geq h(\mathbb{\hat{E}}[X|\mathcal{H}%
_{t}]),\  \ t\geq0, \label{d-JensenX}%
\end{equation}
for each $X\in L_{G}^{1}(\mathcal{F})$ such that $h(X)\in L_{G}^{1}%
(\mathcal{F})$.
\end{proposition}

\begin{proof}
The part (ii) $\Longrightarrow$(i) is already provided by the above
lemma. We can also apply this lemma to prove (\ref{d-JensenX}) for
the case $X\in
L_{ip}^{0}(\mathcal{F})$ of the form $X=\varphi(B_{t_{1}},\cdots,B_{t_{m}%
}-B_{t_{m-1}})$ by using the procedure of the definition of $\mathbb{\hat{E}%
}[\cdot]$ and $\mathbb{\hat{E}}[\cdot|\mathcal{H}_{t}]$ given in
Definitions \ref{d-Def-3} and \ref{d-Def-3-1}, respectively. We then
can extend this Jensen's inequality, under the norm $\left \Vert
\cdot \right \Vert =\mathbb{\hat{E}}[|\cdot|]$ to the general
situation.
\end{proof}

\begin{remark}
The above notion of $G$--convexity can be also applied to the case
where the nonlinear heat equation (\ref{d-eq-heat}) has a more
general form:
\begin{equation}
\frac{\partial u}{\partial t}-G(u,\nabla u,D^{2}u)=0,\  \
u(0,x)=\psi(x)
\label{d-PG-psi}%
\end{equation}
(see Examples 4.3, 4.4 and 4.5 in \cite{Peng2005}). In this case a $C^{2}%
$-function $h:\mathbb{R\longmapsto R}$ is said to be
$G$--\textbf{convex} if the following condition holds for each
$(y,z,A)\in \mathbb{R}\times \mathbb{R}^{d}\times \mathbb{S}(d)$:
\[
G(y,h^{\prime}(y)z,h^{\prime}(y)A+h^{\prime
\prime}(y)zz^{T})-h^{\prime }(y)G(y,z,A)\geq0.
\]
We don't need the subadditivity and/or positive homogeneity of
$G(y,z,A)$. A particularly interesting situation is the case of
$g$--expectation for a given generating function $g=g(y,z)$,
$(y,z)\in \mathbb{R}\times \mathbb{R}^{d}$, in this case we have the
following $g$--convexity:
\begin{equation}
\frac{1}{2}h^{\prime
\prime}(y)|z|^{2}+g(h(y),h^{\prime}(y)z)-h^{\prime
}(y)g(y,z)\geq0. \label{d-g-convex}%
\end{equation}
This situation is systematically studied in Jia and Peng
\cite{JiaPeng}.
\end{remark}

\begin{example}
Let $h$ be a $G$--convex function and $X\in L_{G}^{1}(\mathcal{F})$
such that $h(X)\in L_{G}^{1}(\mathcal{F})$. Then $Y_{t}=h(\mathbb{\hat{E}%
}[X|\mathcal{H}_{t}])$, $t\geq0$, is a $G$--submartingale: For each
$s\leq t$,
\[
\mathbb{\hat{E}}[Y_{t}|\mathcal{H}_{s}]=\mathbb{\hat{E}}[h(\mathbb{\hat{E}%
}[X|\mathcal{F}_{t}])|\mathcal{F}_{s}]\geq h(\mathbb{\hat{E}}[X|\mathcal{F}%
_{s}])=Y_{s}\text{.}%
\]

\end{example}

\section{Stochastic differential equations}

We consider the following SDE driven by $G$-Brownian motion.%
\begin{equation}
X_{t}=X_{0}+\int_{0}^{t}b(X_{s})ds+\int_{0}^{t}h_{ij}(X_{s})d\left
\langle
B^{i},B^{j}\right \rangle _{s}+\int_{0}^{t}\sigma_{j}(X_{s})dB_{s}^{j}%
,\ t\in \lbrack0,T], \label{d-SDE}%
\end{equation}
where the initial condition $X_{0}\in \mathbb{R}^{n}$ is given and
\[
b,h_{ij},\sigma_{j}:\mathbb{R}^{n}\mapsto \mathbb{R}^{n}%
\]
are given Lipschitz functions, i.e.,
$|\varphi(x)-\varphi(x^{\prime})|\leq K|x-x^{\prime}|$, for each
$x$, $x^{\prime}\in \mathbb{R}^{n}$, $\varphi=b$, $\eta_{ij}$ and
$\sigma_{j}$, respectively. Here the horizon $[0,T]$ can be
arbitrarily large. The solution is a process $X\in
M_{G}^{2}(0,T;\mathbb{R}^{n})$ satisfying the above SDE. We first
introduce the following mapping on a fixed
interval $[0,T]$:%
\[
\Lambda_{\cdot}(Y):=Y\in M_{G}^{2}(0,T;\mathbb{R}^{n})\longmapsto M_{G}%
^{2}(0,T;\mathbb{R}^{n})\  \
\]
by setting $\Lambda_{t}=X_{t}$, $t\in \lbrack0,T]$, with
\[
\Lambda_{t}(Y)=X_{0}+X_{0}+\int_{0}^{t}b(Y_{s})ds+\int_{0}^{t}h_{ij}%
(Y_{s})d\left \langle B^{i},B^{j}\right \rangle
_{s}+\int_{0}^{t}\sigma _{j}(Y_{s})dB_{s}^{j}.
\]

We immediately have

\begin{lemma}
For each $Y,Y^{\prime}\in M_{G}^{2}(0,T;\mathbb{R}^{n})$ we have the
following estimate:%
\[
\mathbb{\hat{E}}[|\Lambda_{t}(Y)-\Lambda_{t}(Y^{\prime})|^{2}]\leq
C\int _{0}^{t}\mathbb{\hat{E}}[|Y_{s}-Y_{s}^{\prime}|^{2}]ds,\ t\in
\lbrack0,T],
\]
where the constant $C$ depends only on $K$, $\Gamma$ and the
dimension $n$.
\end{lemma}

\begin{proof}
This is a direct consequence of the inequalities (\ref{d-Bohner}),
(\ref{d-e2}) and (\ref{d-qua-ine}).
\end{proof}

We now prove that SDE (\ref{d-SDE}) has a unique solution. By
multiplying $e^{-2Ct}$ on both sides of the above inequality and
then integrating them on
$[0,T]$, it follows that%
\begin{align*}
\int_{0}^{T}\mathbb{\hat{E}}[|\Lambda_{t}(Y)-\Lambda_{t}(Y^{\prime}%
)|^{2}]e^{-2Ct}dt  &  \leq C\int_{0}^{T}e^{-2Ct}\int_{0}^{t}\mathbb{\hat{E}%
}[|Y_{s}-Y_{s}^{\prime}|^{2}]dsdt\\
&
=C\int_{0}^{T}\int_{s}^{T}e^{-2Ct}dt\mathbb{\hat{E}}[|Y_{s}-Y_{s}^{\prime
}|^{2}]ds\\
&  =(2C)^{-1}C\int_{0}^{T}(e^{-2Cs}-e^{-2CT})\mathbb{\hat{E}}[|Y_{s}%
-Y_{s}^{\prime}|^{2}]ds.
\end{align*}
We then have
\[
\int_{0}^{T}\mathbb{\hat{E}}[|\Lambda_{t}(Y)-\Lambda_{t}(Y^{\prime}%
)|^{2}]e^{-2Ct}dt\leq \frac{1}{2}\int_{0}^{T}\mathbb{\hat{E}}[|Y_{t}%
-Y_{t}^{\prime}|^{2}]e^{-2Ct}dt.
\]
We observe that the following two norms are equivalent in $M_{G}%
^{2}(0,T;\mathbb{R}^{n})$
\[
\int_{0}^{T}\mathbb{\hat{E}}[|Y_{t}|^{2}]dt\thicksim \int_{0}^{T}%
\mathbb{\hat{E}}[|Y_{t}|^{2}]e^{-2Ct}dt.
\]
From this estimate we can obtain that $\Lambda(Y)$ is a contraction
mapping. Consequently, we have

\begin{theorem}
There exists a unique solution $X\in M_{G}^{2}(0,T;\mathbb{R}^{n})$
of the stochastic differential equation (\ref{d-SDE}).
\end{theorem}

\chapter{Other Topics and Applications}

\section{Nonlinear Feynman-Kac formula}

Consider SDE:%
\begin{align*}
dX_{s}^{t,x}  &  =b(X_{s}^{t,x})ds+h(X_{s}^{t,x})d\left \langle
B\right \rangle
_{s}+\sigma(X_{s}^{t,x})dB_{s},\ s\in \lbrack t,T],\\
X_{x}^{t,x}  &  =x.
\end{align*}%
\[
Y_{s}^{t,x}=\mathbb{\hat{E}}[\Phi(X_{T}^{t,x})+\int_{s}^{T}g(X_{r}^{t,x}%
,Y_{r}^{t,x})dr+\int_{s}^{T}f(X_{r}^{t,x},Y_{r}^{t,x})d\left \langle
B\right \rangle _{r}|\mathcal{H}_{s}].
\]
It is clear that $u(t,x):=Y_{t}^{t,x}\in \mathcal{H}_{0}^{t}$, thus
it is a
deterministic function of $(t,x)$. $u(t,x)$ solves%
$$\left\{\begin{array}[c]{l}
\partial_{t}u+   \sup \{ \left(  b(x)+h(x)\gamma,\nabla u\right)  +\frac
{1}{2}\left(  \sigma(x)\gamma \sigma^{T}(x),D^{2}u\right) \\
\  \  \  \  \   \  \  \  \  \  \  \  \  \  \  \  \  \  \  \  \  \ +g(x,u)+f(x,u)\gamma \}=0,\\
u|_{t=T}    =\Phi.
\end{array}
\right.
$$

\begin{example}
Let $B=(B^{1},B^{2})$ be a $2$-dimensional $G$-Browian motion with
\[
G(A)=G_{1}(a_{11})+G_{2}(a_{22})
\]%
\[
G_{i}(a)=\frac{1}{2}(\overline{\sigma}_{i}^{2}a^{+}-\underline{\sigma}_{i}%
^{2}a^{-})
\]
In this case by It\^{o}'s formula
\end{example}

\[
dX_{s}=\mu X_{s}ds+\nu X_{s}d\left \langle B^{1}\right \rangle
_{s}+\sigma X_{s}dB_{s}^{2},\  \ X_{t}=x.
\]
\newline%
\[
u(t,x)=\mathbb{\hat{E}}[\varphi(X_{T}^{t,x})]=\mathbb{\hat{E}}[\varphi
(X_{T}^{t,x})|\mathcal{F}_{t}].
\]
We have%
\[
u(t,x)=\mathbb{\hat{E}}[u(t+\delta,X_{t+\delta}^{t,x})].
\]
From which it is easy to prove that%
\[
\partial_{t}u+\sup_{\gamma \in \lbrack \underline{\sigma}_{1}^{2},\overline
{\sigma}_{1}^{2}]}(\mu+v\gamma)x\partial_{x}u+\frac{x^{2}}{2}\sup_{\gamma
\in \lbrack
\underline{\sigma}_{2}^{2},\overline{\sigma}_{2}^{2}]}[\partial
_{xx}^{2}u]=0.
\]

\section{Markov-Nisio process}

\begin{definition}
A $n$-dimensional process $(X_{t})_{t\geq0}$ on $(\Omega,\mathcal{H}%
,(\mathcal{H}_{t})_{t\geq0},\mathbb{\hat{E}})$ is called a
Markov-Nisio process if for each $0\leq s\leq t$ and each $\varphi
\in C_{b}(\mathbb{R}^{n})$ there exists a $\psi \in
C_{b}(\mathbb{R}^{n})$ such that
\[
\mathbb{\hat{E}}[\varphi(X_{t})|\mathcal{H}_{s}]=\psi(X_{s}).
\]

\end{definition}

\begin{remark}
When $\mathbb{\hat{E}}$ is a linear expectation then this notion
describes a classical Markovian process related to Markovian group.
Nisio's semigroup (see Nision \cite{Nisio1}, \cite{Nisio2}) extended
Markovian group to sublinear cases to describe the value function of
an optimal control system.
\end{remark}

\subsection{Martingale problem}

For a given function
$F(x,p,A):\mathbb{R}^d\times\mathbb{R}^d\times\mathbb{S}(d)\longmapsto
\mathbb{R}$ we call a sublinear expectation space solves the
martingale problem related to $F$ if there exists a time
consistent nonlinear expectation space $(\Omega,\mathcal{H},(\mathcal{H}%
_{t})_{t\geq0},\mathbb{\hat{E}})$ such that, for each $\varphi \in
C_{b}^{\infty}(\mathbb{R}^{n})$ we have
\[
M_{t}:=\varphi(X_{t})-\varphi(X_{0})-\int_{0}^{t}F(X_{s},D\varphi(X_{s}%
),D^{2}\varphi(X_{s}))ds,\  \ t\geq0
\]
is a $\mathbb{\hat{E}}$-martingale. (see P. \cite{Peng2004},
\cite{Peng2005}).

\section{Pathwise properties of $G$-Brownian motion}

It is proved that
\[
\mathbb{\hat{E}}[X]=\sup_{Q\in \mathcal{P}}\mathbb{E}_{Q}[X]
\]
where $\mathcal{P}$ is a family of probabilities on $(\Omega,\mathcal{B}%
(\Omega))$, $\Omega=C_{0}(0,\infty;\mathbb{R}^{d})$. We set
\[
\hat{c}(A)=\sup_{Q\in \mathcal{P}}\mathbb{E}_{Q}[\mathbf{1}_{A}]
\]
${\hat c}(\cdot)$ is a Choquet capacity. We have

\begin{theorem}
(Denis-Hu-Peng) There exists a continuous version of $G$-Brownian
motion: We can find a pathwise process
$(\tilde{B}_{t}(\omega))_{t\geq0}$ on some $\tilde{\Omega}\subset
\Omega$, with $\hat{c}(\tilde{\Omega}^{c})=0$, such that, for each
$\omega \in \tilde{\Omega}$, $\tilde{B}_{t}(\omega)\in
C_{0}(0,\infty)$ and $\mathbb{\hat{E}}[|B_{t}-\tilde{B}_{t}|]=0$,
$\forall t\in \lbrack0,\infty)$.
\end{theorem}


\chapter{Appendix}
We will use the following known result in viscosity solution theory
(see Theorem 8.3 of Crandall Ishii and Lions \cite{CIL}).

\begin{theorem}
\label{Thm-8.3} Let $u_{i}\in$USC$((0,T)\times Q_{i})$ for
$i=1,\cdots,k$ where $Q_{i}$ is a locally compact subset of
$\mathbb{R}^{N_{i}}$. Let $\varphi$ be defined on an open
neighborhood of $(0,T)\times Q_{1}\times \cdots \times Q_{k}$ and
such that $(t,x_{1},\cdots,x_{k})$ is once continuously
differentiable in $t$ and twice continuously differentiable in
$(x_{1},\cdots,x_{k})\in
Q_{1}\times \cdots \times Q_{k}$. Suppose that $\hat{t}\in(0,T)$, $\hat{x}%
_{i}\in Q_{i}$ for $i=1,\cdots,k$ and
\begin{align*}
w(t,x_{1},\cdots,x_{k})  &  :=u_{1}(t,x_{1})+\cdots+u_{k}(t,x_{k}%
)-\varphi(t,x_{1},\cdots,x_{k})\\
&  \leq w(\hat{t},\hat{x}_{1},\cdots,\hat{x}_{k})
\end{align*}
for $t\in(0,T)$ and $x_{i}\in Q_{i}$. Assume, moreover that there is
an $r>0$
such that for every $M>0$ there is a $C$ such that for $i=1,\cdots,k$%
\begin{equation}%
\begin{array}
[c]{ll}%
& b_{i}\leq C  \text{ whenever \ } (b_{i},q_{i},X_{i})\in \mathcal{P}%
^{2,+}u_{i}(t,x_{i}),\\
& |x_{i}-\hat{x}_{i}|+|t-\hat{t}|\leq r \text{ and }|u_{i}(t,x_{i}%
)|+|q_{i}|+\left \Vert X_{i}\right \Vert \leq M.
\end{array}
\label{eq8.5}%
\end{equation}%
Then for each $\varepsilon>0$, there are $X_{i}\in
\mathbb{S}(N_{i})$ such that\newline\textsl{(i)}
$(b_{i},D_{x_{i}}\varphi(\hat{t},\hat{x}_{1},\cdots,\hat
{x}_{k}),X_{i})\in \overline{\mathcal{P}}^{2,+}u_{i}(\hat{t},\hat{x}%
_{i}),\  \ i=1,\cdots,k;$\newline\textsl{(ii)}
\[
-(\frac{1}{\varepsilon}+\left \Vert A\right \Vert )\leq \left[
\begin{array}
[c]{ccc}%
X_{1} & \cdots & 0\\
\vdots & \ddots & \vdots \\
0 & \cdots & X_{k}%
\end{array}
\right]  \leq A+\varepsilon A^{2}%
\]
\textsl{(iii)} $b_{1}+\cdots+b_{k}=\varphi_{t}(\hat{t},\hat{x}_{1},\cdots,\hat{x}_{k}%
)$\newline where $A=D^{2}\varphi(\hat{x})\in \mathbb{S}^{kN}$.
\end{theorem}
Observe that the above conditions (\ref{eq8.5}) will be guaranteed
by having $u_{i}$ be subsolutions of parabolic equations given in
the following theorem.

\begin{theorem}
\textbf{\label{Thm-dom}} (Domination Theorem) Let $m$-order
polynomial growth functions $u_{i}\in $USC$([0,T]\times
\mathbb{R}^{N})$ be subsolutions of
\begin{equation}
\partial_{t}u-G_{i}(D^{2}u)=0,\  \  \  \ i=1,\cdots,k,\label{visPDE}%
\end{equation}
on $(0,T)\times \mathbb{R}^{N}$. We assume that
$\{G_{i}\}_{i=1}^{k}$ satisfies the following domination condition:
\begin{equation}
\sum_{i=1}^kG(X_i)\leq 0,\ \ \  \text{ for all } X_i\in
\mathbb{S}(N), \ \text{ such that } \sum_{i=1}^kX_i\leq 0,
\label{dom}%
\end{equation}
and $G_{1}$ is monotonic, i.e., $G_{1}(X)\geq G_{1}(Y)$ if $X\geq
Y$. Then the following domination holds: If the initial condition
satisfies $u_{1}(0,x)+\cdots u_{k}(0,x)\leq0$ for each $x\in
\mathbb{R}^{N}$ then we have
\[
u_{1}(t,x)+\cdots u_{k}(t,x)\leq0,\  \  \
\forall(t,x)\in(0,T)\times \mathbb{R}^{N}.
\]
\end{theorem}
\noindent\textbf{Proof of Theorem \ref{Thm-dom}. } We first observe
that for $\bar{\delta
}>0$, since $D^{2}|x|^{2m}\geq0$, the function defined by $\tilde{u}%
_{1}:=u_{1}-\bar{\delta}/(T-t)-\bar{\delta}|x|^{2m}$ is also a
subsolution of
(\ref{visPDE})  with $i=1$, and satisfies with a strictly inequality:%
\begin{align*}
\partial_{t}\tilde{u}_{1}-G_{1}(D^{2}\tilde{u}_{1})  &  =\partial_{t}%
u_{1}-G_{1}(D^{2}u_{1}+\bar{\delta}D^{2}|x|^{2m})-\frac{\bar{\delta}%
}{(T-t)^{2}}\\
&  \leq \partial_{t}u_{1}-G_{1}(D^{2}u_{1})-\frac{\bar{\delta}}{(T-t)^{2}}%
\leq-\frac{\bar{\delta}}{(T-t)^{2}}.
\end{align*}
Since $u_{1}+u_{2}+\cdots u_{k}\leq0$ follows from
$\tilde{u}_{1}+u_{2}+\cdots u_{k}\leq0$ in the limit
$\bar{\delta}\downarrow0$, it suffices to prove
the theorem under the additional assumptions%
\[%
\begin{array}
[c]{cc}%
\partial_{t}u_{1}-G_{1}(D^{2}u_{1})\leq-c,\  \  \ c:=\bar{\delta}/T^{2} &
\text{and}\\
\lim_{t\rightarrow T,|x|\rightarrow \infty}u_{1}(t,x)=-\infty &
\text{uniformly
in }[0,T)\times \mathbb{R}^{N}.\text{ }%
\end{array}
\]
\newline To prove our result, we assume to the contrary that $u_{1}%
(s,z)+\cdots+u_{k}(s,z)=\delta>0$ for some $(s,z)\in(0,T)\times
\mathbb{R}^{N}$ and $\delta>0$. We will apply Theorem \ref{Thm-8.3}
for $x=(x_{1},\cdots
,x_{k})$, $x_{i}\in \mathbb{R}^{N}$ and%
\[
w(t,x):=\sum_{i=1}^{k}u_{i}(t,x_{i}),\  \  \varphi_{\alpha}(x):=\frac{\alpha}%
{2}(\sum_{i=1}^{k-1}|x_{i+1}-x_{i}|^{2}+|x_{k}-x_{1}|^{2}).
\]
Since for each large $\alpha>0$ the maximum of $w-\varphi_{\alpha}$
achieved at some $(t^{\alpha},x^{\alpha})$ uniformly inside a
compact subset of
$[0,T)\times \mathbb{R}^{k\times N}$. Set%
\[
M_{\alpha}=\sum_{i=1}^{k}u_{i}(t^{\alpha},x_{i}^{\alpha})-\varphi_{\alpha
}(t^{\alpha},x^{\alpha}).
\]
It is clear that $M_{\alpha}\geq \delta$. We can also check that
(see \cite{CIL} Lemma 3.1)
\begin{equation}
\left \{
\begin{array}
[c]{l}%
\text{(i) }\lim_{\alpha \rightarrow
\infty}\varphi_{\alpha}(t^\alpha, x^{\alpha
})=0\text{.}\\
\text{(ii)\ }\lim_{\alpha \rightarrow \infty}M_{\alpha}=\lim_{\alpha
\rightarrow
\infty}u_{1}(t^{\alpha},x_{1}^{\alpha})+\cdots+u_{k}(t^{\alpha},x_{k}^{\alpha}))\\
\; \; \; \; \; \; \;
\;=\sup_{(t,x}[u_{1}(\hat{t},\hat{x})+\cdots+u_k(\hat{t},\hat {x})].
\end{array}
\right.  \label{limit}%
\end{equation}
where $(\hat{t},\hat{x})$ is a limit point of
$(t^{\alpha},x_{1}^{\alpha})$.

If $t^{\alpha}=0$, we have%
\[
0<\delta \leq M_{\alpha}\leq
\sum_{i=1}^{k}u_{i}(0,x_{i}^{\alpha})-\varphi
_{\alpha}(t^{\alpha},x^{\alpha}).
\]
But by $\sum_{i=1}^{k}u_{i}(0,x_{i}^{\alpha})\rightarrow \sum_{i=1}^{k}%
u_{i}(0,\hat{x}_{i})\leq0$. So $t^{\alpha}$ must be strictly
positive for large $\alpha$. It follows from Theorem \ref{Thm-8.3}
that, for each $\varepsilon>0$ there exists $X_{i}\in \mathbb{S}(N)$
such that
\[
(b_{i}^{\alpha},D_{x_{i}}\varphi(t^{\alpha},x^{\alpha}),X_{i})\in
\bar {J}_{Q_{i}}^{2;+}u_{i}(t^{\alpha},x_{i}^{\alpha})\; \text{for
}i=1,\cdots,k.
\]
and such that $\sum_{i=1}^{k}b_{i}^{\alpha}=0$,
\begin{equation}
-(\frac{1}{\varepsilon}+\left \Vert A\right \Vert )\leq \left(
\begin{array}
[c]{cccc}%
X_{1} & \ldots & 0 & 0\\
\vdots & \ddots & \vdots & \vdots \\
0 & \ldots & X_{k-1} & 0\\
0 & \ldots & 0 & X_{k}%
\end{array}
\right)  \leq A+\varepsilon A^{2} \label{ine-matrix}%
\end{equation}
where $A=D^{2}\varphi_{\alpha}(x^{\alpha})\in \mathbb{S}^{3N}$ is
explicitely
given by%

\[
A=\alpha J_{3N}+\alpha I_{3N},\; \;J_{3N}=\left(
\begin{array}
[c]{cccc}%
I_{N} & \ldots & -I_{N} & -I_{N}\\
\vdots & \ddots & \vdots & \vdots \\
-I_{N} & \ldots & I_{N} & -I_{N}\\
-I_{N} & \ldots & -I_{N} & I_{N}%
\end{array}
\right)  .
\]
The second inequality of (\ref{ine-matrix}) implies
$\sum_{i=1}^{k}X_{i}\leq 0$. Setting
\begin{align*}
p_{1}  &  =D_{x_{1}}\varphi_{\alpha}(x^{\alpha})=\alpha(2x_{1}^{\alpha}%
-x_{3}^{\alpha}-x_{2}^{\alpha}),\; \\
&  \  \  \vdots \\
p_{k}  &
=D_{x_{k}}\varphi_{\alpha}(x^{\alpha})=\alpha(2x_{k}^{\alpha
}-x_{k-1}^{\alpha}-x_{1}^{\alpha}),
\end{align*}
We have $\sum_{i=1}^{k}p_{i}=0$, $\sum_{i=1}^{k}b_{i}=0$ and $(b_{i}%
,p_{i},X_{i})\in \bar{J}_{Q}^{2;+}u_{1}(t^{\alpha},x_{i}^{\alpha})$.
Thus
\begin{align*}
b_{1}-G_{1}(X_{1})  &  \leq-c,\\
b_{i}-G_{i}(X_{i})  &  \leq0,\  \ i=2,3,\cdots,k.
\end{align*}
This, together with the the domination condition (\ref{dom}) of $G_{i}$, implies %
\[
-c=-\sum_{i=1}^{k}b_{i}-c\geq-\sum_{i=1}^{k}G_{i}(X_{i})\geq0.\  \
\]
This induces a contradiction. The proof is complete.
\endproof

\begin{corollary}
(Comparison Theorem) Let $G_{1},G_{2}:\mathbb{S}(N)\mapsto
\mathbb{R}$ be given functions bounded by continuous functions on
$\mathbb{S}(N)$ and let $G_{1}$
be monotone and%
\[
G_{1}(X)\geq G_{2}(X),\  \  \  \forall X\in \mathbb{S}(N).
\]
Let $u_{2}\in$LSC$((0,T)\times \mathbb{R}^{N})$ be a viscosity
supersolution of $\partial_{t}u-G_{1}(D^{2}u)=0$ and
$u_{2}\in$USC$((0,T)\times \mathbb{R}^{N})$ be a viscosity
subsolution of $\partial_{t}u-G_{2}(D^{2}u)=0$, such that
$u_{1}(0,x)\geq u_{2}(0,x)$, for all $x\in \mathbb{R}^{N}$. Then we
have $u_{1}(t,x)\leq u_{2}(t,x)$ for all $(t,x)\in
\lbrack0,\infty)\times \mathbb{R}^{N}$.
\end{corollary}

\begin{proof}
It suffices to observe that $-u_{1}\in$USC$((0,T)\times
\mathbb{R}^{N})$ is viscosity subsolutions of
$\partial_{t}u-G_{\ast}(D^{2}u)=0$, with $G_{\ast }(X):=-G_{1}(-X)$.
We also have, for each $X_{1}+X_{2}\leq0$,
\begin{align*}
G_{\ast}(X_{1})+G_{2}(X_{2})  &  =G_{2}(X_{2})-G_{1}(-X_{1})\\
&  \leq G_{1}(X_{2})-G_{1}(-X_{1})\\
&  \leq G_{1}(X_{1}+X_{2})\leq0\text{.}%
\end{align*}
We thus can apply the above domination theorem to get
$-u_{1}+u_{2}\leq0$.
\end{proof}

\begin{corollary}
(Domination Theorem) Let $G_i:\mathbb{S}(N)\mapsto \mathbb{R}$,
$i=0,1$, be two given mappings bounded by some continuous functions
on $\mathbb{S}(N)$ and let $u_{i}\in$LSC$((0,T)\times
\mathbb{R}^{N})$ be viscosity supersolutions of
$\partial_{t}u-G_{i}(D^{2}u)=0$ respectively for $i=0,1$  and let
$u_{2}$ be a viscosity subsolution of
$\partial_{t}u-G_{1}(D^{2}u)=0$. We assume that $G_{0}$ is monotone
and that $G_{0}$ dominates $G_{1}$ in the following sense:
\[
G_{1}(X)-G_{1}(Y)\leq G_{0}(X-Y),\  \  \forall X,Y\in
\mathbb{S}(N).\
\]
Then the following domination holds: If
\[
u_{2}(0,x)-u_{1}(0,x)\leq u_{0}(0,x),\  \  \  \forall x\in
\mathbb{R}^{N},
\]
then $u_{2}(t,x)-u_{1}(t,x)\leq u_{0}(t,x)$ for all $(t,x)\in
\lbrack0,\infty )\times \mathbb{R}^{N}$.
\end{corollary}

\begin{proof}
It suffices to observe that $-u_{i}\in$USC$((0,T)\times
\mathbb{R}^{N})$ are viscosity subsolutions of
$\partial_{t}u-G_{\ast}^{i}(D^{2}u)=0$, with
$G_{\ast}^{i}(X):=-G_{i}(-X)$. We also have, for each
$X_{1}+X_{2}+X_{3}\leq 0$,
\begin{align*}
G_{\ast}^{0}(X_{1})+G_{\ast}^{1}(X_{2})+G^{1}(X_{3})  &  =G_{1}(X_{3}%
)-G_{1}(-X_{2})-G_{0}(-X_{1})\\
&  \leq G_{0}(X_{3}+X_{2})-G_{0}(-X_{1})\leq0\text{.}%
\end{align*}
We thus can apply the above domination theorem to get $u_{2}(t,x)-u_{1}%
(t,x)\leq u_{0}(t,x)$.
\end{proof}


\begin{thebibliography}{99}                                                                                               %


\bibitem {ADEH1}Artzner, Ph., F. Delbaen, J.-M. Eber, and D. Heath (1997),
Thinking coherently.

\bibitem {ADEH2}$\sim \ $ (1999), Coherent Measures of Risk, \emph{Mathematical
Finance} 9, 203-228.

\bibitem {Avellaneda}Avellaneda M., Levy, A. and Paras A. (1995). Pricing and
hedging derivative securities in markets with uncertain
volatilities. Appl. Math. Finance \textbf{2,} 73--88.


\bibitem {BCHMP1}Briand, Ph., Coquet, F., Hu, Y., M\'{e}min J. and Peng, S.
(2000) A converse comparison theorem for BSDEs and related
properties of g-expectations, \emph{Electron. Comm. Probab,}
\textbf{5}.

\bibitem {El-Bar}Barrieu, P. and El Karoui, N. (2004) Pricing, hedging and
optimally designing derivatives via minimization of risk measures,
Preprint, to appear in \emph{Contemporary Mathematics.}

\bibitem{Caff1997} X. Cabre and L.A. Caffarelli, Fully nonlinear elliptic
partial di erential equations, American Math. Society (1997).

\bibitem {Chen98}Chen, Z. (1998) A property of backward stochastic
differential equations, \emph{C.R. Acad. Sci. Paris}
\textbf{S\'{e}r.I 326(}4), 483--488.

\bibitem {CE}Chen, Z. and Epstein, L. (2002), Ambiguity, Risk and Asset
Returns in Continuous Time, \emph{Econometrica,} \textbf{70}(4),
1403--1443.

\bibitem {CKJ}Chen, Z., Kulperger, R. and Jiang L. (2003) Jensen's inequality
for g-expectation: part 1, \emph{C. R. Acad. Sci. Paris},
\textbf{S\'{e}r.I 337}, 725--730.

\bibitem {CP}Chen, Z. and Peng, S. (1998) A Nonlinear Doob-Meyer type
Decomposition and its Application. \emph{SUT Journal of Mathematics}
(Japan), \textbf{34}(2), 197--208.

\bibitem {CP1}Chen, Z. and Peng, S. (2000), A general downcrossing inequality
for g-martingales, \emph{Statist. Probab. Lett.} \textbf{46}(2),
169--175.

\bibitem {Touzi}Cheridito, P., Soner, H.M., Touzi, N. and Victoir, N., Second
order backward stochastic differential equations and fully
non-linear parabolic PDEs, Preprint (pdf-file available in
arXiv:math.PR/0509295 v1 14 Sep 2005).


\bibitem {Choquet}Choquet, G. (1953) \emph{Theory of Capacities}, Annales de
Institut Fourier, \textbf{5, }131--295.

\bibitem {Chu-Will}Chung, K.L. and Williams, R. \emph{Introduction to
Stochastic Integration}, 2nd Edition, Birkh\"{a}user, 1990.

\bibitem {CHMP}Coquet, F., Hu, Y., M\'{e}min, J. and Peng, S. (2001) A general
converse comparison theorem for Backward stochastic differential
equations, \emph{C.R.Acad. Sci. Paris, }\textbf{t.333,} Serie I,
577--581.

\bibitem {CHMP3}Coquet, F., Hu, Y., Memin J. and Peng, S. (2002),
Filtration--consistent nonlinear expectations and related
g--expectations, \emph{Probab. Theory Relat. Fields,} \textbf{123},
1--27.

\bibitem {CIL}Crandall, M., Ishii, H., and Lions, P.-L. (1992) User'S Guide To
Viscosity Solutions Of Second Order Partial Differential Equations,
\emph{Bulletin Of The American Mathematical Society,}
\textbf{27}(1), 1-67.

\bibitem {Daniell}Daniell, P.J. (1918) A general form of integral.
\emph{Annals of Mathematics,} \textbf{19}, 279--294.

\bibitem {DM}Dellacherie, C. and Meyer, P.A., \emph{Probabilities and
Potentiel A and B,} North--Holland, 1978 and 1982.

\bibitem {Delbaen}Delbaen, F. (2002), Coherent Risk Measures (Lectures given
at the Cattedra Galileiana at the Scuola Normale di Pisa, March
2000), Published by the Scuola Normale di Pisa.

\bibitem {DPR}Delbaen, F., Rosazza Gianin, E. and Peng S. (2005) m-Stable
sets, risk measures and g-expectations, Preprint.

\bibitem {Denis-M}Denis, L. and Martinin, C. (2006) A theoretical framework
for the pricing of contingent claims in the presence of model
uncertainty, The Ann. of Appl. Probability \textbf{16}(2), 827--852.

\bibitem {Deneberg}Deneberg, D. Non-Additive Measure and Integral, Kluwer, 1994.

\bibitem {El-Bar2005}Barrieu, P. and El Karoui, N. (2005) Pricing, Hedging and
Optimally Designing Derivatives via Minimization of Risk Measures,
Preprint.

\bibitem {EQ}El Karoui, N., Quenez, M.C. (1995) Dynamic Programming and
Pricing of Contingent Claims in Incomplete Market. \emph{SIAM J.of
Control and Optimization}, \textbf{33}(1).

\bibitem {EPQ}El Karoui, N., Peng, S., Quenez, M.C. (1997) Backward stochastic
differential equation in finance, \emph{Mathematical Finance}
\textbf{7}(1): 1--71.

\bibitem {Feyel-DLP}Feyel, D. and de La Pradelle, A. (1989). Espaces de
Sobolev gaussiens. Ann. Inst. Fourier \textbf{39} 875--908.

\bibitem {FS}Fleming, W.H., Soner, H.M. (1992) \emph{Controlled Markov
Processes and Viscosity Solutions.} Springer--Verleg, New York.

\bibitem {F-Sch}F\"{o}llmer \& Schied (2004) Statistic Finance, Walter de Gruyter


\bibitem {F-RG1}Frittelli, M. and Rossaza Gianin, E. (2004) Putting order in
risk measures, Journal of Banking and Finance, \textbf{26}(7)
1473--1486.

\bibitem {F-RG2}Frittelli, M. and Rossaza Gianin, E. (2004) Dynamic convex
risk measures,  Risk Measures for the 21st Century (see
\cite{Szego}) 227--247.

\bibitem {HWY}He, S.W., Wang, J.G. and Yan J.--A. (1992) \emph{Semimartingale
Theory and Stochastic Calculus}, CRC Press, Beijing.

\bibitem {Huber}Huber,P. J., (1981) \emph{Robust Statistics}, John Wiley \& Sons.

\bibitem {IW}Ikeda, N. and Watanabe, S., \emph{Stochastic Differential
Equations and Diffusion Processes}\textit{, }North--Holland,
Amsterdam, 1981.

\bibitem {Ito}It\^{o}, K. \emph{Differential equations determining a Markoff
process,} Journ. Pan--Japan Math. Coll. No. 1077, 1942, In
\emph{Kiyosi It\^{o}: Selected Papers}, Springer, 1987.

\bibitem {Ito-McKean}It\^{o}, K. and McKean, M., \emph{Diffusion Processes and
Their Sample Paths}, Springer--Verlag, 1965.

\bibitem {JiaPeng} Jia, G. and Peng, S. (2007) A new look at Jensen's inequality for g-expectations: g-convexity
and their applications, Preprint.

\bibitem {Jiang}Jiang, L. (2004) Some results on the uniqueness of generators
of backward stochastic differential equations, \emph{C. R. Acad.
Sci. Paris,} \textbf{Ser. I 338} 575--580.

\bibitem {JC}Jiang L. and Chen, Z. (2004) A result on the probability measures
dominated by g-expectation, \emph{Acta Mathematicae Applicatae
Sinica}, English Series \textbf{20}(3) 507--512.

\bibitem {JC1}Jiang L. and Chen Z. (2004) On Jensen's inequality for
g-expectation, \emph{Chin. Ann. Math.} \textbf{25B}(3),401--412.

\bibitem {KSh}Karatzas, I. and Shreve, S. E., \emph{Brownian Motion and
Stochastic Calculus}\textit{,} Springer--Verlag, New York, 1988.

\bibitem{Krylov} Krylov, N.V. (1980) \emph{Controlled Diffusion Processes.}
Springer--Verlag, New York.

\bibitem{Krylov1} Krylov, N.V.

\bibitem {Lyons}Lyons, T. (1995). Uncertain volatility and the risk free
synthesis of derivatives. \emph{Applied Mathematical Finance} 2,
117--133.

\bibitem {Nisio1}Nisio, M. (1976) On a nonlinear semigroup attached to optimal
stochastic control. \emph{Publ. RIMS, Kyoto Univ}., 13: 513--537.

\bibitem {Nisio2}Nisio, M. (1976) On stochastic optimal controls and envelope
of Markovian semi--groups. \emph{Proc. of int. Symp. Kyoto,}
297--325.

\bibitem {Oksendal}\O ksendal B. (1998) \emph{Stochastic Differential
Equations}, \ Fifth Edition, Springer.

\bibitem {Pardoux-Peng}Pardoux, E., Peng, S. (1990) Adapted solution of a
backward stochastic differential equation, \emph{Systems and Control
Letters,} \textbf{14}(1): 55--61.

\bibitem{Peng1992}Peng, S. (1992) A generalized dynamic programming principle
and Hamilton-Jacobi-Bellman equation. \emph{Stochastics and
Stochastic Reports}, \textbf{38}(2): 119--134.

\bibitem{Peng1997}Peng, S. (1997) Backward SDE and related g--expectation, in
\emph{Backward Stochastic Differential Equations}, Pitman Research
Notes in Math. Series, No.364, El Karoui Mazliak edit. 141--159.

\bibitem {Peng1997b}Peng, S. (1997) BSDE and Stochastic Optimizations,
\emph{Topics in Stochastic Analysis,} Yan, J., Peng, S., Fang, S.,
Wu, L.M. Ch.2, (Chinese vers.), Science Press, Beijing.

\bibitem{Peng1999}Peng, S. (1999) Monotonic limit theorem of BSDE and
nonlinear decomposition theorem of Doob-Meyer's type, \emph{Prob.
Theory Rel. Fields} \textbf{113}(4) 473-499.

\bibitem {Peng2003}Peng, S. (2004) Nonlinear expectation, nonlinear
evaluations and risk measurs, in K. Back T. R. Bielecki, C. Hipp, S.
Peng, W. Schachermayer, \emph{Stochastic Methods in Finance
Lectures}, 143--217, LNM 1856, Springer-Verlag.

\bibitem{Peng2004}Peng, S. (2004) Filtration Consistent Nonlinear
Expectations and Evaluations of Contingent Claims, \emph{Acta
Mathematicae Applicatae Sinica,} English Series \textbf{20}(2),
1--24.

\bibitem {Peng2005}Peng, S. (2005) Nonlinear expectations and nonlinear Markov
chains, \emph{Chin. Ann. Math.} \textbf{26B}(2) ,159--184.

\bibitem {Peng2005a}Peng, S. (2004) Dynamical evaluations, \emph{C. R. Acad.
Sci. Paris,} Ser.I \textbf{339} 585--589.

\bibitem {Peng2005b}Peng, S. (2005), Dynamically consistent nonlinear
evaluations and expectations, preprint (pdf-file available in
arXiv:math.PR/0501415 v1 24 Jan 2005).

\bibitem {Peng2006a}Peng, S. (2006) $G$--Expectation, $G$--Brownian Motion and
Related Stochastic Calculus of It\^{o}'s type, preprint (pdf-file
available in arXiv:math.PR/0601035v1 3Jan 2006), to appear in
\emph{Proceedings of the 2005 Abel Symposium}.

\bibitem {PX2003}Peng, S. and Xu, M. (2003) Numerical calculations to solve
BSDE, preprint.

\bibitem {PX2005}Peng, S. and Xu, M. (2005) $g_{\Gamma}$--expectations and the
Related Nonlinear Doob-Meyer Decomposition Theorem, preprint.

\bibitem{Peng2007} Peng, S. (2007) Law of large numbers and
central limit theorem under nonlinear expectations, in
arXiv:math.PR/0702358v1 13 Feb 2007

\bibitem {Protter}Protter, Ph. \emph{Stochastic Integration and Differential
Equations,} Springer--Verlag, 1990.

\bibitem {Revuz-Yor}Revuz, D., and Yor, M. \emph{Continuous Martingales and
Brownian Motion,} Springer--Verlag, 1991.

\bibitem {Roazza2003}Rosazza, E. G., (2003) Some examples of risk measures via
g--expectations, preprint, to appear in \emph{Insurance: Mathematics
and Economics}.

\bibitem{Shiryaev} Shiraev, A.N.  \emph{Probability,} 2nd Edition, Springer-Verlag, 1984.

\bibitem{Szego}Szeg\"{o}, G. (edits) \emph{Risk Measures for the 21
Centrury,} Wiley-Finance, 2004.

\bibitem{WangL} L. Wang, (1992) On the regularity of fully nonlinear parabolic
equations: II, Comm. Pure Appl. Math. 45, 141-178.

\bibitem {Yan}Yan, J.-A. (1998) \emph{Lecture Note on Measure Theory}, Science
Press, Beijing (Chinese version).

\bibitem {Yong-Zhou}Yong, J., Zhou, X. (1999) \emph{Stochastic Controls:
Hamiltonian Systems and HJB Equations.} Springer--Verlag.

\bibitem {Yosida}Yosida, K. (1980) \emph{Functional Analysis,} Sixth-Edition, Springer.
\end{thebibliography}
\end{document}